\documentclass[a4paper,12pt]{amsart}
\pdfoutput=1
\usepackage[T1]{fontenc} 
\usepackage{mathtools}
\usepackage[all]{xy}
\usepackage[english]{babel}
\usepackage{frcursive} 
\usepackage[protrusion=true,expansion=true]{microtype}
\usepackage{eucal}
\usepackage{graphicx}
\usepackage{epsfig}
\usepackage{mathrsfs}
\usepackage{amssymb}
\usepackage{amsxtra}
\usepackage{enumerate}

\newtheorem{theorem}{Theorem}[section]

\newtheorem{proposition}[theorem]{Proposition}
\newtheorem{notation}[theorem]{Notation}
\newtheorem{lemma}[theorem]{Lemma}

\newtheorem{corollary}[theorem]{Corollary}

\newtheorem{definition}[theorem]{Definition}

\theoremstyle{remark}

\newtheorem{example}[theorem]{\bf Example}

\def \1{\mathbb {1}}
\def \RM{\mathbb {R}}
\def \NM{\mathbb{N}}
\def \ZM{\mathbb{Z}}
\def \CM{\mathbb{C}}

\def \QM{\mathbb{Q}}
\def \PM {\mathbb{P}}

 \def \Vol {{\rm Vol}}
 \def \Hom {{\rm Hom}}
 \def \grad {{\rm grad}}

\def \Im {{\rm Im\,}}

\def \ord {{\rm ord\,}}

\def \Der {{\rm Der\,}}

\def \Ham {{\rm Ham\,}}

\def \p {{\rm exp\,}}
\def \Id {{\rm Id\,}}

\def \d{\partial}
\def\dt{\delta} 
\def\a{\alpha}
\def\b{\beta}
\def\e{\varepsilon}  
\def\g{\gamma}

\def\l{\lambda}

\def\p{\varphi}

\def\G{\Gamma}   
\def\D{\Delta}
\def \s{\sigma}

\def \to{\longrightarrow} 

\def\del{\nabla}
\def \< {{\langle }}
\def \> {{\rangle }}
\def \( {\left( }
\def \) {\right) }

\newcommand{\Bt}{{\mathcal B}}
\newcommand{\Ct}{{\mathcal C}}
\newcommand{\Dt}{{\mathcal D}}

\newcommand{\Mt}{{\mathcal M}}

\newcommand{\Ot}{{\mathcal O}}

\newcommand{\Rt}{{\mathcal R}}

\newcommand{\Ut}{{\mathcal U}}
\newcommand{\Vt}{{\mathcal V}}
\newcommand{\Wt}{{\mathcal W}}

\renewcommand{\mod}{{\rm  mod\,}}

\newcommand{\bu}{\bullet}
\newcommand{\bs}{\blacksquare}
\newcommand{\cro}{\times}
\newcommand{\lra}{\longrightarrow}

\title[The Herman invariant tori conjecture]{The Herman invariant tori conjecture}
\author{  Mauricio  Garay and Duco van Straten}

\begin{document}
\begin{abstract}   We study a normal form at a critical point of an
  analytic Hamiltonian somewhat reminiscent of the Pad\'e approximants in
  classical analysis. Under an arithmetic condition on the frequency we
  prove that the  normal form converges over a Cantor set. Using this result,
  we deduce that the Hamiltonian is integrable over a Cantor set. Using the notion of
  curvedness, we show that in the real elliptic case this implies the existence of a
  positive measure set of invariant tori near the critical point, of density equal to one.
  This solves a conjecture formulated by M. Herman in the 90's.
\end{abstract}
\maketitle
\begin{flushright}{\em To the memory of J.-C. Yoccoz.}\end{flushright}
 
\section{Statement of the theorem}

Investigations into normal forms of Hamiltonian systems can be traced back to the 
earliest beginnings of celestial mechanics in the 
works of Euler, Laplace, Delaunay, and others. However, the understanding of simple examples is still very far from satisfactory. A fundamental class of examples starts with a system of non-interacting oscillators
$$H_0=\frac12 \sum_{i=1}^d\omega_i (p_i^2+q_i^2) .$$
In this case, the solutions to the {\em Hamilton equations}
$$\left\{ \begin{array}{rcr}
 \dot{p}_i&=&-\d_{q_i} H\\
 \dot{q}_i&=&\d_{p_i} H
\end{array}
\right.
$$
are easily integrated: the trajectory
$$p_i=-\e_i \sin(\omega_it+\a_i),\ q_i=\e_i \cos(\omega_it+\a_i) $$
remains on the torus defined by the equations
$$ p_i^2+q_i^2=\e_i^2, i=1,2,\ldots,d.$$
The closure of a trajectory defines a sub-torus of dimension 
$$k:=\dim_\QM \QM\omega_1+\QM\omega_2+\ldots+\QM\omega_d .$$
These motions on tori are called {\em quasi-periodic}; we consider the case of periodicity (i.e. $k=1$)
as a particular case of quasi-periodicity.
When the dimension of the torus is not maximal, i.e. when $k <d$, the frequency vector
$\omega=(\omega_1,\dots,\omega_d) $ is called {\em resonant} and otherwise the torus is called
a {\em KAM torus}.  The subject of KAM theory is the study of the persistence of these tori
under various types of perturbations~\cite{Arnold_KAM,Kolmogorov_KAM,Moser_KAM}. Of course,
under perturbation, the frequency of motion may vary from torus to torus and there is a specific
{\em frequency map} that assigns the frequency of motion to a given KAM torus.

The following strong local version of the KAM theorem was conjectured by Herman~\cite{Herman_ICM}:
\begin{theorem}
 Consider a real analytic Hamiltonian function
 $$H_0=\frac{1}{2}\sum_{i=1}^d\omega_i (p_i^2+q_i^2) +\dots$$
 where the dots denote higher order terms in the Taylor expansion. If the frequency vector $\omega=(\omega_1,\dots,\omega_d)$ satisfies a Bruno condition, then the set of KAM tori has Lebesgue density equal to one at the origin.
 
\end{theorem}
In particular there are infinitely many KAM tori in any neighbourhood of the origin and the union of these form a positive measure set in $\RM^{2d}$.
The condition introduced by Bruno to which we refer, is an arithmetic condition, much weaker than the original {\em Diophantine conditions} introduced by Siegel and Kolmogorov and is now standard in the study of dynamical systems~\cite{Bruno,Kolmogorov_KAM,Siegel_linearisation,Siegel_vecteurs}. It is defined as follows: we start from a decreasing sequence $a=(a_n)$ of positive real numbers and define the set
$$\RM^d(a):=\{ \a \in \RM^d:\forall n \in \NM,\ \forall J \in \ZM^d\setminus \{0\},\ \| J \| \leq 2^n,\   (\a,J) \geq a_n \} $$
This set contains non-resonant vectors only and the sequence $(a_n)$ gives a practical way to control the distance to the resonance hyperplanes
\[ (\a, J)=0 .\]
The sequence $a$ is called a {\em Bruno sequence} if the infinite product
$\prod_{k \geq 0} a_k^{1/2^k} $ converges (to a positive number) or equivalently if
$$\sum_{k \geq 0} \frac{|\log a_k|}{2^k}<+\infty. $$
Note that the set $\RM^d(a)$ is contained inside the complement of the union of the resonance hyperplanes. A vector is said to satisfies a {\em Bruno condition}
if it lies in some $\RM^d(a)$, where $a$ is a Bruno sequence. 

The  {\em Lebesgue density of a set $X \subset \RM^d$ at a point $p$}  is defined  (when the limit exists) by
$$\dt(X,p):=\lim_{\e \to 0} \frac{m(X \cap B(p,\e))}{m(B(p,\e))} .$$
Here $B(p,\e)$ is the ball of radius $\e$ centred at $p$ and $m$ stands for the Lebesgue measure. A classical
theorem of Lebesgue asserts that for almost all points $p$ the limit exists and is equal to $0$ or $1$.

In the context of KAM theory, one usually assumes some non-degeneracy conditions on the formal frequency mapping,
for instance {\em R\"u{\ss}mann non-degeneracy conditions}~\cite{Russmann_KAM}.  These conditions involve
computations on the formal frequencies and therefore require computing higher order terms in the Birkhoff normal
form. In practice they are difficult to check.  Assuming such R\"u{\ss}mann non-degeneracy conditions, a version
of the above theorem was proven by Stolovitch in~\cite{Stolovitch_KAM}. It was believed by experts that R\"ussmann
conditions were necessary (see e.g.~\cite{Sevryuk_KAM}) and probably many of them considered the conjecture should
be wrong or at least not as the general rule for KAM theorem. The apparent contradiction between the believed necessity of R\"ussmann conditions and the Herman conjecture was pointed out back in 2014 by the first author~\cite{CRAS_KAM}.

The main reason to doubt the necessity of the R\"u{\ss}mann non-degeneracy is the following. In the most degenerate case $H=H_0$, the frequency map is constant and the space is foliated by invariant tori, most of which are KAM tori. Moreover, in the case that the formal  frequency mapping is constant, then a famous theorem of R\"u{\ss}mann
states that the system is Liouville integrable and that the normalising series converges (\cite{Russmann,Stolovitch,Vey})! In fact, this remarkable theorem led us to think that the conjecture, without any further conditions, must be true. 

Since the original arguments by the first author were posted on the ArXiv in 2012~(\cite{Herman}), there
have been many other attempts to settle the conjecture by more direct methods~(see e.g.~\cite{eliasson2015around}), but as far as we are aware of, no proofs were found. Ten years have passed since and we thought it might be useful
to simplify and clarify the original arguments in order to make it accessible to a more general audience. The
presentation in this paper does not depend on earlier unpublished results. It is also self-contained except for
the arithmetic density theorem of  \cite{arithmetic}.
\section{Strategy of the proof}
In the proof we are going to present, there are three main strands of ideas:\\
\vskip0.3cm
\begin{itemize}
\item[ALGEBRA:] We introduce a new type of normal form which is a variation on the Birkhoff normal form that we
  call the {\em Hamiltonian normal form}. This normal form provides an approximation by rational functions, while
  the Birkhoff normal form is an approximation by polynomials. The Birkhoff normal form is the formal expansion at a
  point and consequently does not 'see' the nearby resonances. In contrast to this, the Hamiltonian normal form is
  a more global object with poles along the resonance hyperplane. The difference between the two is reminiscent
  to that between a Taylor series and its Pad\'e approximants.
 \item[ANALYSIS:] The convergence of the iteration process does not take place in a single Banach space, but uses various sequences of such spaces. We consider holomorphic functions on sets 
   $U_0 \supset U_1 \supset \dots $ which have either continuous, bounded or $L^2$ extension to the boundary. This is the classical framework introduced already by Kolmogorov and is well-known and common to all current approaches to
   KAM theory. But there are at least two differences to be noticed: first, we consider several of these different
   completions of the  space of holomorphic functions at the same time in a coherent framework, and second we {fix} the Cantor set on which the iteration takes place. So we shall prove that the Hamiltonian normal form leads to a convergent iteration over a {\em fixed Cantor set}, whereas in the usual  proofs this set is constructed step by step and changes under the transformations of the iteration.
   At first glance, this part might seem irrelevant to an expert that can provide any estimate whenever needed, but
   we hope that the details of our more systematic approach nevertheless attracts some attention, as it leads to a
   significant simplification of the set-up.
 \item[GEOMETRY:] The Hamiltonian normal form replaces the formal frequency mapping by a sequence of algebraic manifolds that we call the {\em frequency manifolds}. Locally around the origin these manifolds are tangent to the graphs of the Taylor polynomials of the Birkhoff normal forms; these frequency manifolds however are globally defined.  Given a decreasing sequence $a$ of positive real numbers and a submanifold $X \subset \CM^d \times \CM^d$, we have to consider the subset $X(a):=X \cap (\CM^d \times \CM^d(a))$.   We develop the theory of non-degeneracy for such submanifolds. This part can be considered as a KAM-tailored complement to the classical works in Diophantine geometry pursued by Kleinbock, Margulis, Arnold, Pyartli and others~\cite{Kleinbock_Margulis,Margulis,Pyartli}.
\end{itemize}
 
Our arguments are of general nature and we are confident that there will be no difficulty in adapting this result
to different but similar situations.

The structure of the paper is as follows:

In \S\ref{BirkhoffNormalForm} we fix the algebraic framework and notations of the paper. We
give a quick review of the classical Birkhoff normal form. Then we add  $d$ additional {\em Moser variables} $\tau_1,\tau_2,\ldots, \tau_d$ and recall the basic relation between the Moser variant of the Birkhoff normal form and the classical one.\\

In \S\ref{HamiltonianNormalForm} we describe the main idea of the iteration
that leads to our Hamiltonian normal form in the spirit of parametric KAM theory~\cite{Broer_Huitema_Sevryuk_families,Broer_Huitema_Sevryuk_book, Broer_Huitema_Takens}. To achieve this, one introduces $d$ further independent frequency variables  $\omega_1,\omega_2,\ldots,\omega_d$ and the Hamiltonian normal form  is a power series in the $2d$ variables  $\tau, \omega$. There are $d$ {\em frequency relations} between these
$\omega$ and $\tau$ variables. We show that by solving these frequency relations in terms of the $\tau$-variables,
the Hamiltonian normal form reduces to the Birkhoff normal form.  We construct an explicit solution of the homological equation and we give a precise formal definition of the iteration. We formulate some 
statements on the precise orders of the terms appearing in the iteration.
This section concludes the {\em algebra} part of the paper.

In \S\ref{SmallDenominators} we start the functional analytic part of the paper.  We describe in some detail our functional analytic
underpinning for dealing with the Banach spaces arising from the small
denominators. 

In \S\ref{KolmogorovSpaces} we sketch out our framework of {\em Kolmogorov spaces}
that is a convenient tool to handle the estimates that arise from the analysis
of convergence of our iteration process. 

In \S\ref{S::fixed_point},  we state and prove a fixed point theorem in Kolmogorov spaces. The theorem implies,
under a Bruno condition on the frequency, the convergence of the Hamiltonian Normal Form iteration over a Cantor-like set.  All estimates we
need are all simple applications of the general functional analytic constructions. This ends the functional analytic part  of the paper.

In \S\ref{ArithmeticDensity} we investigate the non degeneracy conditions of Diophantine geometry introduced by Kleinbock and Margulis
(\cite{Kleinbock,Kleinbock_Margulis}) from an abstract standpoint. Then we recall the arithmetic density theorem of \cite{arithmetic}. This is the only result to which we refer in the paper without proving it. The paper is however elementary and the measure estimate can in fact easily be improved (we intend to publish a result with a sharper estimate elsewhere).

In \S\ref{S::curvedness}, we show that the curvedness condition of Kleinbock and Margulis is automatically satisfied for the frequency map. This is done via studying the implicit equations for the more global frequency manifolds, which are defined by frequency relations between $\omega$ and $\tau$-variables. We study curvedness properties of the frequency manifolds and formulate them in terms of R\"u{ss}mann spaces, which are under some control during the iteration. This part confirms the pertinence of our framework.

In \S \ref{SolutionHermanConjecture} we restrict back to the real domain and
show in some detail how our result on curvedness implies the existence of a
positive measure set of invariant tori near an elliptic fixed point.
This completes our proof of the Herman conjecture.

\noindent {\em Acknowledgement.}{ It appeared to us that over the decades KAM theory has become a beautiful but
  very specialised domain, the entrance to which has been restricted to a small community of experts. Being not
  part of this community, our ignorance forced us to develop an independent version of KAM theory~\cite{groupes,Oberwolfach_2009,Abstract_KAM,Lagrange_KAM}. Accidentally, the Herman conjecture turned out to be a necessary corollary of these considerations~\cite{Herman,Oberwolfach_2012}. It seems fitting to us to quote Grothendieck, who once expressed our
  impressions so eloquently:\\

"{\em La mer s'avance insensiblement et sans bruit, rien ne semble se casser rien ne bouge l'eau est si loin on l'entend \`a peine.. Pourtant elle finit par entourer la substance r\'etive, celle-ci peu \`a peu devient une presqu'\^ile, puis une \^ile, puis un \^ilot, qui finit par \^etre submerg\'e \`a son tour, comme s'il s'\'etait finalement dissous dans l'oc\'ean s'\'etendant \`a perte de vue}\footnote{The sea grows imperceptibly, silent, nothing seems to break, nothing moves, the water is so far away you can hardly hear it... However, it ends up surrounding the restive substance, this one little by little becomes a peninsula, then an island, then an islet, which ends up being submerged in its turn, as if it had finally dissolved in the ocean stretching as far as the eye can see ...}~..."\\
\vskip0.1cm

The introduction of Banach space functors in 2018 gave us a deeper understanding on KAM theory and helped us to rewrite the original work from 2012 on the Herman conjecture. In this paper we only present the bare minimum of this theory and we refer the interested reader to~\cite{KAM_theory_I,KAM_theory_II,KAM_theory_III,Functors} for further details.

Our feeling is that there remains further foundational work to be done in KAM theory, which is in a state somewhat
comparable to that of Italian algebraic geometry, before the introduction of sheaves. In our opinion, the fact
that the Herman conjecture remained unproven so long illustrates that KAM theory is just at its infancy. 
  
When this work started,  J.-C. Yoccoz was among the very few enthusiastic dynamicists eager to bridge the frontiers between mathematical communities. After two months of numerous exchanges, Yoccoz was forced to stop due to health problems. In the meantime, he had made several influential and motivating remarks, and for this reason we dedicate this research to his memory.
 
\section{\bf The Birkhoff normal form}
\label{BirkhoffNormalForm}
\subsection{The algebraic context}
The idea of normal forms is usually formulated in terms of a set $X$ of
objects on which a group $G$ acts. The archetypal case, discussed by
Arnold in~\cite{Arnold_matrices}, is that of the
set $X=M(n,\CM)$ of $n \times n$ matrices on which the group $G=GL(n,\CM)$
acts by conjugation: $P \cdot M=PMP^{-1}$. However, in most interesting cases the set $X$ is
an infinite dimensional vector space: usually a space of functions or function germs and
the group $G$ a subgroup of diffeomorphisms or changes of variables. \\
\begin{figure}[htb!]
\includegraphics[width=0.4\linewidth]{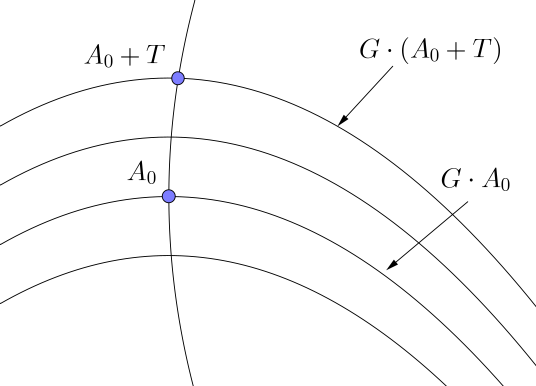}
\end{figure}
 
Depending on the problem at hand, these 
functions and diffeomorphisms may be $C^\infty$, analytic or even given by
formal power series. In the Hamiltonian context one is dealing with Hamiltonian 
functions and symplectomorphism. One then studies the orbits in $X$ and
tries to find an appropriate transversal slice $T$ to the orbits. 
By transforming an element inside the slice, one says it is brought to 
{\em normal form.}  Such transformations are usually obtained from an iteration
scheme.

Often the functions or elements of $X$ form a ring, and the coordinate transformations
from $G$ induce automorphisms of this ring. These automorphisms are often
obtained by exponentiating derivations, which is the algebraic avatar of
obtaining diffeomorphisms from flows of vector fields. In the Hamiltonian
context the ring carries a Poisson-bracket and the automorphisms we use
are required to be symplectomorphisms, i.e. preserve this Poisson-bracket.
So normal form problems for the symplectic group find their natural expression
in terms of Poisson-algebras.

\subsection{The symplectic Poisson algebra}
We will be concerned with the structure of an analytic Hamiltonian system 
with $d$ degrees of freedom near a critical point of the form
\[H=\sum_{i=1}^d \alpha_i p_i q_i+O(3).\]
We assume that the {\em frequency vector}:
\[\alpha:=(\alpha_1,\alpha_2,\dots,\alpha_d) \in \CM^d\]
is {\em non-resonant}, i.e., its components $\alpha_i$ are  $\QM$-linearly independent. We are interested in the question which terms appearing in $H$ 
may be transformed away using symplectic coordinate transformations.\\

We can consider the Hamiltonian $H$ as an element of the formal power 
series ring
$$P:=\CM[[q,p]]:=\CM[[q_1,\dots,q_d,p_1,\dots,p_d]].$$ 
The Poisson bracket of $f,g \in P$, defined by
$$\{ f,g \}=\sum_{i=1}^d \d_{q_i}f\d_{p_i}g-  \d_{p_i}f\d_{q_i}g,$$
makes $P$ into a Poisson-algebra.
An element $h \in P$ is a power series that can be written as
\[h:=\sum_{a,b} C_{a,b} p^a q^b,\;\;\; C_{a,b}\in \CM, \]
where we use the usual multi-index notation, so that
\[p^aq^b:=p_1^{a_1}p_2^{a_2}\ldots p_d^{a_d} q_1^{b_1}q_2^{b_2}\ldots q_d^{b_d},\] 
and so on. We will assign weight $1$ to each of the variables, so 
that the monomial $p^aq^b$ has weight $|a|+|b|$.
We write $h=O(k)$ 
if $h$ only contains monomials of degree $\ge k$, and say that $h$ {\em has
order $k$}.  If $h$ is analytic, it is represented by a convergent series, 
and our usage of the $O$ corresponds to its usual meaning. Algebraically,
the filtration by order is the filtration of $P$ by the powers of the
maximal ideal $\mathcal{M}$:
\[ P \supset \mathcal{M} \supset  \mathcal{M}^2 \supset \mathcal{M}^3 \supset \ldots supset \{0\},\]
where
\[ \mathcal{M}^k:=\{h \in P\;|\; h = O(k) \}.\]

In a similar way, we can truncate a vector field by truncating its coefficients, but taking the shift of 
grading by $1$ into account (due to the fact the derivative decreases the degree by one). Thus a vector field of order $d$ maps the space of power series of order $k$ to power series of order $d+k$.
 
\begin{notation} If $h$ belongs $P$ or a filtered $P$-module, we denote by
$$\left[ h\right]_i^j $$
for the sum of terms of $h$ of weight (=degree) $\geq i$ and $<j$, 
so that $\left[h\right]_i^{i+1}$
represents the part of $h$ of pure weight $i$. When $j=+\infty$ we omit the letter $j$, when $i=0$ we omit the letter $i$. 
\end{notation}

\begin{definition}
A derivation $v \in Der(P)$ that preserves the Poisson-bracket: 
\[ v(\{f,g\})=\{v(f),g\}+\{f,v(g)\}\]
is called a {\em Poisson-derivation} and we denote by $\Theta(P)$ the vector 
space of all Poisson-derivations or Poisson vector fields.
\end{definition}
The map 
\[  P \longrightarrow \Theta(P),\;\;\; h \mapsto \{-,h\}\]
associates to $h$ the corresponding Poisson-derivation, usually called the
{\em Hamiltonian vector field} of $h$. If $h=O(k)$ and $f=O(l)$, then clearly 
$\{f,h\}=O(k+l-2)$, so the vector field $v:=\{-,h\}$ is said to be 
{\em of order $k-2$}, although the coefficients of the vector field $v$ are $O(k-1)$. 
The following is immediate:\\
\begin{lemma} If $h=O(3)$, then one can {\em exponentiate} $v$ and obtain a Poisson 
automorphism of the ring $P$:
\[e^v=Id+\{-,h\}+\frac{1}{2!}\{\{-,h\},h\}+\ldots \in Aut(P).\]
\end{lemma}

If $v$ happens to be analytic then it defines a vector field in a neighbourhood of 
the origin and our derivation $v$ is simply the Lie derivative along this vector field. 
The formal power series $e^v$ is in that case an analytic automorphism and thus defines 
an associated analytic change of variables, the time $=1$ flow of the vector field. 

\subsection{The Birkhoff normal form}
If we let
\[h_0:=\sum_{i=1}^d \alpha_i p_iq_i ,\]
then
\[ \{h_0,p^aq^b\}=(\alpha,a-b) p^aq^b, \]
where $(-,-)$ denotes the standard euclidean scalar product. 
So if $\alpha$ is non-resonant, then each monomial $p^aq^b$ with $a \neq b$
appearing in $H=h_0+O(3)$ can be removed by an application of the
derivation
\[ v_{a,b}=\{-,\frac{1}{(\alpha,a-b)}p^aq^b\} .\]

As the application of $e^{-v_{a,b}}$ to $H$ will remove the term $p^aq^b$ from $H$,
we see that one can construct a sequence of automorphisms 
\[ \varphi_0:=e^{-v_0},\;\;\;\varphi_1:=e^{-v_1},\;\;\;\varphi_2:=e^{-v_2},\ldots \in Aut(P),\]
that remove successively all monomials $p^aq^b$, $a \neq b$ from the Hamiltonian $H$.

To write this iteration more explicitly, let us introduce some notation:

We consider the $\CM$-linear map
$$j:P \to \Theta(P),\ p^aq^b \mapsto \left\{ \begin{matrix} {\displaystyle \{-,\frac{1}{(\alpha,a-b)}p^aq^b\}}  &\text{ if } a \neq b \\
\ \\
0 &\text{ otherwise }\end{matrix} \right. $$
Then we define the iteration by putting $H_0=H$ and
\begin{align*}
v_k&=j([H_k]_{k+3}^{k+4})\\
H_{k+1}&=e^{-v_k}H_k
\end{align*}
so that the automorphism
\[\Phi_k :=\varphi_{k-1}\ldots\varphi_1 \varphi_0,\ \p_i=e^{-v_i}\]
maps $H$ to $H_k$. The infinite composition
\[\Phi:=\ldots \varphi_k \varphi_{k-1}\ldots \varphi_1 \varphi_0 \in Aut(P)\]
is a formal symplectic coordinate transformation that removes all monomials $p^aq^b$, $a \neq b$ from our Hamiltonian $H$, hence we see

\begin{theorem} For any non-resonant $H=h_0+O(3) \in P$ there exists an automorphism
 $\Phi \in Aut(P)$ such that
\[ \Phi(H)=B_H,\]
where $B_H$ is a series of the form
$$B_H:=\sum_{a \in \NM^d}C_ap^aq^a. $$
The series $B_H$ is called the {\em Birkhoff normal form} of $H$. 
\end{theorem}

There exist several variants of this algorithm, differing in details. For example, one may remove several terms at the same time,  which may lead to different normalising transformations $\Phi$, but it is known that different choices lead to the same series $B_H$. 

\begin{example}
Take $d=1$ and consider the Hamiltonian function
$$H(q,p)=pq+p^3+q^3. $$
We determine a sequence of vector fields $v_0, v_1, v_2,\ldots$, where $v_k$ obtained be removing simultaneously all terms of degree $k+3$. The iteration then begins with
\begin{align*}
H_0&=pq+p^3+q^3\\
v_0&=\{-, 1/3(p^3-q^3)\}\\
H_1&=pq-3p^2q^2+4p^4q+4pq^4+O(6)\\
v_1&=0\\
H_2&=pq - 3q^2p^2 + 4p^4q + 4q^4p - 3/2p^6 - 3/2q^6 - 12p^3q^3 +O(7)\\
v_2&=\{-, 4/3(p^4q-pq^4)\}\\
H_3&= pq- 3p^2q^2  -12p^3q^3 +O(7) \\
\dots
\end{align*}
From this we can read off the first three terms of the Birkhoff normal form and continuing the process one finds
$$B_H(\tau)=\tau-3\tau^2-12\tau^3-105 \tau^4-1206\tau^5-16002\tau^6-232416\tau^7-3592377\tau^8+o(\tau^8) $$
where $\tau=pq$. (One can show that in this case the inverse power series 
to $B_H(s)$ is a hypergeometric function: $\tau= b \cdot \mbox{}_2F_1(1/3,2/3,1;27 b)$, $b:=B_H(\tau)$.)
\end{example}
\subsection{ The Moser Extension}
\label{SS::Moser}
As the monomials $p_iq_i$ ($i=1,2,\ldots,d$) Poisson commute with the Birkhoff normal form $B_H$, 
Birkhoff normalisation implies that any non-resonant Hamiltonian $H$ is {\em formally completely integrable}. To express this in a manifest way, it is useful to enlarge the ring $P$ and consider
\[Q:=\CM[[\tau,q,p]]=\CM[[\tau_1,\ldots,\tau_d,q_1,\ldots,q_d,p_1,\ldots,p_d]],\] 
with the extra $\tau$-variables, introduced by Moser. With the same definition
of the Poisson-bracket as before, $Q$  becomes a Poisson algebra with Poisson 
centre $Q_0:=\CM[[\tau]]$. We will assign weight $=2$ to the variables $\tau_i$,
so that the $d$ elements
\[ f_i:=p_iq_i-\tau_i \in Q\]
are homogeneous of degree two. These elements Poisson commute, $\{f_i,f_j\}=0$,
and we obtain a Poisson commuting sub-algebra
\[ \CM[[\tau,f]]=\CM[[\tau,f_1,f_2,\ldots,f_d]]=\CM[[\tau,p_1q_1,\ldots,p_dq_d]]\]
containing $Q_0$. The $f_1,f_2,\ldots,f_d$ also generate an ideal\footnote{Here and in the sequel, the notation $\langle f_1,\dots,f_k\rangle$ stands for the ideal generated by elements $f_1,\dots,f_k$.}
\[I=\langle f_1,f_2,\ldots,f_d \rangle  \subset Q=\CM[[\tau,q,p]]\] 
and clearly, the canonical map
\[  \CM[[\tau,p,q]] {\to} \CM[[p,q]],\;\;\;p_i \mapsto p_i,\;q_i\mapsto q_i,\; \tau_i \mapsto p_iq_i .\]
induces an isomorphism of the factor ring $Q/I$ with our original ring $P$:
\[ Q / I \stackrel{\sim}{\to} P . \]
Although $f_i$ maps to zero under this map, the derivation $\{-,f_i\}$ induces
the non-zero derivation $\{-,p_iq_i\}$ on $P$, so the map $Q \to P$ is {\em not}
a Poisson-morphism. The ideal $I^2=\langle f_1,\dots,f_k \rangle^2 \subset Q$ is the square of the ideal $I$, i.e. generated by the elements $f_if_j$, $1 \le i,j\le d$, and plays a very distinguished role in dynamics. 

\begin{lemma} If $T\in I^2$, then $H$ and $H+T$ induce the same Hamiltonian vector 
field on $Q/I=P$. 
\end{lemma}
\begin{proof}
If $T \in I^2$, then $\{h,T\} \subset I$. As a consequence, the difference
between $\{h,H\}$ and $\{h,H+T\}$ belongs to $I$, which is mapped to $0$ in $P$.
\end{proof}

Extending the multi-index notation in an obvious way, we can write
\[ p^aq^a=(\tau+f)^a=\tau^a+\sum_{i=1}^d \partial_{\tau_i} \tau^a f_i+I^2 .\]
The term $\tau^a$ is in the centre of $Q$, whereas the above lemma
implies that  $p^aq^a$ and $\sum_{i=1}^d\d_{\tau_i}\tau^a f_i$ define the 
same derivation on the ring $P=Q/I$.

We can consider the Birkhoff normal form series $B(pq)=B_H(pq)$ as an element 
of $Q$. When we write $pq=\tau+f$, then we find:
\[ B(\tau+f)=B(\tau)+\sum_{i=1}^db_i(\tau)f_i \ \mod I^2 .\]
The formal power series $b_1,\dots,b_d \in \CM[[\tau]]$ are obtained as partial derivatives of $B$, 
considered as a series in the $\tau_i$-variables:
\[b=(b_1,\dots,b_d)=\del B(\tau) .\]

The first term $B(\tau)$ we also call the Birkhoff normal form, written in the
$\tau$-variables. It belongs to the Poisson centre $Q_0$ and is dynamically
trivial, but gets mapped to the non-trivial element $B_H \in P$. The second
term $\sum_{i=1}^db_i(\tau)f_i$ carries the dynamical information in $Q$, but is
mapped by the canonical map $Q \to P$ to zero. 

One has $b(0)=\alpha$, and the higher order terms describe how the frequencies change with $\tau$ and for 
this reason we call it the {\em formal frequency map}.  If the system happens to be integrable, then the series are convergent and the vector $b(\tau)=(b_1(\tau),b_2(\tau),\dots,b_d(\tau))$ is the frequency of motion on the corresponding manifold defined by
$f_i(\tau,q,p)=0$, $i=1,2,\ldots,d$.
\begin{example}
 Take $d=1$, the Hamiltonian
 $$H(q,p)=B(qp)=pq+(pq)^2 $$
 is already in Birkhoff normal form. In the Moser algebra we have
 \begin{align*}
 H(q,p) &= \tau+\tau^2+(pq-\tau)+2\tau (pq-\tau)+(pq-\tau)^2\\
   &= B(\tau)+B'(\tau) f+f^2\\
 &=(1+2\tau) pq \ \mod I^2 \oplus \CM[[\tau]] .
 \end{align*}
\end{example}

\section{\bf  The Hamiltonian Normal Form}
\label{HamiltonianNormalForm}
\subsection{Introductory example}
 Consider again the anharmonic oscillator
$$H(q,p)=pq+p^3+q^3. $$
We will give an alternative method to bring it back to a normal form. At
first encounter it may look silly and just more cumbersome, but we will see
that in general it leads to a great simplification of the set-up.
First, we detune the frequencies and consider the function:
$$F_0= (1+\omega)pq+p^3+q^3.$$
The idea is then to take back this function to 
$$A_0=(1+\omega)pq$$
via a Poisson automorphism.
The initialisation of our iteration is therefore
\begin{align*}
A_0&=(1+\omega)pq,\\
F_0&=A_0+p^3+q^3=(1+\omega)pq+p^3+q^3.
\end{align*}
Our first objective is to get rid of the cubic term. This is accomplished by observing that
$$p^3+q^3=\{ A_0,\frac{1}{3 (1+\omega)} (p^3-q^3)\}$$
So we choose
$$v_0=\{ -, \frac{1}{3(1+\omega)} (p^3-q^3)\}.$$
so that the automorphism $\p_0=e^{-v_0}$ transforms $F_0$ into
\begin{align*}
F_1(q,p)&=e^{-v_0}(A_0+p^3+q^3)  \\
  &=A_0-v_0(p^3+q^3)+\frac{1}{2!}v_0^2(A_0)+O(6)\\
&=(1+\omega)pq-\frac{3}{1+\omega}q^2p^2+\frac{4}{(1+\omega)^2}(p^4q+q^4p)+O(6).\\
\end{align*}
Now $H$ is recovered from $F_0$ by setting $\omega=0$. As the automorphism $\p_0$ sends the line $\omega=0$ to
itself, so $H=H_0$ is mapped to the restriction of $F_1$ to $\omega=0$:
$$H_1(q,p)= pq-3 q^2p^2+4(p^4q+q^4p)+O(6)$$
So in this way we got rid of the cubic term in $H_0$.

Let us now proceed to the next order.
Now we look at the terms of degree 4 and 5. The degree 5 term can be eliminated 
by a Hamiltonian vector field:
$$\{ A_0, \frac{4}{3(1+\omega)^3} (p^4q-q^4p)\}=\frac{4}{(1+\omega)^2}(p^4q+q^4p), $$
but something new happens: to suppress the second term $-\frac{3}{1+\omega}q^2p^2$ we need
a non-Hamiltonian Poisson vector field. This is done in two steps. First we note that
$$q^2p^2=(qp-\tau)^2+2\tau qp-\tau^2 .$$
The reason for rewriting the term in this way, is the fact that the terms in the space
$I^2\oplus \CM[[\omega,\tau]]$
{\em do not change the Hamiltonian derivation on the curve $qp=\tau$}.
So we choose 
$$v_1=\{ -, \frac{4}{3(1+\omega)^3} (p^4q-q^4p)\}-\frac{6\tau}{(1+\omega)}\d_\omega $$ 
and get that
\begin{align*}v_1(A_1)&=\frac{4}{(1+\omega)^2}(pq^4+qp^4)-\frac{6\tau}{1+\omega} pq\\
  &=\frac{4}{(1+\omega)^2}(pq^4+qp^4)-\frac{3}{1+\omega} q^2p^2 \ \mod I^2 \oplus \CM[[\omega,\tau]] .
\end{align*}
The difference  $[F_1]_4^6-v_1(A_1)$ is seen to be
\[ \frac{ 6\tau}{1+\omega}pq -\frac{3}{1+\omega}p^2q^2=\frac{3\tau^2}{1+\omega}-\frac{3(pq-\tau)^2}{1+\omega} \in I^2+\CM[[\omega,\tau]]\]
so we get that
$$F_2=e^{-v_1}F_1= (1+\omega)\tau+\frac{3}{1+\omega}\tau^2+(1+\omega)f +\frac{-3}{1+\omega} f^2+O(6).$$
So the transformation did not bring $F_1$ back to $A_0$ as there are, like in the Birkhoff normal form, terms which cannot be eliminated by the iterative process. But unlike the Birkhoff normal form, these
residual terms are irrelevant for studying the dynamics!
Now what happens to our function $H_1$? It is mapped to $H_2$, the restriction of $F_2$ to 
$$\p_1(\omega)=0,\ \p_1=e^{-v_1} $$
As $v_1$ contains a non-Hamiltonian term $-6\frac{\tau}{1+\omega}\partial_w$, the line $\omega=0$ is
not preserved and more precisely we have:
$$\p_1(\omega)=\omega-\frac{6\tau}{1+\omega} .$$
If we solve the equation $\phi_1(\omega)$ for $\omega$,
$$ \omega-\frac{6\tau}{1+\omega} =0,$$
and substitute the result in $F_1$, then we recover the first terms of the Birkhoff normal form. 

In the next step, we define
\begin{align*}
A_2&=(1+\omega)\tau+\frac{3}{1+\omega}\tau^2+(1+\omega)f +\frac{-3}{1+\omega} f^2\\
   &=(1+\omega)pq +\frac{3}{1+\omega}\tau^2 \;\; \mod I^2 ,
\end{align*}
where $f:=pq-\tau$, $I=(f)$ so that 
\[ F_2=A_2 +O(6).\]
Then we have to look at the terms of degree
$6,7,8,9$ appearing in $F_2$ and determine a vector field $v_2$
\[v_2(A_2) =[F_2]_6^{10} +t , \;\;\; t \in I^2+\CM[[\omega,\tau]] .\]
To see these terms, we have to keep much more terms in the expansions. 
We find
\[A_{H,0}=A_{H,1} =(1+\omega)\tau, \;\;\;\]
\[A_{H,2}=(1+\omega)\tau+\frac{3\tau^2}{1+\omega},\]
\[A_{H,3}=(1+\omega)\tau+\frac{3\tau^2}{1+\omega}+\frac{6\tau^3}{(1+\omega)^3} - \frac{9\tau^4}{(1+\omega)^5}.\]
In this example the denominators appearing are rather simple; in examples
with more variables much more complicated denominator structure arise.
\subsection{The small denominator ring}
The above process shows the appearance of rational functions. It is important that the Hamiltonian
normal form iteration is not only an iteration in power series, but makes sense for series with
coefficients which are rational functions in the frequencies. The Birkhoff normal relates to the
expansion at some point of these expressions. But if we truncate a Taylor expansion say
$$\frac{1}{1-x}=1+x+x^2+\dots $$
and get polynomials
$$1+\dots+x^n $$
then these will hardly give any information beyond the disk of radius one. In the case of Hamiltonian mechanics with more than one degree of freedom, we have to imagine that during the iterative process the resonance hyperplane define a dense set. So apparently there is little hope that the Birkhoff normal form will be of any
use~\footnote{The work of Stolovitch cited in bibliography seems to indicate the contrary however...}.

As we saw in our example, we need to consider the Moser variables $\tau$ independently from the frequency variables, which means that we add variables $\omega_1,\dots,\omega_d$.
For a fixed frequency vector $\alpha \CM^d$, we define the ring $SD_{\alpha}$ of {\em small denominators} at $\alpha$ as the subbing of the field $\CM(\omega)=\CM(\omega_1,\omega_2,\ldots,\omega_d)$ of rational functions, defined by 
localisation of $\CM[\omega]=\CM[\omega_1,\omega_2,\ldots,\omega_d]$ with respect with the multiplicative subset $S$ generated by all linear polynomials $(\a+\omega, J), J \in \ZM^n\setminus \{0\}$:
\[SD_{\alpha}:=\CM[\omega]_{S}:=\CM[\omega,\frac{1}{(\a+\omega,J)}, J \in \ZM^n\setminus\{0\}]
\subset \CM(\omega) .\]
The elements of this ring may have poles along the resonance hyperplanes
$$H_J=\{ \omega \in \CM^d: (\a+\omega,J)=0 \} $$
We will be concerned with the following ring
\begin{definition}
$$R:=SD_{\alpha}[[\tau,p,q]] \subset \CM[[\omega,\tau, p,q]],$$
a subring of the power series in $4d$ variables 
$$\omega_1,\omega_2,\ldots,\omega_d,\tau_1,\ldots,\tau_d,q_1,\ldots,q_d,p_1,\ldots,p_d .$$
We provide $R$ with the standard Poisson bracket, so that the Poisson center of
$R$ is the ring
\[R_0:=SD_\a[[\tau]]=\subset \CM[[\omega_1,\ldots,\omega_d,\tau_1,\ldots,\tau_d]].\]
\end{definition}
Note that the variable $\omega$ has a more {global} character. The relevant filtration of $SD_\a$ is
given by the order of the poles along the resonance hyperplanes, but we will not use it in this paper.

It will be shown that, contrary to what is expected for the Birkhoff normal
form, this new normal form iteration converges in an appropriate sense. As 
a result, we have a better control over the invariant tori.
The following sub-algebra of $R$ is of importance for our discussion:

\begin{definition} We call the Poisson-commutative algebra
\[ M:=SR_0[[f]]=R_0[[pq]] \subset R,\]
the {\em Moser-algebra} of $R$. We also put
\[M_0:= R_0+ I^2 \cap M \subset M,\]
where 
\[ I:=\langle f_1,f_2,\ldots,f_d \rangle  \subset R.\]
\end{definition}

As before, we denote the vector space of Poisson derivations of the $R$ by 
$\Theta(R)$, which has the structure of a module over the Poisson centre 
$R_0$. One has:

\begin{lemma} The Poisson derivations of $R$ decompose into Hamiltonian and non-exact parts:
$$\Theta(R)=\Ham(R) \oplus \Der(R_0)  .$$
\end{lemma}
So an element of $\Theta(R)$ is of the form
\[ v=\{-, h \}+ w\] 
with 
\[ w=\sum_{i=1}^d a_i \frac{\partial}{\partial \omega_i} + b_i \frac{\partial}{\partial \tau_i}, \;\;\;a_i,\;b_i \in R_0.\]

\subsection{{The Hamiltonian normal form iteration}}

\begin{definition} The {\em $\omega$-extension of} $H \in P$ is the element
\[ F= H+\sum_{i=1}^d \omega_i p_i q_i \in R \]
For the $\omega$-extension of $h_0=\sum_{i=1}^d\alpha_ip_iq_i$ we keep a 
special notation:
\[A_0:=\sum_{i=1}^d(\alpha_i+\omega_i) p_iq_i \in R .\]
\end{definition}
So $A_0$ is obtained from $h_0$ by {\em detuning} the frequencies in the most 
general way. One also may interpret it as a versal deformation of $h_0$. 
Starting from a Hamiltonian 

\[ H=\sum_{i=0}^d \alpha_i p_iq_i+O(3),\]
we first form the $\omega$-extension of $H$:
\begin{align*}
F_0&:=H+\sum_{i=1}^d \omega_ip_iq_i\\
   &=A_0+O(3)\\
   &=A_0+[F_0]_3^4+O(4).
   \end{align*}

When we solve a {\em homological equation} of the form
$$v_0(A_0)=[F_0]_3^4 +t_0,\;\;\; t_0 \in M_0 ,$$ 
we obtain a Poisson derivation $v_0$, which we can 
exponentiate to produce an automorphisms  $e^{-v_0}$. The application of $e^{-v_0}$ to $F_0$ produces $F_1$, where this term is  removed; we put $A_1=A_0+t_0$. In this particular it turns out that $t=0$ but at the next level we have to solve
$$v_1(A_1)=[F_1]_4^6 + t_1 ,\;\;\; t_1 \in M_0$$ 
for the degree $4$ and $5$ part of $F_1$ on $A_1$ and, as a general rule $t_1 \neq 0$. Then the application of $e^{-v_1}$ to $F_1$ produces $F_2$, where now these terms of degree $4$ and $5$ are removed, but certain terms in $M_0$ are introduced. These remaining terms we add to $A_1$ and obtain $A_2$. Next we solve the homological equation for the terms of degree $6,7,8,9$ of $F_2$, but now on  $A_2$, etcetera. Thus we obtain a by iteration a sequence of triples
\[ (F_n, A_n, v_n ), \;\;\;n=0,1,2,\ldots\]

\subsection{Ordering the terms of the expansion}
For convenience of the reader we include the following diagram that
indicates the degrees of the quantities that appear in the iteration.\\
\vskip0.1cm
\[
\begin{array}{|c||c|c|c|c|c|c|c|c|c|c|c|c|c|c|c|c|c|}
\hline
 &2&3&4&5&6&7&8&9&10&11&12&13&14&15&16&17&18\\
\hline
\hline
F_0&\bu&\cro&\bs&\bs&\bs&\bs&\bs&\bs&\bs&\bs&\bs&\bs&\bs&\bs&\bs&\bs&\bs\\
\hline
F_1&\bu&\circ&\cro&\cro&\bs&\bs&\bs&\bs&\bs&\bs&\bs&\bs&\bs&\bs&\bs&\bs&\bs\\
\hline
F_2&\bu&\circ&\bu&\circ&\cro&\cro&\cro&\cro&\bs&\bs&\bs&\bs&\bs&\bs&\bs&\bs&\bs\\
\hline
F_3&\bu&\circ&\bu&\circ&\bu&\circ&\bu&\circ&\cro&\cro&\cro&\cro&\cro&\cro&\cro&\cro&\bs\\
\hline
F_4&\bu&\circ&\bu&\circ&\bu&\circ&\bu&\circ&\bu&\circ&\bu&\circ&\bu&\circ&\bu&\circ&\cro\\
\hline
\end{array}
\]\ \\
\vskip0.2cm
The bullets $\bu$ and circles $\circ$ represent terms of $A_n $.   They belong to the $M_0$ part of the Moser-algebra : the $\bu$ terms are constant in columns, the circles  $\circ$ are zero, as the Moser-algebra only has terms of even degree.\\
So $\bu$ and $\circ$ represents the {\em normal form range}, consisting 
of terms of $F_n$ of degree
\[ 2 \le degree <2^n+2\]
The crosses $\times$ represent the terms of $F_n$ that determine the derivations $v_n$. 
These make up what we call the {\em active range} of degrees: 
\[2^n+2 \le degree <2^{n+1}+2 .\] 
The black squares $\bs$ represent the terms of $F_n$ that of degree higher
than $2^{n+1}+2$ that do not directly influence the next iteration step, but of course 
must be carried along.\\

We now rewrite the iteration in a form where this trichotomy in degrees 
is manifest. Consider the decomposition
\[ F_n:=A_n+M_n+U_n=\bu+\times+\bs,\]
where
\[A_n:=[F_n]^{2^n+2},\;\;\; M_n:=[F_n]_{2^n+2}^{2^{n+1}+2},\;\;\; U_n:=[F_n]_{2^{n+1}+2},\]
are the lower, middle and upper part of $F_n$.

The Hamiltonian normal form iteration produces a sequence $(F_n,A_n,v_n)$: 
the series 
\[ F_0=H+\sum_{i=1}^d \omega_i p_iq_i =A_0+O(3)\]
is transformed by 
\[\Phi_n:=e^{-v_{n-1}}\dots e^{-v_0} \]
to a series of the form
\[F_n=A_{n}+O(2^{n}+2) .\]
If we let $n$ go to $\infty$, we obtain a formal Poisson automorphism
\[\Phi_{\infty}:=\ldots e^{-v_n}\dots e^{-v_0} \in Aut(R),\]
and obtain
\[F_\infty:=\Phi_{\infty}(F_0)=A_\infty,\;\;\;A_{\infty} \in A_0+M_0\]
The automorphism $\Phi_{\infty}$ transforms the perturbation $F_0=A_0+O(3)$
back to the normal form $A_0$, plus terms that have no effect on
the dynamics.

\begin{definition} Let $H=\sum_{i=1}^d \alpha_i p_iq_i+\ldots \in P$.
The {\em $k$-th Hamiltonian normal form of $H$} is the series
\[A_{H,k} :=A_k \;\; \mod I \in \CM[[\omega,\tau]],\]
obtained from $A_k$ by the substitution $p_iq_i=\tau_i$.
The {\em Hamiltonian normal form of $H$} is the series
\[ A:=A_H:=A_{\infty}\;\; \mod I  \in \CM[[\omega,\tau]].\]
\end{definition}

\subsection{ The homological equation}
\label{SS::Homequation}
We now describe a specific way to solve the homological equation for the $v_k$.
In the algorithm for the Birkhoff normal form the derivations $v_0, v_1, v_2, \ldots$ were determined by applying them to the fixed element $h_0$, whereas 
here the sequence is determined by applying them to elements $A_0, A_1, A_2,\ldots$ that is determined in the iteration process. The infinitesimal action 
$$ \Theta(R) \to R,\;\; v \mapsto v(A_0) $$
on
\[A_0:=\sum_{i=1}^d(\alpha_i+\omega_i) p_iq_i \in R .\]
takes a simple form in the monomial basis:\\
\begin{align*}
\{ A_0, p^a q^b\}&=(\alpha+\omega,a-b)p^aq^b,\\
\d_{\omega_k} A_0&=p_kq_k,\\
\d_{\tau_k} A_0&=0 .\\
\end{align*}

\begin{definition} 
\label{D::L}
We define a $\CM[[\omega,\tau]]$-linear map 
\[L: R \to \Theta(R)=\Ham(R) \oplus \Der(R_0), m \mapsto Lm \]
by setting for $a \neq b$:
\[ Lp^aq^b :=\{-,\frac{1}{(\alpha+\omega,a-b)}p^aq^b\}.\]
For $a=b$, or more generally for a series 
\[m=g(p_1q_1,p_2q_2\,\ldots,p_dq_d)=g(pq)\]
we set
\[Lm:=\sum_{i=1}^d \frac{\partial g(\tau)}{\partial \tau_i}\d_{\omega_i}.\]
\end{definition}

 
\begin{definition} For $A=A_0+T,\ T \in I^2$ we define a linear map
\[ j_{A}: R \to \Theta(R)\]
in terms of $L$ by the formula
\[j_{A}: m \mapsto Lm-L(Lm(T))=L(m-Lm(T))\]
\end{definition}

\begin{proposition} For any $A=A_0+T$, $T \in I^2$ and  any $m \in R$, we have 
\[j_A(m)(A)=m+t,\;\;t \in R_0+I^2,\]
\end{proposition} 
\begin{proof}
First, for $A=A_0$ we have $j_{A_0}=L$.
For $m=p^aq^b$ with $a \neq b$ we have 
\[j_{A_0}(m)(A_0)=\{A_0,\frac{1}{(\alpha+\omega,a-b)}p^aq^b\}=p^aq^b=m\]
and for $m=g(pq)$ we have, with $g_i=\partial_{\tau_i} g$,
\begin{align*}
j_{A_0}(m)(A_0)&=\sum_{i=1}^d g_i(\tau) \frac{\partial A_0}{\partial \omega_i}=\sum_{i=1}^n g_i(\tau)p_iq_i\\
          &=\sum_{i=1}^d g_i(\tau) f_i \ \mod R_0=g(pq)\ \mod R_0+I^2,
\end{align*}
where we used the Taylor expansion
\[g(pq)=g(\tau+f)=g(\tau)+\sum_{i=1}^d g_i(\tau) f_i\ \mod I^2.\]
This shows the correctness for $T=0$. For the general case $A=A_0+T$, we get

\begin{align*}
 j_A(m)(A_0+T)&=Lm(A_0)+Lm(T)-L(Lm(T))A_0-L(Lm(T))(T) \\
              &=m+Lm(T)-Lm(T)-L(Lm(T))(T)\ \mod R_0+I^2\\
              &=m-L(Lm(T))(T)\ \mod R_0+I^2.\\
 \end{align*}
Because $T \in I^2$, it follows that $Lm(T) \in I$. Furthermore,
for any $g \in I$, we have $Lg(T) \in I^2$. This can be seen by writing
$g$ as $\CM[[\omega,\tau]]$-linear combination of terms of the
form $p^aq^b f_i$. If $a \neq b$, $\{T,p^aq^bf_i\} \in I^2$, whereas for
$a=b$, we obtain a combination of terms $\partial_{\omega_i}T$, which is in
$I^2$, as the generators $f_i=p_iq_i-\tau_i$ are independent of $\omega_i$. 
\end{proof}

\subsection{The HNF (Hamiltonian Normal Form) iteration}
With an Hamiltonian $H=\sum_{i=1}^d \alpha_ip_iq_i+O(3) \in P$ as input,
we begin the iteration with the {\em initialisation step} 
\begin{align*}
F_0&=H+\sum_{i=1}^d \omega_ip_iq_i=A_0+O(3)\\
A_0&=\sum_{i=1}^d(\alpha_i+\omega_i)p_iq_i\\
v_0&=j_{A_0}\left( [F_0]_3^4\right).\\
\end{align*}

The next terms are determined by the {\em iteration step}: from $F_n, A_n$ we then obtain
\begin{align*}
F_{n+1}&= e^{-v_{n}}F_{n},\\
A_{n+1}&=A_{n}+\left[F_{n}-v_n(F_n)\right]_{2^{n}+2}^{2^{n+1}+2},\\
v_{n+1}&=j_{A_{n+1}}(\left[F_{n+1}\right]_{2^{n+1}+2}^{2^{n+2}+2})
\end{align*}
It is useful to define the {\em increments}
\[S_{n+1}:=\left[F_{n}-v_n(F_n)\right]_{2^{n}+2}^{2^{n+1}+2},\]
so that:
\[A_{n+1}=A_n+S_{n+1} .\]
There are a few simple but important points to notice:

\begin{proposition}
 \begin{enumerate}[{\rm i)}]
\item The derivation $v_n$ has order $2^n$, i.e. $v_{n}=[v_{n}]_{2^n}$.
\item $F_{n} =A_{n}+O(2^{n}+2)$.
\item $S_n \in M_0$.
\end{enumerate}
\end{proposition} 
\begin{proof}
  
i) From the recursive definition we see that $v_n$ is obtained by
solving the homological equation with the terms of degrees
${2^n+2}$ up to ${2^{n+1}+2}$ from $F_n$. Taking Poisson-bracket
with a term of degree $2^n+2$ shifts degrees by $2^n$, and similarly for the
non-exact part of $v_n$. So indeed $v_n$ has order $2^n$.\\ 

ii) This follows from an easy induction on $n$. By definition, the statement 
holds for $n=0$. Let us assume that
 $$F_n=A_{n}+O(2^{n}+2)$$ 
From the definition of $F_{n+1}$ we have 
\[F_{n+1}=e^{-v_n}F_n=F_n-v_n(F_n)+\frac{1}{2}v_n^2(F_n)-\ldots \]
and as $v_n$ has order $2^n$, it follows that
\[v_n^2(F_n)=O(2+2^n+2^n)=O(2^{n+1}+2).\]
So we have
\[ F_{n+1}=A_{n}+[F_n-v_n(F_n)]_{2^n+2}^{2^{n+1}+2}+O(2^{n+1}+2)=A_{n+1}+O(2^{n+1}+2).\]
iii) We use induction on $n$ and assume that $S_{n} \in M_0=R_0 + I^2 \cap M$. 
From (ii) we have
\[F_n =A_{n}+O(2^{n}+2).\]
The derivation $v_n$ is constructed
to solve the homological equation up to terms of high order:  
\[v_n(A_{n})=[F_n]_{2^n+2}^{2^{n+1}+2}+t+O(2^{n+1}+2),\;\;\; t \in M_0\]
As we have
\[v_n(F_n)=v_n(A_{n}+O(2^{n}+2))=v_n(A_{n})+O(2^{n+1}+2),\]
we see that the increment
\[[F_n-v_n(F_n)]_{2^n+2}^{2^{n+1}+2} \in  M_0,\]
hence also $S_{n+1} \in M_0$.
\end{proof}
Let us denote by $B_n$ the sum of the middle and upper term, so that
$$F_n=A_n+B_n. $$
As $v_n$ is of order $2^n$, we have
$$ \left[F_{n}-v_n(F_n)\right]_{2^{n}+2}^{2^{n+1}+2}=\left[B_{n}-v_n(A_n)\right]_{2^{n}+2}^{2^{n+1}+2}.$$
Now write
$\tau_n=\left[ - \right]_{2^{n}+2}^{2^{n+1}+2},\ \s_n=\left[ - \right]_{2^{n+1}+2} $,
The iteration is defined by:
\begin{align*}
B_{n+1}&= \s_n(e^{-v_{n}}F_{n}),\\
A_{n+1}&=A_{n}+\tau_n(B_n-v_n(A_n)),\\
v_{n+1}&=j_{A_{n+1}}(\tau_{n+1}(B_{n+1}))
\end{align*}
and is obtained by iterating the maps:\\
\begin{center}
 
\begin{tabular}{| c |}
\hline
\ \\

$\phi_n:(A,B,v) \mapsto (A,0,0)+f_n(A,B,v) $\\
\ \\
$f_n(A,B,v)=\left(\tau_n(B-v(A)), \s_n(e^{-v}(A+B)),j_{A+\tau_n(B-v(A))} \circ \tau_{n+1}(e^{-v}(A+B))\right) $\\
\ \\
\hline
\end{tabular}

\end{center}
 \ \\
The precise form of the iteration is irrelevant the main point is that the orders of the exponents which are involved grow rapidly.
Contrary to the  the classical Kolmogorov scheme or the Newton method,
we disregard the quadratic nature of the iteration and concentrate only on the fact that the order of the terms rapidly increases.\\

We shall see that it is important to be very precise about the degrees in the truncation, as 
otherwise we would loose control over the fields $v_n$ and the convergence 
properties of the iteration could be spoiled.
 

\section{Small denominators}
\label{SmallDenominators}
In the next three sections, we prove a convergence result for the HNF iteration.
For this, we will work in various Banach space completions of spaces of
holomorphic functions. The small denominators that arise in the iteration
stem from two different sources. The first one is classical: by taking
derivatives of a function (or more generally Hadamard products), we need to shrink the domain of definition, which results in some loss of control over the function. The simplest avatar of this phenomena is given by the {\em Cauchy inequalities}, when we compute a derivative. The second source comes from the comparison of different types Banach spaces. Such comparisons lead to norm estimates
very similar that of a differential operator. Changing of Banach spaces
is particularly important, as it enables us to choose our Banach space
accordingly to the type of operations we are considering.

\subsection{Banach space completions of $\Ot$}
For an open set $U \subset \CM^d$, the set $\Ot(U)$ of holomorphic functions
on $U$ has the natural structure of a Fr\'echet space. 

For an arbitrary set $X \subset \CM^d$, we denote by $\Ot(X)$ the space of function on $X$
which are holomorphic in a neighbourhood of $X$. For instance, if $X$ is a point then $\Ot(X)$ is the space of germs of holomorphic functions at this point.

The set $\Ot(X)$ has the structure of an $LF$-space~\cite{Grothendieck_EVT}. However, we will not use this $LF$-topology in this paper, but rather work with certain Banach spaces
obtained from $\Ot(X)$ by putting additional boundary conditions and which are {\em Banach space completions} of the space
$\Ot(X)$.\\

Recall that a function $f:X \to \RM$ defined on a closed subset $X$ of Euclidean space is called {\em Whitney differentiable} at $x \in X$, if there exists functions $D^If(x)$ called the Whitney derivatives of $f$ at $x$ such that:
$$f(y)=\sum_{|I|\le m}\frac{D^If(x)}{|I|!}(y-x)^I+o(\|y-x\|^m) .$$
The sole difference that the Whitney definition bears, with respect to the standard definition, is the requirement of uniformity of the limit in the $x,y$ variables. The {\em Whitney extension theorem} says that any Whitney differentiable function is the restriction of a $C^\infty$-function~\cite{Whitney_extension}.
%

\begin{definition}
Given an set $X \subset \CM^d$, we denote by $\Ot^k(X)$ 
the Banach space of functions which are holomorphic
in the interior of $X$ and have a bounded Whitney $C^k$-extension to the closure of $X$:
$$\Ot^k(X)=C^k_b(\overline{X},\CM) \cap \Ot(\mathring{X})$$
For $k=0$, we use the notation $\Ot^c$ instead of $\Ot^k$.\\
When $X$ is a Lebesgue measurable set, we define similarly $\Ot^h(X)$ as
the Hilbert space of square integrable functions on $\overline{X}$ for the Lebesgue
measure which are holomorphic in the interior of $X$.
\end{definition}
We denote by $|f|$ the norm of $f \in \Ot^\bullet$ in any of these completions.
This will not lead to any ambiguity if we always explicitly state
to which Banach space $f$ belongs.

In the text we will frequently make use of open and closed polydiscs.

\begin{definition}
  For a $d$-tuple of positive numbers $\rho:=(\rho_1,\rho_2,\ldots,\rho_d)$ we set
  \[D_{\rho}:=\{z=(z_1,z_2,\ldots,z_d) \in \CM^d\;|\;\| z_i \| \leq \rho_i,i=1,2,\ldots, d\}\]
  
  For $\rho=(r,r,\ldots,r)$ we simply write $D_r:=D_{\rho}$ and call it the
  polydisc of radius $r$.
\end{definition}
\subsection{The Cauchy-Nagumo lemma}

\begin{definition}
Given two sets $V \subset U \subset \CM^d$, 
we denote by $\dt(U,V)$ the supremum of the real numbers $r$ for which
$$V +  D_r \subset U , $$  
and call it the {\em Huygens distance} between $U$ and $V$.
\end{definition}
 
The derivative of a holomorphic function inside an open set $U \subset \CM$ remains holomorphic and the map
$$\Ot(U) \to \Ot(U),\ f \mapsto f' $$ 
is even continuous in the Fr\'echet topology. This is of little use because in Fr\'echet spaces
all kind of pathologies may occur. If we add boundary conditions and consider for instance the
Banach completion $\Ot^c(U)$, then the map becomes an unbounded operator but the general theory
of unbounded operators is not sufficiently precise for our purposes.

\begin{lemma}
  \label{L::Cauchy_Nagumo}
  Let $V \subset U \subset \CM^d$ be two sets with positive Huygens distance:
$\dt(U,V)=r>0.$
For a differential operator 
\[P=\sum_{|J| \leq k} a_J \d^J  \in L(\Ot^c(U),\Ot^c(V))\] 
of order $k$ we have:
\[ \|P\| \le C \frac{k!}{r^k}  \]  
where $C=\sup_{|J| \leq a_k,\ z \in U}| a_J(z) |$. 
\end{lemma}
\begin{proof}
If $z \in V$ and $f \in \Ot^c(U)$, then one has
 $$f(z)=\frac{1}{(2\pi i)^d}\int_{\g_z} \frac{f(\xi)}{\prod_{i=1}^d(\xi_i-z_i)}  d\xi_1 \wedge \dots \wedge d\xi_d, $$
where $\g_z$ denotes the cycle defined by $|\xi_i-z_i|=r$. We write it
symbolically as
$$f(z)=\frac{1}{(2\pi i)^d}\int_{\g_z} \frac{f(\xi)}{(\xi-z)}  d\xi . $$
Differentiation under the integral sign leads to
\[\d^I f(z)=\frac{I!}{(2 \pi i)^d}\int_{\g_z} \frac{f(\xi)}{(\xi-z)^{1+I}} d\xi.\]
We parametrise $\g_z$ by:
$$\theta \mapsto \xi(\theta):=z+r e^{\langle 2\pi i, \theta\rangle} $$
with $e^{\langle 2\pi i, \theta\rangle}=(e^{ 2\pi i \theta_1},\dots,e^{ 2\pi i \theta_d})$
and thus
$$d\xi=(2\pi i r)^d  e^{\langle 2\pi i, \theta\rangle} d\theta,$$
so that
$$\d^I f(z)=\frac{I !}{r^{|I|}}\int_{0}^1 \frac{f(\xi(\theta))}{e^{2\pi i \theta}} d\theta,$$
so finally
$$| \d^I f| \leq \frac{I!}{r^{|I|}}| f |. $$
From this the lemma follows.
\end{proof}
\begin{corollary}
 \label{C::local}
 Let $V \subset U \subset \CM^d$ with Huygens distance $\dt(V,U)=r>0.$ 
Then the restriction mapping
$$\rho^{ck}:\Ot^c(U) \to \Ot^k(V) $$
has norm smaller than $\frac{k!}{r^{|k|}}$
\end{corollary}

This corollary is a first manifestation of what we call the {\em local equivalence} between $\Ot^c$ and $\Ot^k$: there is an obvious map $\Ot^k(U) \to \Ot^c(U)$ of norm $\le 1$ and after shrinking $U$ to $V$, the norm estimates for $\Ot^c$ and  $\Ot^k$ differ only by constant factor, proportional to an inverse power of the Huygens distance between $U$ and $V$.  As we shall see, this will imply that if the Hamiltonian normal form iteration is $C^0$-convergent then it will be $C^\infty$-convergent.
 
\subsection{Local equivalence lemma for $\Ot^h$ and $\Ot^c$}
When we compare the Banach spaces $\Ot^h(U)$ and $\Ot^c(U)$, a phenomenon similar to that of the derivative
arises. If $U$ is a compact open set then any continuous function in $U$ is automatically square integrable, thus there is an obvious operator 
$$\rho^{ch}: \Ot^c(U) \to \Ot^h(U),\ f \mapsto f$$ 
But as a square-integrable function is in general not continuous, the inverse map is an unbounded linear operator. The way in which it is unbounded is similar to what we get for an order $d$ partial differential operator with $d=\dim U$. This gives another example of local
equivalence:
\begin{lemma}
 \label{L::local}
 Let $V \subset U \subset \CM^d$ be two sets with positive Huygens distance: $\dt(V,U)=r>0.$ 
Then the restriction mapping
$$\rho^{hc}:\Ot^h(U) \to \Ot^c(V) $$
has norm smaller than $\pi^{-d/2}r^{-d}$
 \end{lemma}
\begin{proof}
Let $f \in \Ot^h(U)$ and $w \in V$. The Taylor expansion of $f$ at a point $w$ 
\[f(z)=\sum_{J \in \NM^d} a_J (z-w)^J,\ a_J \in \CM . \]
The polydisc $D_w$ centred at $w$ with radius $r$ is contained in $U$. We have
$$\int_{D_w} | f|^2 =\sum_{J \in \NM^d} C(J) |a_J |^2 r^{2|J|+2n},\;\;\;C(I)=\prod_{k=1}^d \frac{\pi}{j_k+1} .$$
So we obtain
\[C(0) |a_0|^2 r^{2n} \le \int_{D_w} | f|^2  \leq  \int_{U} | f|^2=| f |^2 . \]
This shows that
\[ |f(w) |=|a_0| \leq   \frac{c}{r^d} | f |,\;\;c:=\sqrt{\frac{1}{C(0)}}=\pi^{-d/2} .\]
As the point $w \in V$ was general the result follows.
\end{proof}
\begin{corollary}
Assume $U,V$ are polydiscs centred at the origin in $\CM^d$.
 The  truncation maps ($j$ may be equal to $\infty$)
 $$\tau_{i,j}^c:\Ot^c(U) \to \Ot^c(V),\ f \mapsto [f]_i^j $$
 satisfy the norm estimates
 $$|\tau_{i,j}^c| \leq \Vol(U)\pi^{-d/2}r^{-d} $$
\end{corollary}
\begin{proof}
 These maps are obtained by composition:
 $$\Ot^c(U)   \stackrel{\rho^{ch}}{\to}\Ot^h(U) \stackrel{\tau_{i,j}^h}{\to} \Ot^h(U) \stackrel{\rho^{hc}}{\to} \Ot(V) $$
 and $\tau_{ij}$ is an orthogonal projection in a Hilbert space.
\end{proof}

\subsection{Hadamard products}
The truncation operator considered above is a special case of a Hadamard product. Consider a
polynomial
$$f(z)=\sum_{a \in \NM^d,\ a \leq N} \a_a z^a $$
and consider the  {\em Hadamard product}
$$g=\sum_{a \in \NM^d} \b_a z^a \mapsto g \star f:=\sum_{a \in \NM^d} \a_a \b_ a z^a  $$
This is the situation we encounter when solving the cohomological equation for the Hamiltonian normal form.
Then, if $U$ is a polydisk centred at the origin, by the Pythagorean theorem for $f \in \Ot^h(U)$ one has
$$|f \star g|^2=\sum_{a \in \NM^d} |\a_a \b_ a|^2 |z^a|^2 \leq \max |\a_a|^2\, |f|^2  $$
So the Hadamard product gives a bounded linear map
$$H^h_f: \Ot^h(U) \to \Ot^h(U),\ g \mapsto f \star g $$ we get an immediate bound for the Hamadard product in $\Ot^h(U)$:
$$| H_f^h| \leq \max |\a_a|. $$
Now, how to get an estimate for $\Ot^c$? This can be done only up to a shrinking. So
we let $U,V$ be a pair of sets for which $\dt(V,U)=r>0$. The Hadamard product
$$H_f^c:\Ot^c(U) \to \Ot^c(V) $$
can be seen as a composition
$$\Ot^c(U) \stackrel{\rho^{ch}}{\to} \Ot^h(U) \stackrel{H_f^h}{\to} \Ot^h(U) \stackrel{\rho^{hc}}{\to} \Ot(V) $$
By local equivalence, we get the estimate 
$$|H^c_f | \leq Cr^{-d}\max |\a_a|$$
where $C=\Vol(U)\pi^{-d/2}$. 

\subsection{Arnold-Moser lemma} \label{L::Arnold_Moser}
The KAM iteration involves  exponents of increasing degrees which we expect to become smaller
as we iterate the  process. The Arnold-Moser lemma puts this heuristic fact in formal form.

So let 
$$\rho(t,s):\Ot^h(D_t) \to \Ot^h(D_s),\ s <t $$
be the restriction mappings.

\begin{lemma}[\cite{Arnold_KAM,Moser_Pisa_1}] 
 Let $f \in \Ot^h(D_t)$ be such that its Taylor series expansions starts at order $N$:
$$f(z):=\sum_{|I| \geq N} a_Iz^I .$$
then:
$$| \rho(t,s)f | \leq \left(\frac{s}{t}\right)^{d+N} | f | .$$
\end{lemma}
\begin{proof}
The monomials $z^I$ form an orthogonal basis of $\Ot^h(P_s)$ with norms
$$| z^I |=C(I)^{1/2}s^{d+|I|},\ C(I):=\frac{\pi^d}{\prod_{k=1}^d(1+i_k)}. $$
By the Pythagorean theorem, for $f \in \Ot^h(U_t)$, we have:
\begin{align*}
| \rho(t,s)f |^2&= \sum_{|I| \geq N} |a_I|^2 C(I) s^{2d+2|I|} \\
             &= \sum_{|I| \geq N} |a_I|^2 C(I) \frac{s^{2d+2|I|}}{t^{2d+2|I|}}t^{2d+2|I|} \\
             & \leq   \frac{s^{2d+2N}}{t^{2d+2N}}|f|^2 .
\end{align*}
\end{proof}
By local equivalence, we get a  bound for $\Ot^c$ of the form:
$$| \rho(t,s)f | \leq \left(\frac{s}{t}\right)^{d+N} \frac{C}{(t-s)^d}| f | .$$
\subsection{The frequency sets $Z_n$}
The above lemmas show that it is crucial to investigate the Huygens distance between the sets involved
in the iteration. We do this for the frequency variables.

\begin{definition}
For a falling sequence $a=(a_n)$ of real positive numbers, we define
$$\CM^d(a)_n:=\{ \a \in \RM^d: \forall J \in \ZM^d,\forall k \leq n,\ \| J \| \leq 2^k,\   (\a,J) \geq a_k \} $$
so that
$$\CM^d(a)=\bigcap_{n \in \NM} \CM(a)_n $$
For a subset  $X \subset \CM^d $, we use the notation 
$$X(a)=\CM^d(a) \cap X,\ X(a)_n=\CM^d(a)_n \cap X.$$
\end{definition}
\vskip0.1cm
\begin{figure}[htb!]
\includegraphics[width=0.3\linewidth]{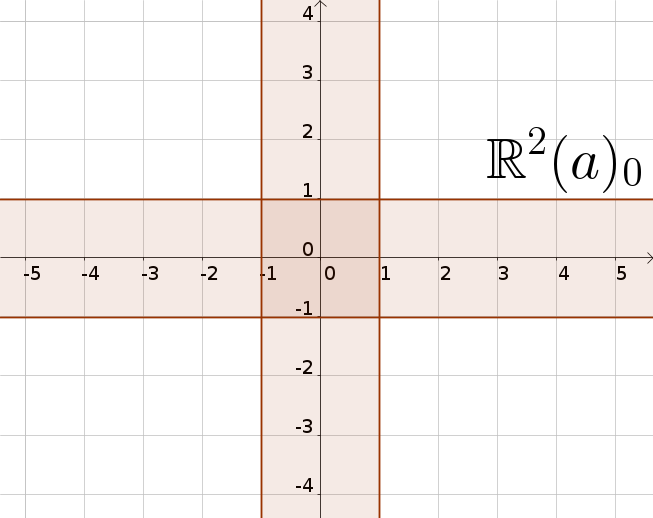} 
\includegraphics[width=0.3\linewidth]{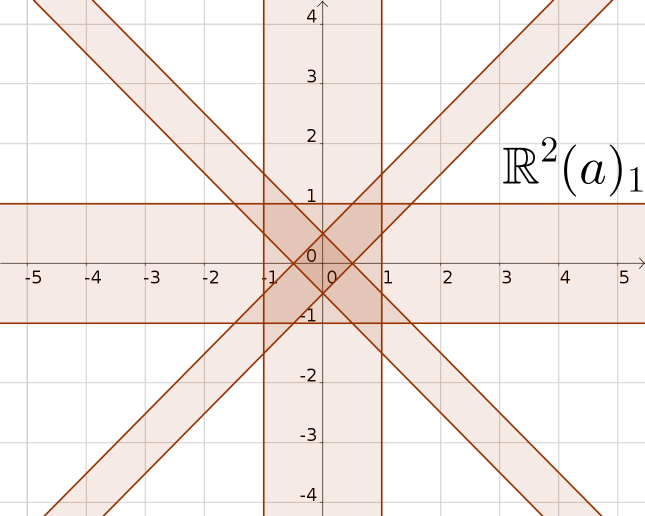} 
\includegraphics[width=0.3\linewidth]{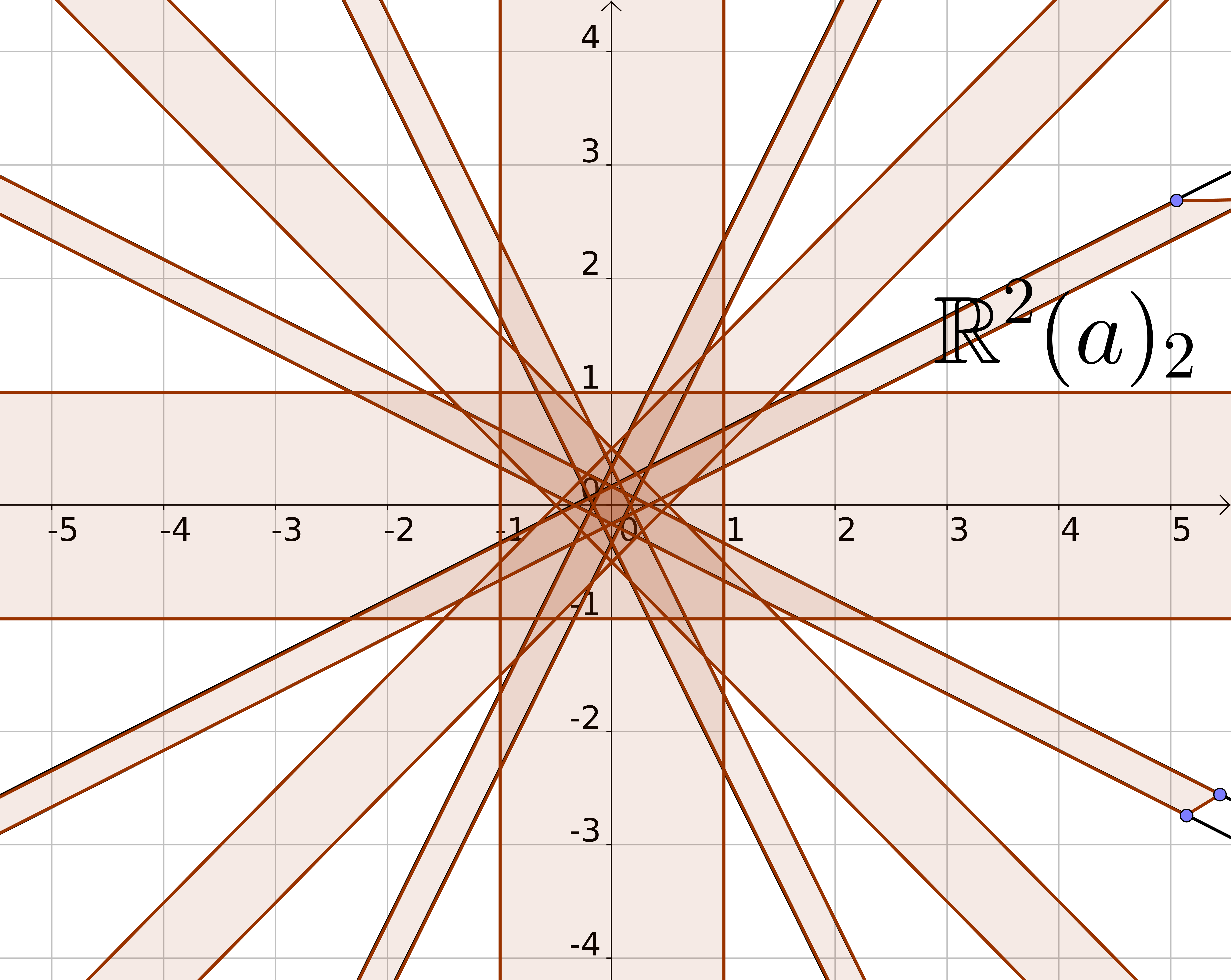} 
\end{figure}
\ \\
As our study is local in the frequency, we need to define a local variant of these sets in neighbourhoods of $\a \in \CM^d$:
\begin{definition} For a fixed decreasing sequence $a$, vector $\a \in \CM^d$ and initial radius $r \in \RM_{>0}$ we define the 
closed set
\[Z_{n,s}:=\{ \omega \in D_s:\forall J \in \ZM^d,\forall k \leq n,\ \| J \| \leq 2^k,\ (J,\a+\omega) \geq a_k(r-s) \}\]
for $s <r$, $n \in \NM$. 
\end{definition}
Many other choices can be made, we only need to have good estimates for the Huygens distance between these different sets.

\begin{lemma} \label{L::calibrated} The Huygens distance $\delta(Z_{n,t},Z_{n,s})$ satisfies the estimate
$$\delta(Z_{n,t},Z_{n,s}) \geq 2^{-n}a_n(t-s).$$
\end{lemma}
\begin{proof}
The proof is straightforward. Assume that
$\omega \in Z_{n,s}$ and take $x \in \CM^d$  satisfying 
$$\| x\|  \leq 2^{-n}a_n(t-s).$$ 
For $k \leq n$ and $\| J \|  \leq 2^k$, we have:
\begin{align*}
|(\a+\omega+x,J)|         & \geq |(\a+\omega,J)|-|(x,J)| \\
                          & \geq a_k(r-s)-\| x \|\, \| J \| \\
                          & \geq a_k(r-s)- a_n(t-s)\\
                          & \geq a_k(r-s)- a_k(t-s)=a_k(r-t) .\\                          
\end{align*}
This shows that $\omega+x \in Z_{n,t}$ and thus proves the lemma.
\end{proof} 
We now define the sets
\[W_{n,s}:=  D_{s}^d  \times Z_{n,s}   \times D_{s}^{2d} \subset  \CM^d \times \CM^d \times \CM^{2d}\]
with coordinates $\tau_i,\omega_i,q_i,p_i$.
From the above distance estimate we deduce the 
\begin{proposition}
\label{P::Hadamard}
 The Hadamard product 
 $$H_f^c(n,t,s):\Ot^c(W_{n,t}) \to \Ot^c(W_{n,s}) $$
 with the rational function
 $$f=\sum_{2^{n} \leq |I+J| < 2^{n+1}} \frac{1}{(\a+\omega,I-J)}q^Ip^J $$
 satisfies the norm estimate
 $$|H^c_f(n,t,s)| \leq \Vol(W_{n,t}) \left( \frac{s}{t}\right)^{d+2^n}a_{n+1}^{-d-1}2^{-d}\pi^{-d/2}(t-s)^{-d}$$
\end{proposition}
\begin{proof}
 By definition of $W_n$, we have
 $$\left| \frac{1}{(\a+\omega,I-J)} \right| \leq a_{n+1}^{-1} $$
 therefore, by local equivalence of $\Ot^c$ and $\Ot^h$, we have the norm estimate
 $$|H_f^c| \leq  \Vol(W_{n,t}) \left( \frac{s}{t}\right)^{d+2^n}a_{n+1}^{-d-1}2^{-(n+1)d}\pi^{-d/2}(t-s)^{-d} .$$
\end{proof}
As we go on, we see that more and more constants are involved in our
estimates but these are always of the same form.

In our iteration process, we will have to choose a strictly decreasing sequence $s_0 >s_1 > s_2,\ldots$ and put
\[ Z_n :=Z_{n,s_n},\;\;W_n:=W_{n,s_n}\]
As the sequence $(s_n)$ is decreasing we have
$$  a_k(s_0-s_{n+1}) \geq  a_k(s_0-s_n),$$
so the sets $(Z_n)$ form a descending chain and can be considered as a 
local variant of the chain $(\CM^d(a)_n)$.

Denote by $B(s_n)$ the closed ball of radius $s_n$ centred at the origin.
First note that $B(s_0) \subset Z_0$. Moreover as the sequence $(s_n)$ is decreasing we have
$$  a_k(s_0-s_{n+1}) \geq  a_k(s_0-s_n),$$
so the sets $(Z_n)$ form a descending chain and can be considered as a 
local variant of the chain $(\CM^d(a)_n)$. Moreover if $s_0<1$ then $a_k> a_k(s_0-s_n)$ for all $n$ and
therefore the condition
$$ (J,\a+\omega) \geq a_k$$
defining $\CM^d(a)$ is stronger than the condition
$$ (J,\a+\omega) \geq a_k(s_n-s_{n+1})$$
defining $Z_n$ which implies that
$$B(s_n) \cap \CM^d(a)_n \subset Z_n .$$
\subsection{Bruno sequences and the absorption lemma.}
Our analysis of the convergence of the Hamiltonian normal form iteration
will depend on the choice of two falling sequences of positive real numbers.
The first sequence $a=(a_0,a_1,a_2,\ldots)$ controls the distance to the
resonance hyperplane in the frequency space and enters in the definition
of the set $\CM^d(a)$. A second sequence $s=(s_0,s_1,s_2,\ldots)$ controls the size of the shrinking sequence of
balls that enters in the definition of the frequency sets $Z_{n}$.

Many different choices are possible for these two sequences, but turns
out that particular convenient choices can be made based on the notion of
{\em Bruno sequence}. Recall that a sequence $a=(a_n)$ of positive real
numbers is called a {Bruno sequence}, if the transformed sequence with terms
\[t_n := a_0 a_1^{1/2} a_2^{1/2^2}\ldots a_n^{1/2^n}\] 
converges to a positive real number
\[\beta(a):=\prod_{k=0}^{\infty} a_k^{1/2^k}=\lim_{n \to \infty} t_n\]
that we will call the {\em Bruno constant} of the sequence $a$.
Equivalently, $a$ is a Bruno sequence if and only if
\[ \sum_{k=0}^\infty \frac{\|\log a_k\|}{2^k} < +\infty .\]
Obviously, the set of Bruno sequences is closed under termwise sum and product of sequences and the formation of the Bruno constant is multiplicative:
\[  \beta(ab)=\beta(a) \beta(b) .\]
For the geometrical sequence $G:=(Cq^n)$ one finds
\[\beta(G)=C^2q^2.\]
The exponential sequence $E:=(e^{\a^n})$ is Bruno for $\|\a\| <2$ and
\[ \beta(E)=e^{\frac{1}{1-\a/2}}\]
  
For the definition of the sets $Z_{n}=Z_{n,s_n}$ we will choose the sequence
$s=(s_n)$ in terms of an auxiliary sequence $\rho=(\rho_n)$
$$s_{n+1}=\rho_n^{1/2^n}s_n,\ s_0=r .$$
If $\rho$ is a Bruno sequence, then  the infinite product
\[ \prod_{k=0}^{\infty} \rho_k^{1/2^k}\] 
converges to a strictly positive number and therefore the sequence
$s=(s_n)$ converges to a positive value $s_\infty =\beta(\rho) s_0$ as well.
 
As we will both consider such sequences and their inverses, we denote respectively by $\Bt^+$ and $\Bt^{-}$ the set of increasing and decreasing Bruno sequences. Note that the set of Bruno sequences some obvious multiplicative properties:
\begin{enumerate}[{\rm i)}]
\item Taking the multiplicative inverse interchanges $\Bt^+$ and $\Bt^-$,
\item The product of two elements in $\Bt^\pm$ is again in
$\Bt^\pm$.  
\item An element of $\Bt^{\pm}$ raised to a positive power remains in $\Bt^\pm$.
\end{enumerate}

We will always take  $a=(a_n)$ and $\rho=(\rho_n)$ to be in  $\Bt^-$, but we
also need to take the sequence $\rho$ small enough with respect to $a$ in order
to counteract certain small denominators of the form
$$ a_n^k(s_n-s_{n+1})^{m}=a_n^k(1-\rho^{1/2^n})^{m}s_n^{m}$$
that appear in norm estimates. The following lemma shows that an appropriate
choice of $\rho$ 'absorbs' these small denominators.

\begin{lemma}\label{L::absorb}  Let $\rho \in \Bt^-$ with $\rho<1/2$ and put
$$s_{n}:=s_0\prod_{k < n} \rho_k^{1/2^k},\;\;\;s_\infty=\beta(\rho)s_0 .$$
Then we have
$$s_n-s_{n+1} \geq 2^{(-n-1)}s_\infty .$$

Moreover for any $k > 0$ and any Bruno sequence $b \in \Bt^-$ we may find a Bruno sequence $\rho \in \Bt^-$ such
\[ \frac{\rho_n^k}{(s_n-s_{n+1})^m}<b_n\]
\end{lemma}
\begin{proof}
The inequality $(1-\frac{x}{m})^m \ge (1-x), \;(m \ge 1)$ implies that
$(1-\frac{x}{2^n}) > (1-x)^{1/2^n},$
which for $x=1/2$ transposes to
\[ 1-2^{-1/2^n} \ge \frac{1}{2^{n+1}}.\]
The falling sequence $\rho$ is 
bounded by $1/2$ and hence:
\[1-\rho_n^{1/2^n} \ge 1-2^{-1/2^n} \ge \frac{1}{2^{n+1}} \]
and therefore
$$s_n-s_{n+1} > 2^{-(n+1)}s_\infty .$$
This proves the first part of the lemma.  For the second part, note that
the first part implies
$$ \frac{\rho_n^k}{(s_n-s_{n+1})^m} \leq  2^{m(n+1)}\rho_n^ks_\infty^{-m}  .$$
Let us now analyse how this estimate behaves when $\rho$ is multiplied
by a positive geometric sequence $u:=(Cq^n)$ with $q<1$ and $C>0$.
For the Bruno sequence $\s=u \rho$ the corresponding transformed falling
sequence:
$$t_n=s_0\prod_{k < n} \s_k^{1/2^k}=s_n\prod_{k < n} C^{1/2^k}q^{1/2^k}.$$ 
Hence we have:
$$t_\infty=\beta(u \rho) s_0 =\beta(u)\beta(\rho)s_0=C^2q^2 s_\infty, $$
and therefore
$$ \frac{\s_n^k}{(t_n-t_{n+1})^m} \leq  2^{m(n+1)}\s_n^kt_\infty^{-m}= 2^{m(n+1)}\rho_n^ks_\infty^{-m}C^{k-2m}q^{kn-2m}  .$$
For $q<2^{m/k}$, the sequence with terms
$$ 2^{m(n+1)}s_\infty^{-m}q^{kn-2m}$$
converges to zero and is therefore bounded by a constant.
Thus by an appropriate choice of $C$, we can arrange to have
\[ 2^{m(n+1)}s_\infty^{-m} q^{kn-2m} C^{k-2m} \rho_n^k \le \rho_n^k .\]
Hence, if $b \in \Bt^-$ is any given falling Bruno sequence, we first
form $\rho:=b^{1/k}$ and then 
$$ \frac{\s_n^k}{(t_n-t_{n+1})^m} \leq \rho_n^k=b_n$$
Thus the logic of the proof is as follows: given a Bruno sequence $b \in \Bt^-$,
we multiply the Bruno sequence $b^{1/k} \in \Bt^-$ by a sufficiently small geometric sequence as above. This defines a new Bruno sequence $\sigma$, which satisfies the second estimate of the lemma.
\end{proof}

\section{Functional calculus in Kolmogorov spaces}
\label{KolmogorovSpaces}
We arrive at the second part of the proof of convergence. We give
a very quick review the theory of Banach functors developed in~\cite{KAM_theory_I,KAM_theory_II,KAM_theory_III,Functors} to which we refer for more details.

The idea behind Banach functors is to give an appropriate language to deal with the many indices involved in KAM theory. In the classical context, these indices should be accordingly tuned and not only the size of a ball and the index of the iteration.  Consider for instance the case of the map $L$, involved in the HNF iteration, which solves the homological equation. It sends functions to linear maps and therefore defines a family of linear maps:
$$L^c(n,u,t,s):\Ot^c(W_{n,u}) \to \Hom(\Ot^c(W_{n,t}),\Ot^c(W_{n,s})),\ m \mapsto Lm$$
So three parameters $u,t,s$ are involved in something that should be considered
as a single mapping. When we compose such mappings, more and more
parameters will appear. A common way to fix this problem is to try at each step to eliminate these new parameters by a specific {\em choice}. This requires
great ability as one needs to foresee which choices are convenient as the
precise choices will affect the estimates in the next step.

We will adopt a different strategy by regarding the parameters as our central
object of study and fix them only at the end of the argument.  In the spirit
of category theory, we give a particular emphasis on morphisms and the way they behave with respect to parameters, due to
the covariant or contravariant nature of functors involved, like the functor
$\Hom(-,-)$.

Our considerations might seem pedantic and unnecessary at first glance. An old mathematician mastering the debauch of indices in tensor calculus may have rejected the subtleties of the definition of manifolds and fibre bundles as well. (As a concrete example, looking back in the twenties at a basic
memoir on tensor calculus, we find that the definition of a manifold is given
in the introduction within a few lines and, of course, no definition at all is
given of a fibre bundle~\cite{Lagrange_tenseur}).
\subsection{Relative objects}
Recall that a category $B$ is called a {\em small category} if its objects and morphisms form a set.
The small categories we will consider are ordered sets $(B,<)$: it has the elements of $B$ as objects, and a single morphism from $a$ to $b$ for all pairs that satisfy $a  \leq b$. The categorical point of view might seem pedantic as we consider only ordered sets, but the reader will rapidly notice that it turns out to be the right geometric approach. It has also the advantages to limitate the range of possibilities for the  constructions assuming these should be functorial.

Let $B$ be a small category and $\Ct$ an arbitrary category. A {\em $B$-object} in $\Ct$
is a covariant functor
$$F:B \to \Ct $$
If $\Ct$ is the category of sets, we speak of a $B$-set or a set over $B$ and similarly for vector spaces, Banach spaces and so on.

When $B$ is an ordered set, the functor property means that for $a <b$, we have maps:
\[ F(a <b) =\rho_{a,b}: X_a \lra X_b,\ F(a) =X_a,\ F(b)=X_b .\]
to which we refer as the {\em connecting morphisms} of
the functor.

As an example, an increasing  chain of sets $X_1 \subset X_2 \subset X_3 \subset \ldots$ may be considered as a relative set over $(\NM, < )$ and a decreasing set as a relative set over the opposite category
$$(\NM,<)^{op}=(\NM,>) $$
\subsection{Geometric picture}
Given a (covariant) functor
$$F:B \to \Ct $$
where $B$ is an ordered set (seen as a small category), we may form its {\em associated bundle}
$$\xi:X \to B $$
The {\em total space $X$} of the functor is by definition the disjoint union
$$X=\bigsqcup_{b \in B}X_b,\ X_b:=F(b)  $$
We regard an element of $X$ as a pair $(b,p)$ with $p \in X_b$ and
the projection is
$$X \to B,\ (b,p) \mapsto b $$

So we get a geometric picture close to that of a fibre bundle. Quite often, we will
use the notation
$$X \to B $$
to replace $F:B \to \Ct$.

The connecting morphisms of
the functor act like a
connection on a fibre bundle. Therefore we say that a map
$$\sigma: A \to X $$ 
is a {\em horizontal section over $A \subset B$},
if it is compatible with the connecting maps. We use the notation $\s \in \G^h(A,X)$.

Any element $x \in X_b$ defines a unique
horizontal and bounded section 
$$\s_x:]-\infty,b]\to X,\ a \mapsto \rho_{ab}(x)$$ 
Here 
$$]-\infty,b]:=\{ a \in B: a \leq b \}.$$
\begin{example}
The functor:
$$F:(\RM_{>0},<) \to  {\text{\bf Sets}}, t \mapsto  D_t $$
will be called the {\em relative polydisk}.
The associated bundle 
$$D \to \RM_{>0}$$ has fibres $D_t$ (polydisk of of radius $t$ in $\CM^d$). For $s < t$, the connecting
morphism is just the inclusion map $D_s \to D_t$.
\end{example}
\subsection{Kolmogorov spaces}
Let us now consider relative Banach spaces:
$$F:B \to {\bf Ban} $$
The associated bundle
$$E \to B$$ is a family of Banach spaces $(E_b)$ parametrised by some base $B$. In particular, we assume a specific norm
$|-|_b$ on $E_b$ is given, elements of $E$ are pairs
$x=(b,v)$ with $v \in E_b$. We often use the notation $|x|$ instead of $|v|_b$.

The set $\Gamma^h(A,E)$ of horizontal sections has the structure of a vector space. It is a simple
but fundamental fact that the vector space of horizontal and {\em bounded}
sections $\Gamma^{\infty}(A,E)$ has the natural Banach space structure, with norms
\[\sup_{b \in A} \|\sigma(b)\|\]

\begin{definition}
 
A relative Banach space $E \to  B$ is called a {\em Kolmogorov space} if the connecting morphisms
$$ E_a \to E_b$$
have norms $\leq 1$.  
\end{definition}
\begin{example}
To describe some typical examples, consider again the {\em relative polydisk}
$D \to B:=(\RM_{>0},>)$.
We have functors
$$F^\bullet:B \to \text{\bf Ban},\ t \mapsto \Ot^\bullet(D_t) $$
Here  $\Ot^\bullet$ stand for any of the completions $\Ot^c,\Ot^k,\Ot^h$
considered in section 4. 

As under restriction of functions the norm can only decrease, the Banach
completions $\Ot^\bullet(D_t)$ are the fibres of a Kolmogorov space
$$\Ot^\bullet(D) \to \RM_{>0}. $$
  
\end{example}
\begin{example}
The (real) one dimensional Kolmogorov spaces give an important class of examples. 
Consider a relative Banach space
$$\RM_{>0} \times \RM \to \RM_{>0},\ (t,v) \mapsto v $$
over $(\RM_{>0},>)$. The connecting morphisms
$$\RM \to \RM,\ v \mapsto \l(t,s) v ,\ t>s$$
identify the fibres above $t$ and $s$.
These maps should satisfy the relation:
$$\l(u,t)\l(t,s)=\l(u,s). $$
for instance of the form:
$$\l(t,s)=e^{g(s)-g(t)} .$$
These define a Kolmogorov space when $\l(t,s) \leq 1$.
\end{example}

\subsection{Hom spaces}
We now define the Hom spaces of our functor categories.
Let $E \to B,\ F \to C$ be Kolmogorov spaces. Then we define a new Kolmogorov space
$$Hom(E,F) \to B^{op} \times C$$ 
whose fibres above $(b,c)$ consists of bounded linear mappings from $E_b$ to $F_c$. Restrictions mappings are induced by that of $E$ and $F$, but as the functor $Hom(-,F)$ is contravariant, we must reverse
the order on the first factor (this is what {\em  op} stands for). 
\begin{example}
Take  $E=F=\Ot^c(D)$, $B=C=\RM_{>0}$ and consider the linear maps
$$L_{t,s}:\Ot^c(D_t) \to \Ot^c(D_s),\ f(z) \mapsto f(2z) $$
They define a horizontal section $L \in \G^h(U, \Hom(\Ot^c(D),\Ot^c(D)))$ over the set 
$$U=\{ (t,s) \in \RM_{>0}^2: 2s<t\} .$$

\vskip0.1cm
\begin{figure}[htb!]
\includegraphics[width=0.7\linewidth]{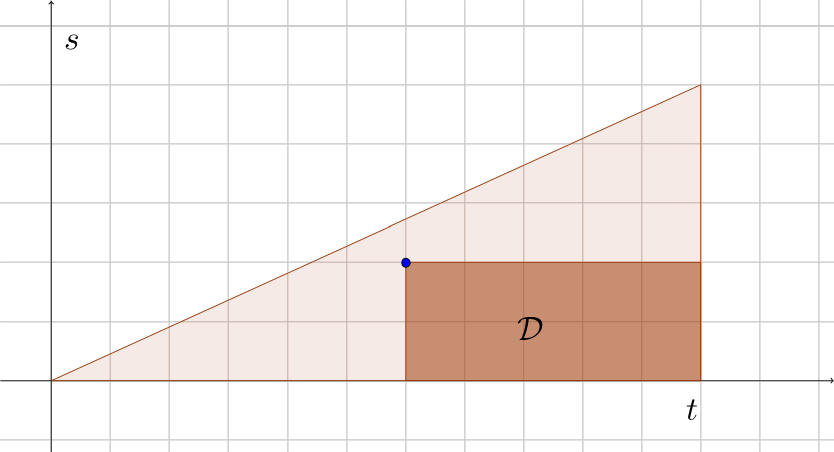} 
\end{figure}
\ \\
This means simply that if $f$ is holomorphic inside $D_t$ with continuous extension to the boundary the same is true for $L(f)$ is holomorphic inside $D_{s}$ for any $s \leq t/2$. Of course if  we consider any $u \geq t$, and any $s \leq t/2$ we have also a map for 
$$\Ot^c(D_u) \to \Ot^c(D_s),\ f(z) \mapsto f(z/2)$$
These are the connecting morphisms of the Kolmogorov space $\Hom(\Ot^c(D),\Ot^c(D)))$:
As in any Kolmogorov space, an element defines a horizontal section over a domain $\Dt$. In our case, this domain looks like a rectangle.

\end{example}

\subsection{Arnold spaces}
In analysis of the HNF iteration, we encounter a more involved situation:
we first defined sets
\[W_{n,s}:=  D_{s}^d \times Z_{n,s}  \times  D_{s}^{2d} \subset  \CM^d \times \CM^d \times \CM^{2d} .\]
For a fixed $n$ one obtains as above Kolmogorov spaces $\Ot^{\bullet}(W_n)$
over $\RM_{>0}$ with  $\Ot^{\bullet}(W_{n,s})$ as fibre over $s$.
So we obtain a sequence of Kolmogorov spaces:
$$\Ot^\bullet(W):\cdots \to \Ot^\bullet(W_{n-1}) \to \Ot^\bullet(W_n) \to \Ot^\bullet(W_{n+1}) \to \cdots $$
For fixed $n$, the sets $W_{n,s}$ define a set $W_n \to B_n:=\RM_{>0}$ and if we let $n$ vary, we have  a set
$$W \to \NM \times \RM_{>0} $$
where $\NM$ stands for the ordered set $(\NM,<)$  seen as a category. 
So we also have a Kolmogorov space
$$\Ot^{\bullet} (W) \to \NM \times \RM_{>0} $$
Note that in the iteration the role played by the indices $n$ and $s$ is
not symmetric: the normal form is horizontal with respect to the real variables and changes with $n$.
\begin{definition}
A functor 
$$F:(\NM,<) \to {\bf Kol} $$
is called an {\em Arnold space} if the  connecting morphisms have norm $\leq 1$.
\end{definition}
An Arnold space defines a sequence of Kolmogorov spaces
 $$\xymatrix@C=1cm@R=1cm{ E:\cdots \ar[r] & E_{n-1} \ar[r] \ar[d] & E_n \ar[r] \ar[d] &
E_{n+1} \ar[r] \ar[d] &\cdots \\
B:\cdots \ar[r] & B_{n-1} \ar[r] & B_n \ar[r] &
B_{n+1} \ar[r] &\cdots \\} $$
with $F(n)=(E_n \to B_n)$. In an Arnold space, the norms of the connecting maps
$$E_{n,t} \to E_{n+1,s} $$
are linear continuous maps with norm $\leq 1$. In practice the basis $B_n$ are all equal.

\subsection{Kolmogorofication}
Any Banach space 
$$E \to B$$
over an ordered base $B$ defines a Kolmogorov space $K(E)$ with fibres
$$K(E)_b=\{ x \in E_b: \sup_{a \leq b}|\s_x(a)| < +\infty \} $$
which we call the {\em Kolmogorofication} of $E$ (in the old notations $|\s_x(a)|=|x|_a$). Note that the Banach space $K(E)_b$ is just the space of horizontal bounded sections of $ E \to B$ over the set $]-\infty,b]$.

Given a  Kolmogorov space, we may rescale the norms by any positive function
$$\l:B \to \RM_{>0} $$
called a {\em weight function.} In general this family of rescaled Banach spaces
is no longer a Kolmogorov space, but one may take its Kolmogorification.
The resulting Kolmogorov space has fibres $E[\l]_t \subset E_t$ with norms
$$|f|_\l=\sup_{s \leq t} \l(s)|\s_f(s)| $$
where $\s_f$ is the horizontal section associated to $f$ .

\begin{example}
Consider a function
$$\l:\RM_{>0} \to \RM_{>0} $$
and the trivial vector bundle
 $$E:=\RM_{>0} \times \RM \to \RM_{>0},\ (t,v) \mapsto t$$
 with norms $|(t,v)|=\l(t)|v|$ (absolute value). If $\l$ is an increasing
 function, then the rescaled family is a Kolmogorov space. If $\l$ is not
 increasing, we consider the Kolmogorification $K(E)$.  If the function $\l$
 is  unbounded then $K(E)$ has trivial fibres (reduced to $\{ 0 \}$), otherwise $K(E)$
 is isomormophic as a vector bundle to $E$ but the Riemannian structure is now given by 
 $$|(t,v)|=\sup_{s \leq t} \l(s)|v| $$
 where $|v|$ denotes the absolute value of $v \in \RM$.
 \end{example}
\subsection{Maximal space}
The notion of {\em maximal ideal} in the local ring $\Ot_{\CM^d,0}$ can be generalised to the context of Kolmogorov space $E \to \RM_{>0} $. Consider the weight function
$$\l: \RM_{>0} \to \RM_{>0},\ s \mapsto s^{-1}$$
The Kolmogorov space $\Mt(E):=E[\l]$ will be called the {\em maximal space} of
$E$.
For the completions $\Ot^c$ and $\Ot^h$, the corresponding spaces, will be simply denoted by $\Mt^\bullet$
instead of $\Mt(\Ot^\bullet(-))$ (but not for $\Ot^k$ as there might be a confusion with the power of the maximal ideal).

\begin{example}
Let us now explain the relation of the maximal space with the maximal ideal of the ring $\Ot_{\CM^d,0}$. Consider again the {\em relative polydisk}
$D \to \RM_{>0} $
with 
$$D_s=\{ z \in \CM^d: |z_1|  \leq s,\dots,|z_d| \leq s \} $$
as fibre over $s$. There is a forgetful map which associates to a holomorphic
function its germ at the origin
$$\Mt^c(D) \to \Ot_{\CM^d,0}$$
We assert that the image of this map is the maximal ideal  $\Mt \subset \Ot_{\CM^d,0}$. Indeed, a holomorphic germ $f \in \Mt$ is by definition of the form
$$f=\sum_{i=1}^d \a_i z_i, $$
where the $\a_i$ are holomorphic in some  common neighbourhood $U$ containing a polydisk $D_t$.
The germ $f$ defines a horizontal section $\s_f$ over the interval $]0,t]$. The value $\s_f(s) \in \Mt^c(D_s)$ is simply the restriction of 
$$f:U \to \CM $$
to $D_s$ where $s \leq t$.
We get the inequality
$$\sup_{z \in D_s}|f(z)| \leq s\sum_{i=1}^d \sup_{z \in D_s}|\a_i|, $$
and therefore the renormalised norm of the section $\s_f$ satisfies
$$|\s_f| \leq \sum_{i=1}^d |\a_i| .$$
This shows that the representative $\s_f(t)$ of the germ $f$ belongs to the Banach space $ \Mt^c(D)_t$.
Conversely assume  that $f \in \Mt^c(D)_t$ and denote by $\s_f$ its associated horizontal section. 
We have
$$|\s_f(s)|\leq s|f|   $$
and taking the limit when $s \to 0$, we deduce that $\s_f(0)=f(0)=0$.
\end{example}

\subsection{Local operators}

\begin{example}
Consider the derivative 
$$D:f \mapsto f' $$
as the prototype  of a partial differential operator. It exhibits an interesting property:
it maps $\Ot^c(D_t)$ to $\Ot^c(D_s)$ for any $s<t$. Indeed if $f$ is holomorphic inside $D_t$ then $f'$ is holomorphic and therefore continuous inside the closure of $D_s$.

This means that
unlike general homomorphisms, the derivative defines a section above $B \times B^{op}$ with $B=(\RM_{>0},
>)$.

Geometrically, we have an horizontal section defined over an upper triangle
\vskip0.1cm
\begin{figure}[htb!]
\includegraphics[width=0.5\linewidth]{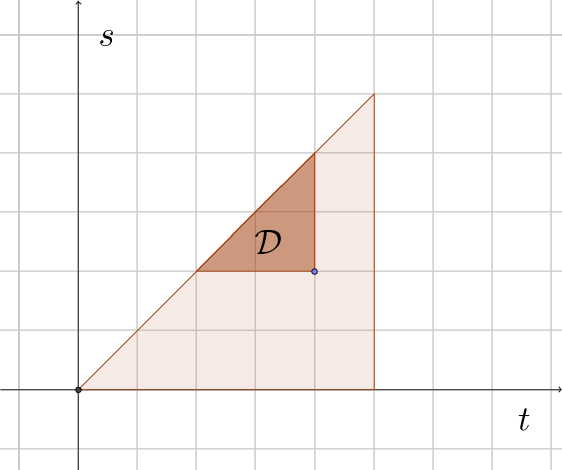} 
\end{figure}
\end{example}
So differential operators are to be considered as sections of the opposite Kolmogorov space $Hom(E,F)^{op}$.
We now consider an arbitrary function
$$\l:B \times C \to \RM_{>0}$$
\begin{definition} Let $E \to B$ and $F \to C$ be Kolmogorov spaces.
The Kolmogorofication of the rescaled space $Hom(E,F)^{op}[\l]$ is called the {\em space of $\l$-local operator} and we denote it by $L^\l(E,F)$.
\end{definition}
A $\l$-local map $u \in L^\l(E,F)_{a,b},\ a > b$ is a horizontal family of maps
$$u_{a',b'} \in L(E_{a'},F_{b'}),\ a' \leq a,\ b' \geq b, $$
with finite norm
$$|u|=\sup_{a' \leq a,b' \geq b}\{ \l(a',b') |u_{a',b'}| \}. $$
\begin{example}
  Consider again the relative polydisk in $\CM^d$
  $$D \to \RM_{>0}. $$
 and the associated Kolmogorov space 
 $$\Ot^c(D) \to \RM_{>0}.$$
 The Cauchy-Nagumo inequalities \ref{L::Cauchy_Nagumo} imply that any order $k$ partial differential
 operator is local for the weights $\l(t,s)=(t-s)^a,\ a \geq k$. 
\end{example}

\begin{example}
 Let us keep the same notation, fix $d=1$ and consider the case of multiplication by $z$
 $$\mu(t,s):\Ot^c(D)_t \to \Ot^c(D)_s,\ f(z) \mapsto zf(z) $$
 As a Banach space mapping its norm is equal to $s$:
 $\| \mu(t,s) \|=s$.
 However if we look at it as a local operator for the weight $\l(t,s)=1$ then 
 $|\mu(t,s)|=t $. Indeed, for $t' \leq t, s'\geq s$, we have
 $$|\mu(t',s')(1)|=|z|=s' $$
 where $1$ stands for the constant function equal to one (to be completly formal we should write $1_{t'}$ as it is defined on the disk of radius $t'$).
 Therefore
 $$|\mu(t,s)|=\sup_{t' \leq t, s'\geq s} |\mu(t',s')| \geq t $$
 as $s'$ can be arbitrary close to $t$. Now obviously
 $$ |\mu(t',s')(f)| \leq s'|f| \leq t|f| $$
 and therefore gives a different norm namely $|\mu(t,s)|=t$. Even in such an example, the naïve operator norm and the one we find are in fact different!

 The difference between the naïve operator norm $\| \cdot \|$ and the Kolmogorov space norm $|\cdot |$ is somehow  similar to the difference between pointwise and uniform convergence: in the later case we consider a sup-norm over a whole triangle instead of the norm at a point.
\end{example}
\begin{example}
 Let us again consider the one dimensional case: 
 $$E:=\RM_{>0} \times \RM \to \RM_{>0},\ (t,v) \mapsto t$$
 with constant connecting maps ($(t,v)$ gets identified with $(s,v)$).
 
 A local operator $u \in L^\l(E,E)_{t,s} $ for a weight $\l(t,s)$  consists of family of linear maps
 $$u(a,b):\RM \to \RM,\ x \mapsto \a(a,b) x ,\ a \leq t,\ b \geq s$$
 such that $\a(a,b)\l(a,b)$ remains finite. Consider for instance the case $\l(t,s)=t-s$ then the operator
 $$u_{t,s}:E_t \to E_s,\ x \mapsto \frac{1}{t-s}x $$
 define a local operator with norm $1$. This finite dimensional model mimics the derivative acting on holomorphic functions.
\end{example}

In the HNF iteration we are concerned with powers of the maximal ideal.
Therefore it is important to generalise the small denominators lemma for
the maximal space. This can be done in one stroke for all the different cases: 

\begin{proposition}
\label{P::maximal}
Let $E,F$ be Kolmogorov spaces over $\RM_{>0}$.
 Any element $u \in L^\l(E,F)$ such that $\Im u \subset \Mt(F)$ induces an element of $ L^\mu(\Mt(E),\Mt(F)) $
 with $\mu(t,s)=(s/t)\l(t,s)$. Moreover the resulting map
 $$L^\l(E,M(F)) \to L^\mu(M(E),M(F)) $$
 has norm $\leq 1$ (over each point).
\end{proposition}
\begin{proof}
 Take $u \in L^\l(E,F)_{t,s}$ satisfying the assumption. We denote by $| \cdot|$ the norms in $E,F$ and by $\| \cdot \|$ the norms in $\Mt(E)$ and $\Mt(F)$. We
  natural inclusions
 $$\rho:\Mt(E) \to E,\ \s:\Mt(F) \to F $$
 and
 \begin{align*}  
 s^{-1}|u(\rho(x))| &\leq s^{-1}\l(s,t)^{-1}|u|\, |\rho(x)| \\
    & \leq ts^{-1}\l(s,t)^{-1}|u|\, \|x\|=\mu(s,t)^{-1} |u|\, \|x\|
 \end{align*}
Therefore
 $$\| u(x)\| \leq \mu(s,t)^{-1} |u|\, \|x\| $$
\end{proof}

\subsection{Composition lemma}
Let $E,F$ be a Kolmogorov space over $\RM_{>0}$.
In practice we often deal with local operators $u \in L^\l(E,F)$ having a weight of the form
$$\l(t,s)=Ct^{-a}(t-s)^b \left(\frac{t}{s} \right)^{n} $$
with $C>0,\ a>0,\ b \geq 0,\ n \in \NM$. In such case, we write provisionally $u \in[C,a,b,n]$ 
\begin{lemma}
\label{L::composition}
 If $u \in L^\l(E,F)$, $v \in L^\mu(F,G)$ are local maps with
 $u \in [C,a,b,n]$ and $v \in [C',a',b',n]$
 then $v \circ u \in [2^{-b-b'}CC',a+a',b+b'] $ and
 $$|v \circ u | \leq|v|\,|u| $$
\end{lemma}
\begin{proof}
Let $m$ be the mid-point of the interval $[s,t]$ then for  $x \in E_t$, we have
\begin{align*}
|v \circ u(x) | &\leq  Cm^{a}(m-s)^{-b} \left(\frac{s}{m} \right)^{n}|u|\, |v(x)|\\
                & \leq  CC'm^{a}(m-s)^{-b}t^{a'}(t-m)^{-b'} \left(\frac{s}{t}\right)^{n}|u|\, |v|\,|x|\\
                & \leq 2^{b+b'}CC't^{a+a'}(t-s)^{-b-b'}\left(\frac{s}{t} \right)^n|u|\, |v|\,|x|\\
\end{align*}
where we identified $x$ with its associated horizontal section.
\end{proof}

\subsection{The map $L_n$}
The main ingredient to solve the homological equation was the  map $L$ (see \ref{SS::Homequation})
defined on monomials by setting for $a \neq b$:
\[ Lp^aq^b :=\left\{ \begin{matrix}\displaystyle{ \{-,\frac{1}{(\alpha+\omega,a-b)}p^aq^b\}},\ a \neq b\\
\ \\
\displaystyle{\sum_{i=1}^d \frac{\partial \tau^a}{\partial \tau_i}\d_{\omega_i}},\ a=b
\end{matrix}
    \right.\]
 
We  now lift this map to the level of Kolmogorov spaces and check that it is local.
First we consider the Kolmogorov space 
$$\Ot^c(W_n) \to ]0,1]$$
with fibres $\Ot^c(W_{n,s})$ and its maximal spaces $\Mt^c(W_n)$. Then we
define the {\em Kolmogorov space of Poisson derivations:} 
$$\Theta^c(W_n) := \Theta(R) \cap L^\l(\Ot^c(W_n),\Ot^c(W_n)), \l(t,s)=e^{-1}(t-s),$$
the constant $e=2.718\dots$ is purely conventional and will help to avoid further constants in the computations. Here $\Theta(R)$
stands for the algebra of Poisson derivations of the Poisson algebra
$$R=SD_{\alpha}[[\tau,p,q]] \subset \CM[[\omega,\tau, p,q]].$$

Now consider the projection:
$$\tau_n^\bullet: \Ot^\bullet(W_n) \to \Ot^\bullet(W_n),\ f \mapsto [f]_{2^{n+1}+2}^{2^{n+1}+2} $$
The map $\tau_n^h$ is an ordinary orthogonal projection and has therefore norm $1$.
We can be more precise: by the Arnold-Moser lemma the operator $\tau_n^h$ is local for the weight
$(s/t)^{k}$, provided $k \leq 2^{n+1}+2+3d $, and for simplicity we take $k=2^n+1$.

Lemma~\ref{L::calibrated} implies that
$$\dt(W_{n,t},W_{n,s}) \geq 2^{-n}a_n(t-s) $$
and therefore local equivalence implies that the map $\tau_n^c$ is  local for the weight 
$$\l(s,t)=\left(\frac{s}{t} \right)^{2^{n}+1}(t-s)^{k}$$
for $k$ big enough (and more precisely $k \geq 4d$ see Lemma~\ref{L::local}) with norm bounded
by $a_n^{-k}2^{kn}$. This is a very rough estimate for the norm of course but sufficient: assuming $(a_n)$ to be a falling Bruno sequence then the sequence $(a_n^{-k}2^{kn})$ is a rising Bruno sequence.

Composing $L$ with $\tau_n^\bullet$ gives a map $L_n^\bullet$.
$$L_n^\bullet:\Ot^\bullet(W_n) \to \Theta^\bullet(W_n) $$
It is local provided $k$ is large enough. To see it, we decompose the map $L_n^\bullet$ and at each step we find operators which are local for a weight of the form $(t-s)^k$ ($k$ big enough) and norms bounded by a Bruno sequence. For instance, the next step  involves Hadamard products with the series
$$f_n=\sum_{|a|+|b| \in [2^{n+1}+2,2^{n+2}+2[} \frac{1}{(\alpha+\omega,a-b)}p^aq^b$$
So we get maps
$$H_n^\bullet: \Ot^\bullet(W_n) \to \Ot^\bullet(W_n) $$
This Hadamard product is local and its norm is again bounded by a Bruno sequence~(Proposition \ref{P::Hadamard}). In the same way, the map sending a function to its Hamiltonian field is, by the Cauchy-Nagumo lemma, again local etc.

\subsection{Locality of  $j_n^\bullet$}
The  solution $j_n$ to the homological equation is now defined by
$$j_{T,n}(x)=L_n x-L_n((L_nx)T) $$
The maps now depend on $T$ but an estimate is easily found:

\begin{proposition}
\label{P::estimate}
The map 
$$j_{T,n}^c:\Ot^c(W_n) \to \Theta^c(W_n)$$
is local for the weight 
$$\l_n(s,t)=3^{-2k-1}(t-s)^{2k+1} \left(\frac{t}{s} \right)^{2^n+1}  $$
and moreover
$$ |j_{T,n}|/(1+|T|) \leq |L_n|+|L_n|^2.$$
\end{proposition}
\begin{proof}
As in the proof of the composition lemma (Lemma~\ref{L::composition}), we take equidistant points $s_0=s<s_1<s_2<s_3=t \leq 1$ and identify
$x,T \in \Ot^c(W_n)_{t}$ with associated horizontal sections.
As
 $$j_{T,n}(x)=L_n x-L_n((L_nx)T) $$
 we have:
 \begin{align*}
  |L_n((L_nx)T)| &\leq   \left(\frac{s_0}{s_1}\right)^{2^{n}+1}\frac{|L_n|}{(s_1-s_0)^k}\, |((L_nx)T)|\\
     & \leq \left(\frac{s_0}{s_2}\right)^{2^{n}+1} \frac{|L_n|}{(s_1-s_0)^{k}(s_2-s_1)}\, |L_nx|\, |T|\\
     & \leq \left(\frac{s_0}{s_3}\right)^{2^n+1}\frac{|L_n|^2}{(s_1-s_0)^{k}(s_2-s_1)(s_3-s_2)^k}\, |x|\, |T|\\
     &\leq 3^{2k+1}\left(\frac{s}{t}\right)^{2^n+1} \frac{|L_n|^2}{(t-s)^{2k+1}}\, |x|\, |T|
 \end{align*}
 From this estimate, we deduce that :
\begin{align*}
 |j_T(x)| & \leq \left(\frac{s}{t}\right)^{2^n+1} \frac{|L_n|+3^{2k+1}|L_n|^2\, |T|}{(t-s)^{2k+1}}\, |x| \\
  & \leq 3^{2k+1}\left(\frac{s}{t}\right)^{2^n+1} \frac{|L_n|+|L_n|^2}{(t-s)^{2k+1}}(1+ |T|)|x| \\
\end{align*}
This proves the proposition.
\end{proof}
Let us make a few comments on this result.
\begin{enumerate}[{\rm 1)}]
 \item If $(a_n) \in \Bt^-$ then $|L_n| \in \Bt^+$ and therefore  $(|j_n,T|/(1+|T|)) \in \Bt^+$ as well.
 \item  By Proposition~\ref{P::maximal}, the map $j_n^c$ induces a map  on the maximal spaces which is local for the weight:
$$\mu_n(s,t)=\left(\frac{s}{t}\right)^{2^n}(t-s)^{k}$$
\end{enumerate}


\subsection{Functional calculus in Banach algebras}

For an ordinary Banach algebra $A$, we may take the image $f(u)$ of $u \in A$ by an analytic series
$$f(z)=\sum_{n \geq 0} a_n z^n \in \CM\{z\}$$ provided that the norm of $u$
is smaller that the convergence radius of $f$. Moreover defining the {\em absolute value} of $f$ by
$$|f|(z):=\sum_{n \geq 0} |a_n| z^n $$
we get the estimate
$$\| f(u)\| \leq |f|(\| u \|) $$
So we have a non linear map
$$f:A(r) \to Hom(A,A),\ u \mapsto f(u)$$
where $A(r)$ is the ball of radius $r$ and $\Hom(A,A)$ is the space of continuous linear maps. We want to emulate this construction in the context of Kolmogorov
spaces.
\subsection{Non linear bounded maps} 
We now define non-linear maps in Kolmogorov spaces.
Consider a Kolmogorov space
$$E \to B $$
Inside the fibre $E_b$ we denote by  $E(r)_b$ the ball of radius $r$ centred at the origin. According to our philosophy we get a relative set
$$E \to \RM_{>0} \times B  $$
whose fibre  above $r,b$ is the ball $E_{b}(r)$. We call this relative set
the {\em ball of the Kolmogorov space $E$}.
We define the Kolmogorov space of {\em bounded mappings}
$$\Bt(E,F) \to \RM_{>0} \times B^{op} \times C    $$
whose fibre above $(r,b,c)$ consists of bounded maps from $E(r)_b$ to $F_c$ with the supremum norm.
\subsection{The Borel map}
We  consider a Kolmogorov space 
$$E \to \RM_{>0}$$
In view of the Cauchy-Nagumo lemma (Lemma \ref{L::Cauchy_Nagumo}), local maps
$u \in L^\l(E,E)$ with weight $(t-s) $ can be seen as  generalisation of
order one partial differential operators, that is, a vector fields. In the
same way that a vector field admits a flow, such operators can be exponentiated and, more generally, we can perform functional calculus on them. As the flow
is defined only for small times, we cannot expect to generalise the existence
of functional calculus to Kolmogorov spaces in a straightforward manner:
singularities should occur even for entire function such as the exponential. 

We study a local situation and in particular our vector fields, which induce change of variables, vanish at the origin. Therefore from now on
we consider local maps for the weight (for technical reasons we multiply
the weight by the constant $e=2.718 \dots$)
$$\l(t,s)=(t-s) e  $$

For a series $f$ we denote the {\em Borel transform} by 
$$\Bt f:=\sum_{n \geq 0} \frac{a_n}{n!}z^n, $$
 \begin{proposition}
 \label{P::Borel} 
Let $f=\sum_{n \geq 0} a_n z^n \in \CM\{z\}$ be an analytic series with $R_f$ as radius 
of convergence. There is a well-defined horizontal section 
$$\b_f \in \G^h(U,\Bt(L^\l(E,E),Hom(E,E))) \text{ over } U=\{ (r,s,t) \in \RM^3_{>0}: r <e^{-1}\l(t,s)R_f \} $$
with
$$\b_f(r,s,t):L^\l(E,E)(r)_t \mapsto Hom(E_t,E_s),\ u \mapsto  \Bt f(u(t,s)).$$ Moreover we have the estimate 
$$|\Bt f(u(t,s))| \leq |f|\left(\frac{| u(t,s) |}{t-s}\right) .$$
\end{proposition}
\begin{proof}
We cut the interval $[s,t]$ into $n$ equal parts:
$$s_k=s+k\frac{t-s}{n},\ k=0,\dots,n $$
We denote by $\| \cdot \|$ the usual Banach space norm. We have:
\begin{align*}
 \| u^n(t,s) \|  =\| \prod_{i=1}^n u(s_{i+1},s_i)|\ &\leq \prod_{i=1}^n\| u(s_{i+1},s_i)\|\\
   & \leq |u(t,s)| e^{-n}\prod_{i=1}^n (s_{i+1}-s_i)^{-1}\\
    &\leq |u(t,s)| e^{-n} n^n (t-s)^{-n}\\
    &\leq | u(t,s) |^nn! (t-s)^{-n}, \\
\end{align*}
(we used the standard inequality $n^n \leq e^n n!$) and therefore:
$$\|\sum_{n \geq 0}  \frac{a_n}{n!} u^n(t,s) \| \leq \sum_{n \geq 0} |a_n| \left(\frac{| u(t,s) |}{t-s}\right)^n $$
assuming
$$\frac{| u(t,s) |}{t-s} \leq R_f \iff | u(t,s) | \leq e^{-1} \l(t,s) R_f $$
This proves the proposition.
\end{proof} 
 
\begin{example}
 Consider the Kolmogorov space
 $$\Ot^c(D) \to \RM_{>0} $$
 where $D \to \RM_{>0}$ has fibres
 $$D_t=\{ z \in \CM: |z| <t \} .$$
Let us compute the exponential of the derivative:
 $$u(t,s):\Ot^c(D_t) \to \Ot^c(D_s),\ g(z) \mapsto g'(z) $$
According to Taylor's formula
$$e^ug(z)=g(z+1) $$
As an unbounded operator, the definition domain of the exponential consists of map having a convergence radius $>s+1$.

The theorem says that the taking the exponential of the derivation, which corresponds to composition the flow $\p$: 
 $$(e^u)_{t,s}:\Ot^c(D_t) \to \Ot^c(D_s)$$ gives a well defined homomorphism provided we have the estimate:
 $$|u(t,s)|<t-s .$$
 But by Cauchy-Nagumo inequalities $|u(t,s)|=1$ and we recover the condition $t>s+1$.
 
 So functorial calculus informs us on the way the disk $D_s$ is translated under the flow of the vector field: if the estimates $t>s+1$ holds then the image of the disk $D_s$ under the flow of $\d_z$  at time $1$ is contained inside the disk $D_t$.
 
   In the HNF iteration, the situation is of course much more complicated: domains might be shrinked and moved in a complicated way. Our formalism takes care by itself of the necessary information to have a well-defined exponential. 
\end{example}
 
\subsection{Product lemma}

A non-local map carries a definition domain, that is the set of parameters over which it is defined. When we compose such maps, convolutions of sets arise. This leads to the following definition:
\begin{definition}\label{D::convolution}\index{convolution of sets}  
The {\em convolution} of $A_1 \subset B_1 \times B_2$ and 
$A_2 \subset B_2 \times B_3$ is the set
\[A_1 \star A_2:=\{(r,t) \in B_1 \times B_3\;|\;\;\exists s \in B_2\;\;\textup{such that} (t,r) \in A_1, (r,s) \in A_2\}\] 
\end{definition}
\begin{theorem} 
\label{T::compexponential1}
Let $E \to \RM_{>0}$  be an Kolmogorov space, and 
\[(u_i) \subset  L^1(E,E)_\l\] 
a sequence of local morphisms with weight $\l(t,s)=t-s$ and norm bounded by~$1$. Assume that 
$\s:=\sum| u_i | <+\infty $,
then the sequence
$$g_n:=e^{u_n}e^{u_{n-1}}\cdots e^{u_0}  $$ converges to a horizontal section $g \in \G^h(A,\Hom(E,E))$ where
\[A:=\{(t,s) \in ]0,\tau]^2\;|\;\sum_{i \geq 0}|u_i|<t-s \}.\]
Furthermore, one has the estimate for the norm function 
$$| g | < \frac{1}{1-\s\nu}  ,$$
with $\nu:= t/(t-s).$
 \end{theorem}
 \begin{proof}
The map $e^{u_i}$ is defined over the simplex
$$A_i:= \{(t,s) \in ]0,\tau]\times ]0,\tau]\;|\;|u_i|<t-s \}$$
and on $A_i$ one has
$$|e^{u_i}| \leq \frac{1}{1-\frac{|u_i|}{t-s}} .$$ 

Let us now consider the composition of two exponentials $e^{u_i} e^{u_{i+1}}$.
As 
$$\frac{1}{1-x} \times \frac{1}{1-y} < \frac{1}{1-(x+y)} $$
for $x,y \in ]0,1[$, the set $A_{i+1} \star A_i$ over which the composition $e^{u_{i+1}} e^{u_i}$ contains the set
$$ \{(t,s) \in ]0,\tau]\times ]0,\tau]\;|\;|u_i|+|u_{i+1}|<t-s \}$$
and one has the estimate
$$|e^{u_{i+1}} e^{u_i}| \leq \frac{1}{1-\frac{|u_i|+|u_{i+1}|}{t-s}}=\frac{t-s}{1-|u_i|-|u_{i+1}|} $$
By a straightforward induction, we get the estimate
$$|e^{u_n}e^{u_{n-1}}\cdots e^{u_0}  | \leq \frac{t-s}{1-\sum_{i= 0}^n|u_i|}=C_n(t,s) $$
over the set
$$ B_n=\{(t,s) \in \RM_{>0}\;|\;\sum_{i=0}^n |u_i|<t-s \}.$$
For $(t,s) \in \bigcap_n B_n$, the sequence $(g_n(t,s)) \subset Hom(E_t,E_s)$ is easily seen to be a Cauchy sequence using the estimate:
\begin{align*}
 |g_{n+1}-g_n|=|(e^{u_{n+1}}-\Id)g_n| &\leq C(t,s)|(e^{u_{n+1}}-\Id)| \\
 &\leq C(t,s)\frac{\nu_{n+1}}{1- \nu_{n+1}} \sim C(t,s) \nu_{n+1}  
\end{align*}
with $\nu_{n+1}=|u_{n+1}|/(t-s) \to 0$ and 
$$C(t,s):=\frac{t-s}{1-\sum_{i \geq 0}^n|u_i|} $$
This proves the theorem.
\end{proof}

\section{The fixed point theorem}
\label{S::fixed_point}
\subsection{Arnold spaces iterations}
Consider an Arnold space
$E \to \NM .$
A sequence $(x_n), x_n \in E_n$  is called {\em summable} if the series $\sum |x_n|$ converges. In practice, we have bounded maps
$$E_n \to E_\infty $$
and therefore a summable sequence defines a sum in $E_\infty$. Assume for instance that $(K_n) \subset \CM^d$ is a decreasing sequence of closed sets, then $E_n=\Ot^c(K_n)$ are  Banach spaces and hence Kolmogorov spaces. Thus we have an Arnold space
$$E \to \NM $$
with fibre $E_n$. If we consider the set of Whitney smooth functions 
$$E_\infty=C^\infty(K_\infty,\CM),\ K_\infty=\bigcap_{n \in \NM} K_n$$
then we have bounded maps
$$E_n \to E_\infty $$
and the image of a summable sequence in $E$ is a summable sequence in $E_\infty$. 
\subsection{Fixed point theorem and HNF iteration}
Let us now come back to the HNF iteration.
For the HNF iteration, we choose a falling Bruno sequence $a \in \Bt^-$ and the sets $W_{n,s}=Z_{n,s}(a) \times D_{s}^3$ with coordinates $\omega_i$, $\tau_i,q_i,p_i$ and $V_{n,s}= Z_{n,s} \times D_{s}$ with coordinates $\omega,\tau$. These define Arnold spaces
$$E:=\Mt^c(V) \times \Theta^c(W) \times \Mt^c(W) \to \NM \times \RM_{>0}$$
 Propositions \ref{P::estimate} and \ref{P::Borel}  imply that there exists an increasing Bruno sequence $(b_n)$ such that
 the maps
$$\phi_n:(A,B,v) \mapsto (A,0)+f_n(A,B,v) $$
involved in the HNF iteration are in fact horizontal sections of $\Bt(E,E)$ over a set $U_n$
$$U_n =\{(r,s,t) \in \RM_{>0}^3: r  \leq \l_n(t,s)  \}$$
for some  function of the form
$$\l_n(t,s)=b_n \left(\frac{t}{s}\right)^{2^n}t^{-l}(t-s)^k  $$
and
$$|\phi_n(x)-\rho_{n,n+1}(x)|_s \leq \l_n^{-1}(s,t)|x|_t $$
Here $(b_n)$ is a falling Bruno sequence which depends on $(a_n)$.

Surprisingly enough, these conditions are sufficient to ensure fast convergence of the HNF iteration:

 \begin{theorem}
  Let $E \to B$ be an Arnold space and define
  $$\l_n(t,s)=b_n \left(\frac{t}{s}\right)^{2^n}t^{-l}(t-s)^k  $$
  with $k,l \geq 0$ and $b \in \Bt^-$. A sequence of sections
  $$\phi_n \in \G(U_n,\Bt(E_n,E_{n+1})) $$
  satisfying the estimate 
  $$|\phi_n(x)-\rho_{n,n+1}(x)|_s \leq \l_n^{-1}(s,t)|x|_t $$
  with
  $$U_n =\{(r,s,t) \in \RM_{>0}^3: r \leq \l_n(s,t)  \}.$$
  Then for any falling Bruno sequence $c=(c_n) \in \Bt^-$ and $r_0 \leq 1$ which satisfy:
  $$c \leq 1 \text{ and } \sum_{n \geq 0} c_n < r_0  $$
  there exists a decreasing sequence $(s_n)$ which converges to a positive limit such that
  $(r_0,s_n,s_{n+1}) \in U_n$ for any $n \in \NM$, the sequence of iterates 
  $$x_{n+1}=\phi_n(s_n,s_{n+1},x_n)$$
  exists for any $x_0 \in E_{s_0}(c_0)$ and moreover
  $|x_n|<c_n $.
  \end{theorem}
\begin{example}
Consider the Arnold space
$$\NM \times \RM_{>0} \times \NM \times \RM_{>0} \to \RM,\ (n,t,x) \mapsto (n,t)   $$
where the connecting morphisms identify the different fibre with the  same $v$.
The theorem says that the real sequence
$$x_{n+1}=x_n+\l_n(s_{n+1},s_n)^{-1},\ \l_n(s,t)=a_nt^{-l}\left(\frac{t}{s}\right)^{2^n}(t-s)^k $$
converges for an appropriate choice of the sequence $(s_n)$, for $k,l>0$.
This statement about real sequence captures all the complexity of the theorem as will be seen in the proof.
\end{example}
\begin{proof}
By the absorption lemma (Lemma \ref{L::absorb}), we may choose $\rho \in \Bt^-,s_0>0$ such that
$$\l_n(s_{n+1},s_n)^{-1}\leq   c_{n+1}$$
Define the converging sequence $s$ by
$$s_{n+1}=\rho^{1/2^{n}}s_n .$$ 
We have $(r_0,s_n,s_{n+1})  \in U_n$. Indeed as 
$$\l_n(s_{n+1},s_n)^{-1}\leq   c_{n+1}$$ and $c_k  \leq 1$ for any $k$, 
we deduce that
$$\l_n(s_{n+1},s_n) \geq 1 $$
while $r_0 \leq 1$. This shows the assertion.

An obvious induction shows that
$$|x_n| \leq \sum_{i=0}^n c_i<r  $$
Indeed 
$$|x_{n+1}| \leq |x_n|+|x_{n+1}-x_n| \leq |x_n|+|\l_n^{-1}(s_{n+1},s_n)| \leq   \sum_{i=0}^{n+1} c_i .$$
This shows that the iteration is well-defined and moreover
$$|x_{n+1}-x_n|\leq \l_n(s_{n+1},s_n)^{-1} \leq c_n$$ 
This concludes the proof of the theorem.
\end{proof}
\subsection{Convergence of the HNF iteration}
For the HNF iteration, we choose a falling Bruno sequence $a \in \Bt^-$ and the sets $W_{n,s}=Z_{n,s}(a) \times D_{s}^3$ with coordinates $\omega_i$, $\tau_i,q_i,p_i$ and $V_{n,s}= Z_{n,s} \times D_{s}$ with coordinates $\omega,\tau$. These define Arnold spaces
$E^k=\Mt(\Ot^k(V) \times \Theta(\Ot^k(V)) \times \Ot^k(W))$.
Therefore applying the fixed point theorem, we get the
\begin{theorem}
\label{T::Lagrange} Consider an analytic Hamiltonian of the form
\[ H=\sum_{i=1}^n \a_i p_iq_i +O(3) \in \CM\{p,q\} \]
and put
\[ F_0:=H+\sum_{i=1}^d\omega_i p_iq_i =A_0+B_0\]
with $\a \in \CM^d(a)$ where $a$ is a falling Bruno sequence.
There exists a falling sequence $s=(s_n)$ converging to a positive number such that the vector fields $v_n$ of the HNF iteration exponentiate to elements
\[ \varphi_n=e^{-v_n} \in Hom(\Mt(\Ot^k(W))_{n,s_n},\Mt(\Ot^k(W))_{n+1,s_{n+1}})\]
and  the composition 
$$\Phi_n=\varphi_{n-1}\varphi_{n-2}\ldots\varphi_1\varphi_0$$
converges to a Poisson morphism 
\[\Phi_{\infty} \in Hom(\Mt(\Ot^k(W))_{0,s_0},\Mt(C^k(W(a)))_{s_\infty}),\] 
which  reduces $F_0$ to its Hamiltonian normal form.
\end{theorem}
This concludes the functional analytic part of the paper. 
 
\section{Arithmetic density}
\label{ArithmeticDensity}

\subsection{Frequency relations}
\label{SS::relation}
Recall that in the Hamiltonian normal form iteration, the original Hamiltonian $H(p,q)$ 
is obtained from $F_0$ by equating to zero the $G_{0,i}:=\omega_i$:
\[F_0(\omega=0,\tau,p,q)=H(p,q).\]
After the first iteration step, the functions $G_{0,i}:=\omega_i$ 
are transformed, we truncate them to define the next frequency relation
$$G_{1,i}(\omega,\tau)=[e^{-v_0}\omega_i]^4,\ G_1=(G_{1,1},\dots,G_{1,d}),$$
but as the vector field $v_0$ happens to be Hamiltonian, there are no 
$\partial_{\omega_i}$ terms, and $G_1=G_0$ and so
$$[A_1(\omega_1=0,\tau,p,q)]^{2^1+2}=[B(q,p)]^{2^1+2}$$
is also the first iterate of the Birkhoff normalisation.
However, at the next step, the vector field $v_1$ might contain a 
non-Hamiltonian part, so that now in general $A_2 \neq A_1$:
$$[A_2(G_2=0,\tau,p,q)]^{2^2+2}=[B_H(q,p)]^{2^2+2} \ \mod I^2 \oplus \CM[[\tau]]$$
with $G_2=[e^{-v_1}G_1]^6$.
To make this relation between $F_2$ and $H_2$ explicit, we need to solve the equation $G_2=0$. From the fact that vector fields $v_1$ has order $2$, we have
\[G_{2,i} = \omega_i +O(2),\]
so the equations $G_{2,i}(\omega,\tau)=0$ can indeed be solved for the $\omega_i$ and we obtain power series 
\[\omega_{2,i}(\tau) \in \CM[[\tau]],\]
such that
\[G_{2,i}(\omega_2(\tau),\tau)=0 .\]
More precisely by the Weierstrass preparation theorem, we have
$$G_{2,i}(\omega,\tau)=u_i(\omega,\tau)(\omega_i-\omega_{2,i}(\tau)), $$
where the $u_i \in R_0$ are units, i.e. $u_i(0) \neq 0$.

Comparing with the Taylor expansion of \ref{SS::Moser}:
$$B(pq)= B(\tau+f)=B(\tau)+\sum_{i=1}^db_i(\tau)f_i\ \mod I^2,$$
we deduce that the coefficients $b_i(\tau)$ of the frequency map are related to the $\omega_{n,i}(\tau)$ by the congruence
\[\alpha_i+\omega_{n,i}(\tau) = b_i(\tau) +O(2^n),\]
The germ at the origin of the Hamiltonian normal form gives the first terms of the Birkhoff normal form
$$A_{H,n}(\omega_n(\tau),\tau)=B(\tau)+O(2^n+2). $$

\begin{definition}
The functions $G_{n,i}$ are called the {\em frequency relations} and the algebraic manifolds defined by these equations are called the {\em frequency manifolds}:
\[X_n:=\overline{\{ (\omega,\tau)\;|\; 
[G_{n,1}]^{2^n+2}(\omega,\tau)=\dots=[G_{n,d}]^{2^n+2}(\omega,\tau)=0 \}} \]
\end{definition}
In a neighbourhood of the origin, the frequency manifolds are the graphs  of the map  germs
$$\tau \mapsto \omega_n(\tau)=(\omega_{n,1}(\tau),\dots,\omega_{n,d}(\tau))$$
These maps identify functions in the $\tau$ variables with functions on $X_n$ and $H_n$ with the restriction of 
$F_n$ to $X_n$. So geometrically $e^{-v_n}$ is the flow at time $-1$ of the vector field $v_n$ and it maps the manifold $X_{n-1}$ to $X_n$.
 
 The problem we now address is to estimate the measure of the manifolds $X_n(a)$ in a common neighbourhood of the origin
 but before that let us consider a baby example.
\subsection{A trivial example}
\label{SS:Freq_manifold}
Let us go back to our $d=1$ example:
$$H(q,p)=pq+p^3+q^3 $$ 
and the ideal $I=(f)$ with $f=pq-\tau$.  
The iteration produces
$$\begin{array}{ l  l  } 
 A_{H,0}&=(1+\omega)\tau\\
    G_0&=\omega\\
 A_{H,1}&= (1+\omega)\tau\\
    G_1&=\omega\\
 A_{H,2}&=(1+\omega)\tau+3\frac{\tau^2}{1+\omega}\\
    G_2&=\omega+\frac{6\tau}{(1+\omega)}+o(2)
\end{array}$$

The frequency manifolds $X_0,X_1$ are just the line $\omega=0$ lying in the $\{ \tau,\omega \}$-plane while, up to first order expansion in $\tau$, the frequency manifold $X_2$ is a parabola:
$$X_2=\{ (\tau,\omega):\omega^2+\omega+6\tau=0 \}  $$
\begin{figure}[htb!]
\includegraphics[width=0.4\linewidth]{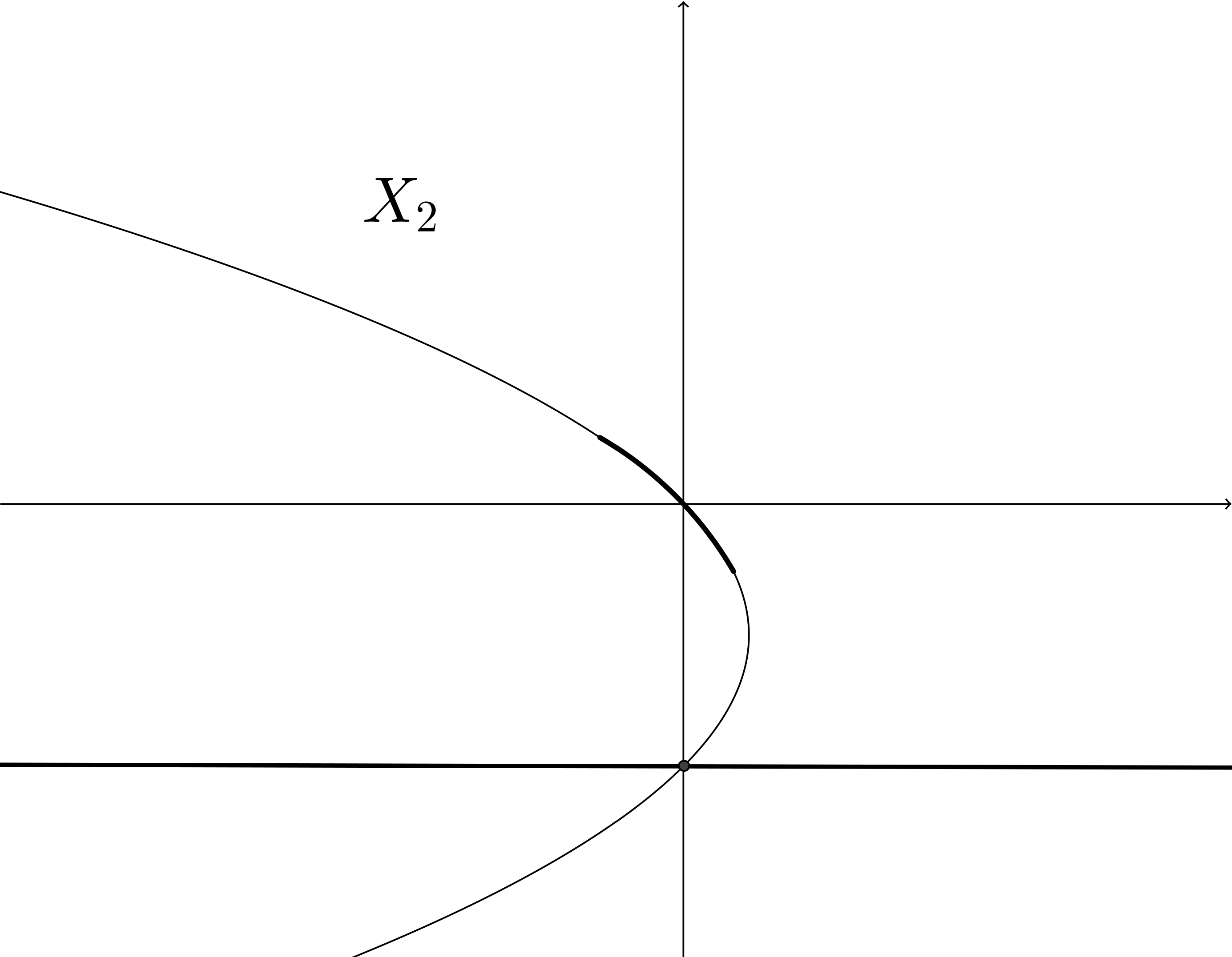} 
\end{figure}
\ \\ 
We now look at the germ of the parabola at the origin, that is, we solve the equation $G_2=0$ gives the first-order frequency of motion
$$\omega_2(\tau)=-6\tau+O(4) $$ 
Substituting $ \omega_2(\tau)$ into  $F_2$ and taking the constant term
by putting $pq=\tau$, we obtain the first two terms of the Birkhoff normal form:
$$ A_{H,2}(\omega_2(\tau),\tau)=\tau-3\tau^2+O(6)$$

Going to the next order one finds:\\
\[\omega_3(\tau)=-6\tau-36\tau^2-420\tau^3+O(8)\]
\[ A_{H,3}(\omega_3(\tau),\tau)=\tau-3\tau^2-12\tau^3-105\tau^4+O(10),\]
which reproduces the first four terms of the Birkhoff normal form. However the Birkhoff normal does not see that the curve bends back to the resonance. The functions $H_2, H_3,\ldots$ can be seen as the germs at the origin of $F_2,F_3,\ldots$ restricted to these curves.

In this example, the only resonance is at $\omega=-1$ so there is no problem of measure estimates. In the higher dimensional case,
we will have to throw away more and more neighbourhoods of resonance hyperplanes in the $\omega$-component as the iteration process goes by. 

It can be proven that the frequency manifolds are, in general, quite far from being graphs: they wind more and more around the origin\footnote{We intend to publish this result elsewhere.}. Now convergence of the Hamiltonian normal form iteration implies convergence of the frequency manifolds as graphs 
in $\RM^d(a)$ when $a$ is a Bruno sequence. This means that the manifold can be cut so to become a
graph. 

Intersection of the frequency manifold  with a two dimensional plane might give something like the following picture:\\
\begin{figure}[htb!]
\includegraphics[width=0.4\linewidth]{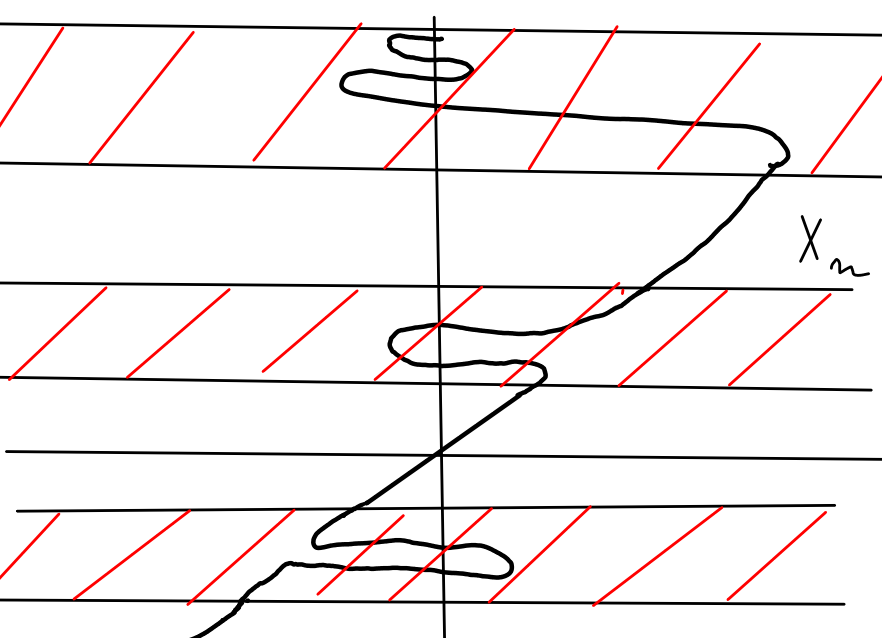} 
\end{figure}
\ \\ 
If we cut the curve along the hatched strips then we get a curve which is a graph.  The problem is now to get good measure estimates for the cut
manifold. This leads to revisit the classical works of Diophantine geometry, in the spirit of Kleinbock and Margulis~\cite{Kleinbock_Margulis}.
\subsection{Lebesgue density}
The  {\em Lebesgue density of a measurable set $X \subset \RM^d$ at a point $p$}  is defined  (when the limit exists) by
\[ \delta(X,p):=\lim_{\e \longrightarrow 0}\frac{m(B_{p,\e} \cap X)}{m(B_{p,\e})}\]
Here $m$ denotes the Lebesgue measure and $B_{p,\e} \subset \RM^d$ is the ball with radius $\e$, centred at $p$.

For instance, the half-line $\RM_\geq 0$ has density $1$ at all of its points except  at the origin where it has density $1/2$. As a further example, consider the disjoint union of intervals:
$$X =\bigcup [a_i,b_i] $$
defined by decreasing geometric sequences
$$b_n=\frac{1}{4^n},\ a_n=\frac{b_n}{2}. $$
The measure of $[a_n,b_n]$ equals $a_n$ and 
$$m([-\e,\e] \cap X) \sim 4^{n-1}\sum_{k \geq n}\frac{1}{4^k}=\frac{1}{3} $$
Both examples are exceptional as we compute the density at ''boundary points'': the {\em Lebesgue density theorem} asserts that $X$ has a density at almost all points $p$, and that for almost all points the density is either $0$ or $1$.

One may replace $\RM^d$ by a $d$-dimensional Riemannian $C^1$-submanifold $M \subset \RM^n$ and check that the density is independent on the metric.
The density thus defined can be computed as in $\RM^d$ using a $C^1$-chart.\\

\subsection{Density of arithmetic classes}
The choice of a positive decreasing sequence $a=(a_n)$ defines a set
$$\RM^d(a):=\{ \a \in \RM^d:\forall n \in \NM,\ \forall J \in \ZM^d\setminus \{0\},\ \| J \| \leq 2^n,\   |(\a,J)| \geq a_n \} $$
lying in the complement of the resonance hyperplanes. We will need to study the complex geometry of such sets, therefore we define
$$\CM^d(a):=\{ \a \in \CM^d:\forall n \in \NM,\ \forall J \in \ZM^d\setminus \{0\},\ \| J \| \leq 2^n,\   |(\a,J)| \geq a_n \} $$
where the brackets $(-,-)$ stand for the complexification of the Euclidean scalar product (and not the Hermitian product).

We also define $X(a)=X \cap \CM(a)$ for any given subset $X$. This sets are obtained by inductive construction and therefore we also define
$$\CM^d(a)_n:=\{ \a \in \CM^d:\forall k \leq n,\ \forall J \in \ZM^d\setminus \{0\},\ \| J \| \leq 2^k,\   (\a,J) \geq a_k \} $$
and the corresponding sets $X(a)_n$.

The following elementary but fundamental result shows that although arithmetic classes may have an empty interior, after replacing $a$ by a slightly smaller sequence $\nu a$, they are {\em big} in the sense of measure theory.

\begin{proposition}[\cite{arithmetic}]
\label{P::arithmetic}
Let $a=(a_n)$ be a positive decreasing sequence and $\beta \in \RM^d(a)$. 
Let $\nu=(\nu_n)$ be another positive decreasing sequence with $\nu_i \le 1$ and 
$\sum_{i=1}^{\infty}\nu_i <\infty$.
Then the Lebesgue density of $\RM^d(\nu a)_{\infty}$ at $\beta$ is equal to $1$: 
\[\lim_{\e \longrightarrow 0}\frac{m(B_{\b,\e}(\nu a))}{m(B_{\beta,\e})}=1\]
Here $\nu a$ is the sequence with terms $\nu_na_n$ and $m$ denotes the
Lebesgue measure, $B_{\b,\e} \subset \RM^d$ the ball with radius $\e$, centred at $\b$.
\end{proposition}
A similar result holds in the  complex case. 
\subsection{Curvedness}
Taking the preimage of a positive measure set by a map might give a zero measure set. Consider for instance the preimage of the set
$$X=\{ (x,y) \in \RM^2:|y|>x^2 \} $$
by the map
$$f:\RM \to \RM^2,\ t \mapsto (t,0) $$
The image of the map is the first coordinate axis which intersects $X$ only at the origin thus we get
$$\dt(X,0)=1,\ \dt(f^{-1}(X),0)=0. $$
\begin{figure}[htb!]
\includegraphics[width=0.4\linewidth]{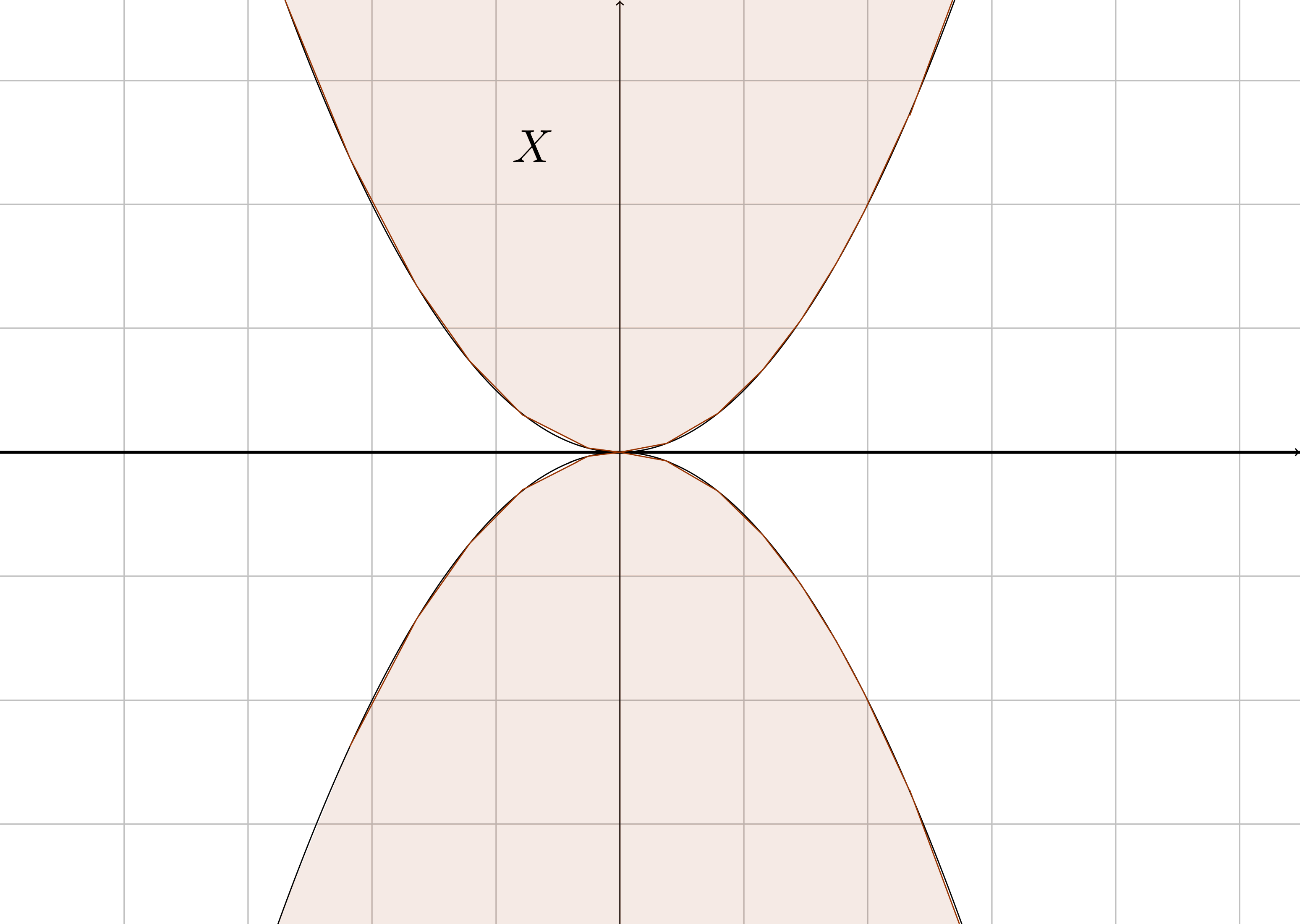} 
\end{figure}
\ \\ 
Arithmetic classes are defined by linear inequalities so we do not get such phenomena.
Say we have a map
$$f:\RM^k \to \RM^d $$
The preimage of an arithmetic class is obtained by throwing away preimages of neighbourhoods of resonance hyperplanes
$$\{ x \in \RM^k:|(f(x),I )| <a_i \} $$
Now we may forget that $I$ has integer components and simply look for linear combinations
$$F_b(x)=\sum_{j=1}^k b_j f_j(x) $$
and estimate the Lebesgue measure of the set $|F| \leq a_i=\e$ when $\e$ is small.   

Curvedness conditions on the map are used to control the density
for such inverse image sets defined by linear inequalities. We start by defining this crucial notion due to Kleinbock and Margulis~\cite{Kleinbock,Kleinbock_Margulis}.

\begin{definition} For a subset $X \subset \RM^d$ we denote by $\text{Span}(X)$ the {\em span} of $X$, i.e. the smallest affine subspace of $\RM^d$ containing $X$. For a map germ
$f:(\RM^m,\a) \to (\RM^d,\b)$ we call the span of $f$ at $\a$ the affine space
\[\text{span}(f,\a) =\cap_U \text{Span}(f(U)) \subset \RM^d ,\]
where the intersection runs over all representatives of the germ $f$.
\end{definition}
Note that we may always take a representative $f:U \to \RM^d$ of $f$ such that
\[\text{span}(f,\a) =\text{Span}(f(U)) \subset \RM^d .\]
\begin{example}
Consider for instance the map:
$$f:(\RM,0) \to (\RM^4,0):t \mapsto (t,t^3,t^5,0)$$
 then we get that
 $$\text{span}(f,\a)=\RM^3 \times \{ 0 \}.$$
\end{example}

\begin{definition}
For a map germ  $f:(\RM^m,\a)  \to (\RM^d,\b)$ and a multi-index $I \in \NM^d$, we denote by
\[ \d^If(\a):=\frac{\partial^I f}{\partial x^I}(\a) \in \RM^d .\]
the vector of $I$-th derivative of $f$ at $\a$. Furthermore we put
\[ \text{span}_l(f,\a):=\text{Span}\{\d^I f(\a): I \in \NM^m,\ |I| \leq l\}.\]
\end{definition}
\begin{example} For the map
$$f:(\RM,0) \to (\RM^3,0):t \mapsto (t,t^3,t^5,0)$$
we get that 
 \begin{align*}
  \text{span}_1(f,\a)&=\text{span}_2(f,\a)=\RM \times \{ 0\},\\
  \text{span}_3(f,\a)&=\text{span}_4(f,\a)= \RM^2 \times \{ 0\},\\
  \text{span}_5(f,\a)&=\RM^3 \times \{ 0 \} \\
 \end{align*}
\end{example}
Clearly one has
\[ \{\b\}=\text{span}_0(f,\a) \subset \text{span}_1(f,\a) \subset \text{span}_2(f,\a) \subset \ldots \subset \text{span}(f,\a).\]

\begin{definition}  
 A $C^k$ map-germ $f:(\RM^m,\a)  \to (\RM^d,\b)$ is called {\em $l$-curved} ($l \le k$) if
  \[ \text{span}_l(f,\a)=\text{span}(f,\a).\]
$f$ is called {\em curved} if it is $l$-curved for some $l$. The smallest possible value of $l$  we call the {\em torsion index} and denote it by $t_\a(f)$.
If such an index does not exist, we set $t_\a(f)=\infty$.
A map $f:\RM^m \to \RM^d$ is curved at $\a$ if the corresponding germ $f:(\RM^m,\a) \to (\RM^d,f(\a))$ is curved.
\end{definition}
\begin{example}
The map germ
$$f:(\RM,0) \to (\RM^3,0):t \mapsto (t,t^3,t^5,0)$$
is curved with torsion index $t_0(f)=5$.\\
A flat map-germ like
$$ f:(\RM,0) \to (\RM,0),\ x \mapsto e^{-1/x^2} $$ is not curved and $t_0(f)=\infty$.
\end{example}
  
The following {\em arithmetic density theorem} was proven in \cite{arithmetic}, it will be our fundamental tool\footnote{In fact a much better bound can be given for the sequence $\nu$.}:
\begin{theorem}
\label{T::arithmetic}
Consider a curved map-germ
$$f=(f_1,\dots,f_d):(\RM^m,\a)  \to (\RM^d,\b) $$
 a real positive decreasing sequence $\s=(\s_k)$ and let $\nu=(\nu_k)$ be a real positive sequence  such that the sequence 
$$(2^{kd}\nu_k^{1/mt_\a(f)})$$ is summable and $\nu_k<1$ for all $k$'s.
Then the Lebesgue-density of the set $f^{-1}(\RM^d(\nu \s))$ at the origin is equal to $1$.\end{theorem}
\subsection{Analytic maps}
Flat maps like $e^{-1/x^2}$ do not occur during the iteration because the frequency manifolds are analytic. However passing to the limit, this could happen but in fact it does not as we shall see. Let us first start by an observation due to Kleinbock: 
 \begin{proposition}[\cite{Kleinbock}]A real analytic map $f:U \supset \RM^m \to \RM^d$
 is curved at any of its points. Moreover if $U$ is pathwise connected, we have
$$\text{span}(f,\a)=\text{span}(f(U)) $$
at any point $ \a \in U$.
\end{proposition}
\begin{proof}
 The sequence of affine spaces
$\text{span}_n(f,\a)$ generated by the vectors $\d^If(\a)$, $|I| \leq n$ stabilises,
say at level $l$. Let $u$ be a linear form vanishing on $\text{span}_l(f,\a)$, 
then the Taylor series of $u \circ f$ vanishes identically and therefore so does its germ at $\a$. Let now $\g$ be any path in $U$, the analytic continuation of $u \circ f$ along $\g$ vanishes and therefore, if $U$ is pathwise connected, it vanishes identically  in $U$.
\end{proof}
Of course if $U$ is not connected then the proposition fails. Take for instance  
$$f:\RM \setminus \{ 0 \} \to \RM,\ x \mapsto \left\{ \begin{matrix}0 & \text{ for } x <0\\ x & \text{ for } x >0 \end{matrix} \right.$$
Take $ \a <0$ and $\b >0$ we have:
$$\text{span}(f,\a)=\{ 0 \},\  \text{span}(f,\b)=\text{span}(f(U)) =\RM$$

\subsection{Why do we need complex geometry?}
In principle, the previous proposition allows us to go from local to global span. Unfortunately the sets $\RM^d(a)$ are disconnected and therefore we cannot apply the proposition. This forces us to look at the {\em complex geometry of arithmetic classes}\footnote{We thank Abed Bounemoura for pointing out that the Emmental properties that we will  now discuss were implicitly used in our original argument.}.

The previous proposition has of course an obvious variant  in the holomorphic case: {\em a holomorphic map $g:U \supset \CM^m \to \CM^d$
 is curved at any of its points and if $U$ is pathwise connected, we have
$$\text{span}(g,\a)=\text{span}(g(U)) $$
at any point $ \a \in U$.}
  
  It is an elementary but important fact that these notions behave correctly with respect to complexification: 
\begin{corollary}
Let $f:(\RM^k,\a) \to (\RM^d,\b)$ be a real analytic map which is the restriction of a complex holomorphic map
$g:(\CM^k,\a) \to (\CM^d,\b)$, then $\text{span}(f,\a)$ is the real part of $\text{span}(g,\a)$ and both have the same torsion index.
\end{corollary}
\begin{proof}
If we write $z_k=x_k+iy_k$ then the Cauchy-Riemann equations imply that
$$\d_{z_k}g(z)=\d_{x_k} f(x) $$
for $z=x \in \RM^k$
\end{proof}
Note that the corollary is wrong if the map $g$ is not holomorphic. Consider for instance, the map
$$g:(\CM,0) \to (\CM,0),\ z \mapsto (z-\bar{z})^2$$
then $f(x)=0$. As $\d_z(z-\bar{z})^2=2$,
we deduce that
$$\text{span}(g,0)=\CM \text{ and } \text{span}(f,0)=\{ 0 \} $$

Assume  now that we have a sequence of holomorphic maps defined on a compact set $K \subset \CM^n$: 
$$g_n:K \to \CM^d$$
which $C^k$-converges to a limit, $k=t_\a(g)$: 
$$g_\infty: K \to \CM^d $$
then the conclusion of the corollary still holds as the equalities
$$\d_{z}^Ig_n(x)=\d_{x}^I f_n(x) $$
pass to the limit.
\subsection{Emmental properties}
The complement of a complex analytic hypersurface in a complex manifold is pathwise connected. Our aim is to prove a KAM variant of this statement. Of course, if we consider a frequency manifold $X_n$ and look at the complement of a finite number of resonance hyperplanes, then we are back to complex geometry and it is
pathwise connected. But our situation
is different as we consider the manifold $X_n(a)$: the complement in $X_n$ of neighbourhoods of resonance hyperplanes. So the topology is a little bit more involved than usual.

\begin{proposition}
\label{P::Emmental}
Let $a=(a_n)$  be a decreasing positive sequence and $U \subset \CM^d$ be a disked set\footnote{If $x,y \in U$ then $x+\l (y-x) \in U$ for any complex $\l \in \CM$
satisfying $|\l| \leq 1$.}  For any $x,y \in \RM^d(a)_n$ there is a path joining $x$ to $y$ inside $U( a)_n$ for any $n \in \NM \cup \left\{ \infty \right\}$. 
\end{proposition} 
\begin{proof}
Let $L$ be the complex line passing through $x$ and $y$. It is parametrised by
$$\CM \to \CM^d,\ t \mapsto (1-t)x+t y   $$
 A set of the form:
$$C_I= \{\beta \in \CM^d: |(\beta,I)|  < a_k\}$$
is a cylinder. It intersect the line $L$ along a disk $D_I$. The center of the disk $D_I$ is the intersection of the  hyperplane
$$\Xi_I = \{\beta \in \CM^d: |(\beta,I)| =0\}$$
with the complex $L$ both given by real equations, it is therefore a real point. Therefore the set $L \cap D(a)_n$ is a complement of a discrete number of disks whose centers are real points. 

Consider now the disk $D_{x,y} \subset L \cap U$ centred at $x$ of radius $\| x-y\|$. Assume the disk $D_I$ intersects $D_{x,y}$. Both $x,y$ are real and do not lie on $D_I$ therefore $D_I$ is a disk contained in $D_{x,y}$ and has radius less than $1/2\| x-y\|$. So moving along purely imaginary values of $t$ until we reach the boundary of the disk $D_{x,y}$ and then on its boundary we join the point $x$ to the point $y$.\\
 
 \begin{figure}[htb!] 
\includegraphics[width=0.5\linewidth]{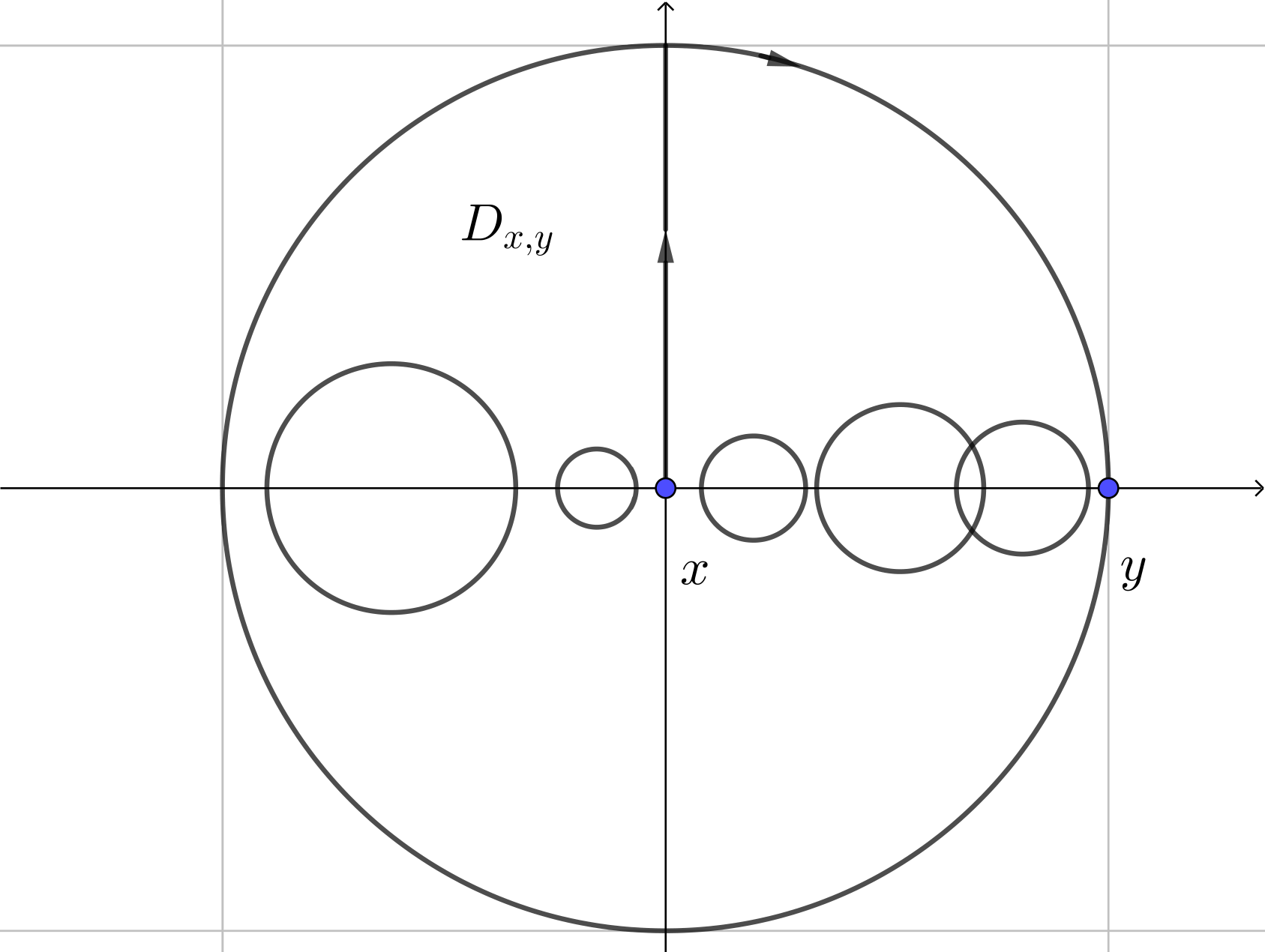}
\end{figure}
\end{proof}
The result of the proposition will be called the {\em Emmental property.} It admits the following variant:
\begin{proposition}
\label{P::Emmental_2}
Let $a$  be a positive decreasing sequence  and $U \subset \CM^d$ a disked set at distance $R>0$ from the origin. If 
$$\forall k \in \NM,\ a_k<\frac{R}{2^{k}}$$
then the set $U(a)_n$ is pathwise connected. Moreover if $a=O(4^{-k})$ then $U(a)$ is also pathwise connected.
\end{proposition} 
\begin{proof}
We prove that any two points $x$, $y$ with all but one equal coordinates can be joined by a path. By induction this will prove the proposition. To simplify the notations we assume that $x_i=y_i$ except for $i=d$. We also assume that $x$ is the centre of the ball.

Let $L$ be the complex line passing through $x$ and $x+e_d$ where $e_d$ is the last coordinate vector
$$e_d=(0,0,\dots,0,1) $$
The line is parametrised by
$$\CM \to \CM^d,\ t \mapsto x+t e_d   $$
and contains the point $y$.
 A set of the form:
$$C_I= \{\beta \in \CM^d: |(\beta,I)|  < a_k\}$$
is a cylinder. It intersect the line $L$ along the disk 
$$D_I=\{ t \in \CM : |I_d t+(x,I)| < a_k \}$$ 
 
\begin{lemma}
 For $I \neq J$, the disks $D_I$ and $D_J$ are disjoint.
\end{lemma}
\begin{proof}
The vectors 
$$u_1=(-I_2,I_1,0,\dots, 0),\ u_2=(0,-I_3,I_2,0,\dots,0),\dots$$
form an integer basis of the resonance hyperplane $\Xi_I$ and we have
$$\left(u_I,\frac{J}{\| J \|}\right) \geq \frac{1}{\| J \|} $$
Denote by $c_I \in U$ the centre of the disks $D_I$:
$$c_I=\sum_{j=1}^{d-1} \a_j u_j $$
The distance between the point $c_I$ and the hyperplane $H_J$ admits the estimate:
$$ d(c_I,\Xi_J)=\left|\left(c_I,\frac{J}{\| J \|}\right)\right| \geq \frac{\sum_{j=1}^{d-1} |\a_j| }{\| J \|} $$
Using the triangular inequality, we get that:
$$\| \sum_{j=1}^{d-1} \a_j u_j\| \leq \| I\| \leq \sum_{j=1}^{d-1} |\a_j| $$
and therefore
$$\frac{\sum_{j=1}^{d-1} |\a_j| }{\| J \|} \geq \frac{\| c_I \| }{\| I\|\,\| J \|} \geq \frac{R }{\| I\|\,\| J \|} $$
We have thus proved that
$$\| c_I-c_J\| \geq d(c_I,\Xi_J) \geq \frac{R }{\| I\|\,\| J \|} \geq \frac{R }{2^{n+m}} $$
where
$$n=\lceil \log_2\| I\| \rceil,\ m=\lceil \log_2\| J\| \rceil $$
From their equation, we see that the radius of the disks $D_I,D_J$ are respectively equal to $a_n/I_d$ and $a_m/J_d$ and therefore the assumption:
$$a_k<\frac{R}{2^{k}}$$
implies that these do not intersect.

\end{proof}

It is now easy to join the points $x,y$. We follow a straight line, if we hit the boundary of a circle $\overline{D_I}$ then we turn around this circle. At least one of the two moves does not hit the boundary of the polydisk in this way we construct a path joining the two points: 
 
 \begin{figure}[htb!] 
\includegraphics[width=0.5\linewidth]{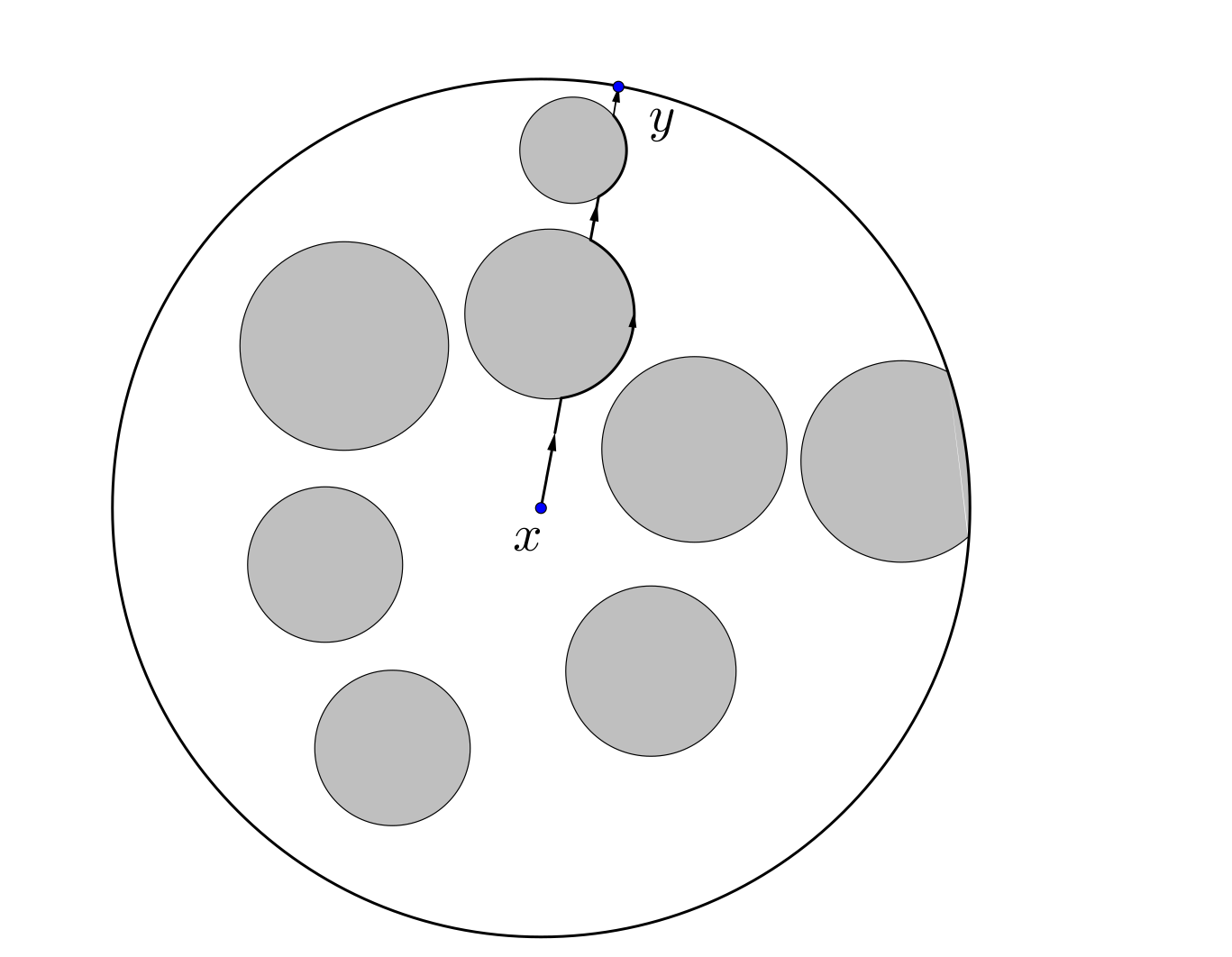}
\end{figure}
If we now let $n \to \infty$ and assume that $x,y \in U(a)$ then we construct inductively a sequence of paths $\g_n$ joining $x$ to $y$ by turning around the boundaries of the circles $\overline{D_I}$. The $C^0$-norm of the difference is bounded by the size of the additional circles therefore:
$$\|\g_n -\g_{n+1}\| \leq a_{n+1} $$
The sequence $(\g_n)$ is therefore a Cauchy sequence and thus converges to a limiting $C^0$-path. This proves the proposition.
\end{proof}
\subsection{From local span to global span}
To go from the span at the origin of the frequency manifold $X_n(a)$ to its global span, we need to prove connectedness of this set. The proof is similar to the one which asserts that irreducible analytic varieties are locally connected as we shall now see.

Recall that a function $f$ is  holomorphic inside a set $X \subset \CM^d$ if it is holomorphic inside an open neighbourhood containing $X$. Similarly we say  that $V \subset X$ is a complex manifold if it is the intersection with $X$ of a complex manifold.
\begin{proposition}
\label{P::connected}
Let  $ K $ be a pathwise connected compact set in $\CM^d$ containing the origin and let $U \subset \CM^d$ be an open neighbourhood of the origin. Let 
$$G_n: U \times K \to \CM^d,\ (x,y) \mapsto G_n(x,y) $$
be sequence of analytic maps  which is $C^1$-converging to a limit $G_\infty$.
Assume  that:
\begin{enumerate}
 \item $G_\infty(-,y=0)$ is not identically zero.
  \item $\d_y G_\infty$ is invertible inside $U \times K$.
\end{enumerate}
Then there exists a neighbourhood $W$ (independent on $n$) of the origin  such that for any $n$ large enough, the set $V(G_n) \cap W$ is a pathwise connected complex manifold.
\end{proposition}
\begin{proof}
First note that the assumptions and the implicit function theorem imply that, for $n$ large enough, $V(G_n)$ is locally a graph and hence a smooth manifold. Therefore the only non trivial part is the connectedness statement.

Choose $x \in U$, let
 $u:\CM^d \to \CM$ be a linear form and $r$ be such that for $y=0$, the function
$$f_{u,y}:D_r \to \CM,\ t \mapsto u \circ G_\infty(t x,y) $$ 
vanishes only at $t=0$. Here $D_r \subset \CM$ denotes the disk of radius $r$ centred
 at the origin.

The function $f_{u,0}$ is a holomorphic and therefore the multiplicity of a point $p \in D_r$ as a zero of $f_{u,0}$ is given by the integral formula
$$Z_u(y)=\frac{1}{2i\pi} \int_\g \frac{f'_{u,y}(\xi)}{f_{u,y}(\xi)} d\xi $$
where $\g$ is the boundary of the disk $D_r \subset \CM$ and $y=0$. 

As the function $Z_u$ depends continuously on the parameter $y$ and  $K$ is pathwise connected,
we  may choose a neighbourhood $W_u \subset K$ such that for any $(x,y) \in W_u$, the function $Z_u$ remains constant. 

Similarly we define
$$f_{u,y}:D_r \to \CM,\ t \mapsto u \circ G_n(t x,y) $$
and a corresponding function $Z_{u,n}(y)$.
As $(G_n)$ converges to $G_\infty$, for $n$ large enough and $Z_{u,n}$ is  integer valued, we conclude
that it is equal to $Z_u$ for $n$ large enough. By taking a finite number of
independent linear forms, we finally have a neighbourhood $W'$ over which all these integrals remain constant whatever linear form we take. 

We may now repeat the standard proof of local connectedness of irreducible complex analytic varieties using the Weierstrass division theorem to the map $G_n$ restricted to $W'$~(see e.g.~\cite[Chapter 4, Theorem 4.19]{Demailly_livre}). This shows that $V(G_n) \cap W'$ is pathwise connected and concludes the proof of the proposition.
 \end{proof}
 We may apply the proposition to the frequency manifolds $X_n$: 
 \begin{corollary} Let $H=\sum_{i=1}^d \a_i p_iq_i+O(3)$ be an analytic Hamiltonian and let $a \in \Bt^-$
 be such that $a_k<2^{-k}\| \a\|$ for any $k \in \NM$. Then there exists a neighbourhood of the origin $W \subset \CM^{2d}$ such that for $n$ large enough the complex manifolds $W \cap  X_n(a)$ are pathwise connected. If moreover $a_k=O(4^{-k})$ then $W$ can be chosen so that $W \cap X_\infty(a)$ is also pathwise connected. In this common neighbourhood the span of $X_n(a)$ equals the span of $X_n$ at the origin.

\end{corollary}
\section{Curvedness of the HNF}
\label{S::curvedness}
We will now show that the Hamiltonian normal form is automatically curved, if it converges.
We are not able to prove the result directly for the Birkhoff normal form. In fact we are not able to use the Birkhoff normal form at all. The reason is the following. In the context of the iteration for the Hamiltonian normal form we were dealing with functions depending on $d$ frequency variables $\omega_i$ and $d$ Moser variables $\tau_i$. These two sets of variables play a very different role and
give rise to different filtrations on the ring of functions. For the Moser
variables it is the filtration by the powers of the maximal ideal
that is relevant, whereas for the frequency variables we have to consider the filtration defined by valuations on the resonant hyperplanes.

When we represent the frequency manifolds by solving the equations $G_n(\omega,\tau)=0$ and write the $\omega$ variables in terms of the $\tau$ variables, then these filtrations interact in a complicated way. For this reason it is useful
to reformulate curvedness and the arithmetic density theorem directly in terms
of the implicit equations and consider the graphs defined by them in the
product space $\RM^d \times \RM^d$. To do this we first go a little bit further in the study of curvedness for maps.

In the whole section we consider pairs of variables $x_i,y_i$ playing the role of $\tau_i$ and $\omega_i$ together with the filtration which assigns the weight $1$ to $x_i$ and $0$ to $y_i$.
%
\subsection{The R\"u{\ss}mann space}
It is a basic idea in algebraic geometry to look at equations rather than the manifolds defined by these equations. This is what we do now.
So dual to the chain of affine spaces
\[\text{span}_0 (f,p) \subset \text{span}_1(f,p)  \subset \text{span}_2(f,p) \subset \ldots \subset  \text{span}(f,p).\]
of a germ $f:(\RM^m,\a) \to (\RM^d,\b)$ we can consider the chain 
\[ \Rt_0 \supset   \Rt_1 \supset \Rt_2 \supset \ldots \supset \Rt(f,p) \]
consisting of the spaces of affine linear functions vanishing on them:

\begin{definition}
The $k$-th R\"u{\ss}mann space of the germ $f$ is defined as
\[ \Rt_k(f,\a):=\{ \ell \in (\RM^d)^*:  \ell_{|\text{span}_k(f,p)}=0 \} .\]
The R\"u{\ss}mann space is defined as
\[ \Rt(f,\a):=\{\ell \in (\RM^d)^*: \ell_{|\text{span}(f,p)}=0 \} .\]
\end{definition}
\begin{example}
 Let us consider the function ($m=1,d=3$):
 $$f:(\RM,0) \to (\RM^3,0),_ t \mapsto  (t+t^2,t^3,t+t^2+t^3) $$
 We have
 \begin{align*}
  \text{span}_0(f,0)&=\{ 0 \},\\ 
  \text{span}_1(f,0)&=\text{span}_2(f,0)=\text{span}\{ (1,0,1)\},\\
 \text{span}(f,0)& =\text{span}_3(f,0)=\text{span}\{(1,0,1),(0,1,1) \}
 \end{align*}
 The R\"ussmann spaces are ($x,y,z$ are coordinate functions on $\RM^3$):
 \begin{align*}
  \Rt_0(f,0)&=(\RM^3)^*,\\ 
  \Rt_1(f,0)&=\Rt_2(f,0)=\text{span}\{y,x-z\},\\
 \Rt(f,0)& =\Rt_3(f,0)=\text{span}\{x+y-z \}
 \end{align*}
\end{example}

The consideration of the dual space is useful for formulating a
determinantal characterisation.

\begin{definition}
Let $I=(I_1,\dots,I_{s})$ be a sequence of multi-indices in $ \NM_{>0}^m$ and
write $\d^I f$ for the $d \times s$ matrix with $j$-th columns $\d^{I_j}f (p)$.
Let $r+s=d$ and consider linear forms $\ell_1,\ell_2,\ldots,\ell_r$ in the
variables $y_1,\ldots,y_d$. Then we set:
$$\dt^I(f,\ell):=\det (\d^I f,\grad_y(\ell_1),\dots,\grad_y(\ell_r)) .$$
\end{definition}

\begin{proposition}
\label{P::determinant}
Let $\ell_1,\dots,\ell_r \in \Rt(f,\a)$ be elements of the R\"u{\ss}mann space
$\Rt(f,\a)$ and $I=(I_1,I_2,\ldots, I_s)$ a sequence of multi-indices.
If 
$$\dt^I(f,\ell))(\a) \neq 0 $$
then $\ell_1,\dots,\ell_r$ generate $\Rt(f,\a)$ and the torsion index satisfies the estimate
\[t_\a(f) \le \max_j \{|I_j|\} \]
Conversely, if for some elements $\ell_1,\dots,\ell_r \in \Rt(G)$ there exists $I \in \left(\ZM_{>0}^{m}\right)^d$ such that the above determinant is non-zero,
then these generate the R\"u{\ss}mann space.
\end{proposition}
\begin{proof}
 We take a representative of the germ:
 $$f: U \to \RM^d,\ \text{span}(f,\a)=\text{span}(f(U))$$
 The vectors $\d^{I_j}f(\a)$ generate $\text{span}(f,\a)$ and therefore
 the $f$ is curved with torsion index $t_\a(f) \le \max_j \{|I_j|\} $.
\end{proof}
\begin{example}
 With the notations of the above example, we have:
 $$\dt^{1,3}(f,x+y-z)=\begin{vmatrix}
                       1 & 0 & 1 \\ 0 & 6 & 1 \\ 1 &6 &- 1
                      \end{vmatrix}
 =-18 \neq 0$$
\end{example}

The above result has many applications

\begin{corollary}
The torsion index $t_\a(f)$ is upper semi-continuous for curved germs.
\end{corollary}
\begin{proof}
If $\dt^I(f,\ell)(\a) \neq 0$ then, by continuity of the determinant, we may find representative of 
$f:U \to \RM^d$ for which
the determinant is always non-zero.
\end{proof}
 
\subsection{Implicit formulation}
We consider $C^k$-maps
$$G :\RM^m \times \RM^d \supset U \to \RM^d$$
where $U$ is an open subset and their germs. We assume the zero set:
$$V(G):=\{ (x,y) \in U: G(x,y)=0 \} $$
is the graph of a function $f$. The definition of the span  can be adapted to the implicit case:  

\begin{definition} For a map germ
$G:(\RM^m \times \RM^d,\a) \to (\RM^d,\b)$ we call the span of $G$ at $a$ the affine space
\[\text{Span}(G,\a) =\cap_U \text{Span}(V(G)) \subset \RM^d ,\]
where the intersection runs over all representatives of the germ $G$.
\end{definition}

The dual of the span is the {\em R\"u{\ss}mann space of $G$} denoted by $\Rt(G,\a)$.
Note that we may always take a representative $G:U \to \RM^d$ of $f$ such that
\[\text{Span}(G,\a) =\text{Span}(V(G)) \subset \RM^d \]
and for analytic maps we can be more precise:
\begin{proposition}
\label{P::span}
Consider a real analytic map $G: U  \to \RM^d$ and assume that
the analytic variety $V(G)$ is pathwise connected
then for any point $p \in V(G)$, we have
$$\text{Span}(G,p)=\text{span}(V(G)) .$$
\end{proposition}
\begin{proof}
Take $u \in \Rt(G,p)$ then the analytic function $u \circ G$ vanishes on the germ of $V(G)$ at $p$ and, as $V(G)$ is pathwise connected, it also vanishes everywhere on $V(G)$.
\end{proof}
 
The following proposition shows that, for the HNF iteration, the R\"ussmann spaces at the origin of the frequency manifolds stabilise:
\begin{proposition}
\label{P::filtration} Let  $G=(G_1,\dots,G_d)$ and $G'=(G_1',\dots,G'_d)$ be such that for some $N>0$, we have:
$$G_i'=[G_i]^N $$
for a given filtration such that $w(y_i)<N$ then their R\"u{\ss}mann spaces satisfy
 $$ \Rt(G) \subset \Rt(G') $$
\end{proposition}
\begin{proof}
 To say that an affine function $\ell(y)$ is contained in $\Rt( G) $ means that:
 $$  \ell(y)=\sum_{i=1}^k a_i G_i$$
 and truncating this equation at degree $N$ shows that $\ell(y)$ lies also in $\Rt(G')$.
\end{proof}
\subsection*{Differential expression of curvedness}
Our next observation is that the determinant $\dt^I(f,\ell)$ can be expressed in terms of the function $G$. Differentiating
the equation $G(x,f(x))=0$, we get that:
$$\d_x G(x,f(x))+\d_yG(x,f(x))J_f=0 $$
where $J_f=\d_x f$ is the Jacobian matrix of $f$.  
 So assuming $\d_yG$ to be invertible, we get that:
$$J_f=\left( \d_yG\right)^{-1}\d_x G$$
whenever $G(x,y)=0$. This expresses the first derivatives of the function $f$ in terms of $G$.
 If we continue differentiating, we get the:
\begin{proposition} For any multi-index $I$ there exists a polynomial $P$ such that
$$\d_x^I f=\frac{P_I(\d^{J_1}G,\dots,\d^{J_k}G)}{\det(\d_y G)^{2|I|+1}} $$
with where the $J_i$'s run over all multi-indices with $|J| \leq | I |,$ 
when both sides are evaluated along $G(x,y)=0$.
\end{proposition}
\begin{proof}
We use induction on $|I|$.
Assuming
$$\d^I f=\frac{P_I(\d^{J_1}G,\dots,\d^{J_k}G)}{\det(\d_y G)^{2|I|+1}}{\det(\d_y G)^{2|I|+1}} $$
we get that its derivative with respect to $x_k$ is of the form
$$\d_{x_k}\d^I f=\frac{d_{x_k}Q(\d^{J_1}G,\dots,\d^{J_k}G)}{\det(\d_y G)^{2|I|+1}}{\det(\d_y G)^{2|I|+2}}  $$
for some polynomial $Q$. For each monomial appearing in $Q$, we have
$$d_{x_k} (\d^JG)^K=K_k(\d^JG)^{K-e_k}d_{x_k} (\d^J G)$$
where $e_k=(0,\dots,0,1,0,\dots,0)$. Finally $d_{x_k}$ commutes with $\d^J$ and
$$d_{x_k} G=\d_{x_k} G+\d_{y}G \d_{x} f $$
and as $\d_xf=Jf=-(\d_yG)^{-1}\d_x G$, this shows that $\d_{x_k}\d^I f$ is of the desired form and
proves the proposition.
\end{proof}
The proof of the proposition gives a specific solution that we denote by $d^I$: 
$$d^IG(x,y)=\d_x^I f(x)\text{ when } G(x,y)=0.$$
For instance, in the one-dimensional case one has:
\begin{align*}
d^1 G&=\frac{-G_x}{G_y} ,\\
d^2 G&=\frac{-G_{xx}-G_{xy}f'}{G_y}+\frac{G_x(G_{xy}+G_{yy}f')}{G_y^2} ,\\
   &=-\frac{G_{xx}}{G_y}+2\frac{G_{xy}G_x}{G_y^2}-\frac{G_xG_{yy}}{G_y^3} .\\
\end{align*}

Assuming $\d_yG(p) \neq 0$, the {\em torsion index of $G$ at the point $p$} can be defined either as the torsion of the map $f$ solving
$G(x,f(x))=0$ or using the operators $d^I$, as we may take the implicit version of the $l$-span
\[ \text{Span}_l(G,\a):=\text{Span}\{d^I G(\a): I \in \NM^m,\ |I| \leq l\}\]
The determinant
$$\D^I(G,\ell)= \det (d^I G,\grad_y(\ell_1),\dots,\grad_y(\ell_r))$$
restricts to $\dt^I(f,\ell)$ along $V(G)$ therefore we may reformulate Proposition~\ref{P::determinant} in this context:
\begin{proposition}
\label{P::implicit}
Let $\ell_1,\dots,\ell_r \in \Rt(G,\a)$ be elements of the R\"u{\ss}mann space
$\Rt(G,\a)$ and $I=(I_1,I_2,\ldots, I_s)$ a sequence of multi-indices.
Assume that:
$$\Delta^I(G,\ell))(\a) \neq 0,\ \det \d_yG(\a) \neq 0 $$
then $\ell_1,\dots,\ell_r$ generate $\Rt(G)$ and
\[\tau_\a(G) \le \max_j \{|I_j|\} \]
Conversely, if for some elements $\ell_1,\dots,\ell_r \in \Rt(G)$ there exists $I \in \left(\ZM_{>0}^{m}\right)^d$ such that the above determinant is non-zero then these generate the R\"u{\ss}mann space.
\end{proposition}
 


\subsection{Finiteness theorem}
The torsion index is semi-continuous but can we guarantee such a property for a converging sequence? The natural answer would be no.
Indeed the torsion index is given by a non vanishing determinant which is expressed as a partial differential operator $P$. Of course if a sequence $(G_n)$ converges to a limit and $G$ there is a priori no reason that
 $P(G_n) \neq 0$ implies $P(G) \neq 0$. However in our case, the answer will be positive. The reason is that we consider convergence inside the maximal space and not standard convergence so we may consider the statement up to a loss of regularity, more precisely:
\begin{theorem}
Let  $U \to ]0,1]$ be a relative open neighbourhood of the origin inside $ \CM^k$ and $ K $ a pathwise connected compact set. Let 
$(G_n)  \subset \Mt(\Ot^k(U \times K))$
be a sequence converging to a limit\footnote{Reminder: the space $\Ot^k(-)$ is obtained by $C^k$ completion of holomorphic functions and $\Mt(\Ot^k(-))$ denotes the associated maximal space.} 
 $G_\infty$.
Assume that:
\begin{enumerate}
\item $\cap_s U_s=\{ 0 \}$.
 \item $G_\infty(-,y=0)$ is not identically zero.
 \item $\d_y G_\infty$ is invertible at each point.
 \item The R\"ussmann space $\Rt(G_n,0)$ at the origin is independent on $n$.
 \item The sequence $(G_n)$ converges to a limit $G_\infty$ inside the Banach space
 $\Mt(\Ot^k(U \times K))$ for $k=\tau_0(G_n)$.
\end{enumerate}
Then there exists  $s \in [0,1[  $ and a compact neighbourhood of the origin $K' \subset K$ such that
\begin{enumerate}
 \item $\tau_p(G_\infty)=\tau_0(G_n)$ for any $p \in U_s \times K'$.
 \item The manifold $V(G_\infty) \cap (U_s \times K')$ is contained inside $\text{Span}(G_n)$ and more precisely $\text{Span}(G_n)=\text{Span}(G_\infty)$. 
\end{enumerate}
In particular the manifold $V(G_\infty)$ is curved at the origin.
\end{theorem}
 \begin{proof}
\begin{lemma}
Assume that there is a partial
 differential operator $P$ of order $k$ such that $P(G_0)(0) \neq 0$. Then there exists $s>0$ and a compact neighbourhood of the origin $K' \subset K$ such that neither $P(G_n)$ nor $P(G_\infty)$ vanishes inside $U_s \subset K'$ for any $n$.
  \end{lemma}
\begin{proof}
By continuity of the map $P(G_0)$,
 there exists a compact neighbourhood of the origin $K' \subset K$ such that for $s$ small enough, we have
$$m=\inf_{p \in U_s \times K'} |P(G_0)(p)|>0 $$
Put:
$$ R=\max_{n} | G_{n}-G_0|$$
As $P$ is of order $k$, the mean value theorem implies the existence of a constant $M$ such that: 
$$| P(G_n)(z)-P(G_0)(z)|_s \leq  M\sup_{z \in K_s} |G_n(z)-G_0(z)| \leq Ms|G_n-G_0| \leq RMs$$
Note that in the left hand side of the first inequality, we consider a $C^k$ norm while in the right hand side it  is the $C^0$ norm. In the next inequality, the definition the $\Mt^c$-norm gives an $s$ factor.

 We deduce that:
 $$|P(G_n)(z)|\geq m-RM s  \xrightarrow[s \to 0]{} m $$
In particular for any
 $$s< \frac{m}{RM} $$
 we have 
 $$\inf_{z \in K_s} | P(G_n)(z) |>0.$$
This proves the lemma.
\end{proof}
We may now prove the first part of theorem.
Choose a basis $\ell_1,\dots,\ell_k \in \Rt(G_0,0) $ and an index $I$ such that 
$$\D^I(G_0,\ell)(0) \neq 0$$
and $\tau_0(G_0)=|I|$. By Proposition~\ref{P::connected}, we may assume up to further shrinking that $V(G_n)\cap (U_s \times K')$ is pathwise connected. Applying the lemma with $P=\D^I(,-,\ell)$ we get that, we have $\tau_p(G_n) \leq \tau_0(G_n)=\tau_0(G_0)$ for $p \in V(G_n)\cap (U_s \times K') $.

 We now show that:
\[ V(G_{\infty}) \subset \text{Span}(G_0,p)\]
As $V(G_0)$ is pathwise connected, we also have (Proposition~\ref{P::span}):
$$\text{Span}(G_0,p)=\text{Span}{(V(G_0))}$$ 
Take a point $p_0=(x_0,y_0) \notin \text{Span}(G_0,p)$ with $y_0 \in K$.
We have 
$$\text{Span}(G_0,p)=\text{Span}(G_n,p)$$
therefore none of the points $(x,y_0)$ lie in $V(G_n)$ , thus the sequence of holomorphic
maps:
   $$x \mapsto G_n(x,y_0) $$
   is everywhere non-zero. Therefore, by {\em Hurwitz theorem} (see e.g. \cite[Chapter 5, Theorem 2]{Ahlfors}),   its limit is also
non-zero. This shows that $p_0 \notin V(G_\infty)$ and hence we get the inclusion
$V(G_\infty,p) \subset \text{Span}(G_n,p)$. As the span is the smallest affine
space containing a given set, this implies the inclusion
\[ \text{Span}(G_\infty,p) \subset \text{Span}(G_n,p).\]
We have
$$   \D^I(G_\infty,\ell)(p) =\lim_{n \to +\infty}\D^I(G_n,\ell)(p) \geq \inf_{n \in \NM}    \D^I(G_n,\ell)(p)>0,$$
for $|p| \leq \dt$, where $\ell_1,\dots,\ell_k$ generate the R\"u{\ss}mann space of $G_n$ at $p$. 
As $ \D^I(G_\infty,\ell)(p) \neq 0 $ we deduce that $\tau(G_\infty,p) \leq \tau(G_n,p)$ and that $\Rt(G_\infty,p) \subset \Rt(G_n,p)$, which gives the converse
inclusion:
\[  \text{Span}(G_n,p) \subset \text{Span}(G_\infty,p)  .\]
This concludes the proof of the theorem.
\end{proof}
  From this theorem, we immediately deduce that the Hamiltonian normal form is curved at the origin.
\section{Solution to the Hermann conjecture}
\label{SolutionHermanConjecture}
We describe now an application of our theorem to the analysis of 
invariant tori near elliptic critical points of analytic Hamiltonians,
whose frequency satisfies a Bruno condition. For this we have to consider
the appropriate real form of $H$ and restrict to the real domain.

\subsection{Hyperbolic and Elliptic fixed points}
The dynamics of the harmonic oscillator
\[H_e=\frac{1}{2}\sum_{i=1}^d \beta_i(p_i^2+q_i^2)\]
describes quasi-periodic motions with frequency vector $\beta$. All orbits are bounded and the phase space is filled out by a $d$-parameter family
of invariant tori $p_i^2+q_i^2=t_i$, on which the solutions spiral around. 
The geometry of the situation is well-known: the fibres of the map
\[ \RM^{2d} \to \RM_{> 0}^d,\;\;\;(q,p) \mapsto p^2+q^2:=(p_1^2+q_1^2,\ldots,p_d^2+q_d^2)\]
are tori, which are of real dimension $d$ over the strictly positive orthant $\RM_{>0}^d$.  

In the real domain there is a big difference in the dynamical behaviour between
$H_e$ and  its hyperbolic cousin
\[H_h=\sum_{i=1}^d \alpha_i p_i q_i,\]
for which all orbits are unbounded and there exist no invariant tori.

Yet when considered over $\CM$, the canonical coordinate transformation $\phi$
\[ p_j \mapsto \frac{1}{\sqrt{2}}(p_j+i q_j),\;\;\;q_j \mapsto \frac{1}{\sqrt{2}}(q_j+ip_j)\]
maps $H_h$ to $H_e$, when we put
\[\beta =  i \alpha .\]
Another way of expressing the relation between $H_h$ and $H_e$ is by saying 
that the evolution for
$H_h$ in purely imaginary time is equivalent to the real time evolution of 
$H_e$ and vice versa. As a consequence of this relation, we can immediately translate results about $H_h$ into results about $H_e$.
\subsection{Real Hamiltonian normal form}
Consider a real analytic Hamiltonian of the form
\[ H=\frac12 \sum_{i=1}^d \alpha_i (p_i^2+q_i^2) +O(3) \in \RM\{p,q\}.\]
and assume that the frequency vector $\alpha \in \RM^d(a)$, where $a \in \Bt^-$ is summable.

By Theorem~\ref{T::Lagrange}, so for appropriate choices
of the sequence $\rho$ and radius $s_0$, we find sets
\[ W_0 = (Z_0)_{s_0}  \times D_{s_0}^{3d}\]
and 
\[ W_{\infty}=(Z_{\infty})_{s_\infty} \times D_{s_{\infty}}^{3d},\]
such that the iterates of the Hamiltonian normal form
\[A_0,\;\; A_1,\;\;A_2,\ldots\]
converges in the Banach  space $\Ot^c(W_\infty)$ to an element $A_{\infty}$
and  the sequence 
\[ \Phi_0=e^{-v_0},\;\;\Phi_1=e^{-v_1}e^{-v_0},\;\ldots, \Phi_n=\prod_{i=0}^n e^{-v_k},\] 
converges in the operator norm to
$\Phi_{\infty} \in L(\Ot^c(W_0),\Ot^c(W_\infty))$.
This transformation maps, for any $k$, the subspace $\Ot^k(W_0)$ to $\Ot^k(W_\infty)$. In particular,
if the set $W_0$ is chosen inside the holomorphy domain of $A_0$, then $A_0$ is also $C^\infty$ on $W_0$ and
therefore belongs to $\Ot^k(W_0)$ for any $k$. The function $A_\infty$ is then  $C^\infty$ on $W_{\infty}$ and for fixed
$\omega \in Z_{\infty,s_\infty}$ it is holomorphic.

Of course, in a sense we get a 'half-way theorem', as we start with a real 
Hamiltonian, but obtain a statement about its behaviour in the complexified 
domain. But it is clear from the explicit form of the description of the 
iteration that, starting from a real Hamiltonian, the algorithm produces 
{\em real} vector fields $v_n$, which exponentiate to {\em real} analytic 
coordinate transformations $\varphi_n=e^{-v_n}$, etc. As a consequence
the limit transformation $\Phi_{\infty}$ is 'real'. Furthermore, we remark
that it follows from the construction of the vector fields $v_n$ that the
transformation $\varphi_n$ maps the subspace $F_n^c=\Ot^c(V_n)$, ($V_n=W_n \times D_n$) to $F_{n+1}^c$, so that $\Phi_{\infty}$ maps $F_0^c$ to $F_{\infty}^c$
and by the regularity property $F_0^k$ to $F_{\infty}^k$.

So the logic of our argument is the following: if the frequency vector is real $\a \in \RM^d(a)$, we 
consider the sequences 
$$\nu_k=(2^{-(k+1)d^2\tau_0}),\ a'=\nu a.$$
Here $\tau_0$ is the torsion of the formal frequency map (the gradient of the Birkhoff normal form).
In the previous section we show that the frequency map $G_\infty$ 
defines a manifold
$$X_\infty:=\{ (\omega,\tau) \in U(\nu a) \;|\;\;G_{\infty,1}(\omega, \tau)=\dots=G_{\infty,d}(\omega,\tau)=0 \}  $$
contained inside its span at the origin and therefore
has density one at the origin by arithmetic density. We will now see how this positive measure set parametrise invariant tori.

By Whitney extension theorem, we may solve the implicit equations defining $X_\infty$ just like for the formal case. This defines the {\em frequency map} of our Hamiltonian system
\[ \b:\Ut \to \RM^d, \;\;\tau \mapsto \a+f_\infty(\tau).\]
The Taylor series of this map at the origin 
at order $k$ coincides with that of the formal frequency map $b$.
Of course, the process of extension is not unique, but the restriction 
to the preimage of $\RM^d(a)$ is the same for any choice.
The construction can done so that the frequency map is a curved map (just corestrict the map to the span before taking the Whitney extension).

\subsection{The coordinate transformation}
We now translate geometrically our result.   The coordinate functions $\omega_i, \tau_i, q_i,p_i$ can be considered as elements of
the space $E^c_0 \subset \Ot^c(W_0)$ and we write
\[\omega'_i=\Phi_\infty(\omega_i),\;\;\tau'_i=\Phi_\infty(\tau_i),\;\;q'_i=\Phi_\infty(q_i),\;\;p'_i=\Phi_\infty(p_i) .\] 
Note that $\tau'=\tau$ and 
\[\omega'_i \in F_{\infty}^k \subset \Ot^k(V_{\infty}),\;\;\;\textup{for all}\;\; k \in \NM, \] 
so is independent of $q,p$. These functions define a $C^\infty$-map
\[\phi':W_{\infty} \lra W_0,\;\;x \mapsto (\omega'(x),\tau'(x),q'(x),p'(x)),\ x=(\omega,\tau,q,p)\]
and for each $g \in \Ot^c(W_0)$ we have the relation
\[g(\phi'(x))=\Phi_\infty(g)(x).\]
The reality of $\Phi_\infty$ implies that the map $\phi'$ maps the
real part 
$$\Wt_{\infty}:=W_\infty \cap \RM^d\textup{ to } \Wt_{0}:=W_0 \cap \RM^d.$$ 
Thus  we obtain a real $C^\infty$-map 
\[\varphi': \Wt_{\infty} \to \Wt_0.\]
As the map $\p'$ sends the $(\omega,\tau)$-space to itself, it is fibred over the $(\omega,\tau)$-space and we obtain a
commutative diagram:
$$\xymatrix{ \Wt_\infty \ar[r]^{\p'} \ar[d] & \Wt_0 \ar[d]\\
\Vt_\infty \ar[r]^{\psi'} & \Vt_0
}$$
with $\psi'=(\omega',\tau')$ and vertical maps in the diagram forget the 
coordinates $q,p$.

By Whitney extension theorem the component functions $\omega',\tau', q', p'$ of
$\varphi'$ define  $C^\infty$-maps 
\[\psi:\RM^{2d} \to \RM^{2d},\;\;\; \p:\RM^{4d} \to \RM^{4d}.\]
We restrict $\psi$ and $\p$ to the preimages
\[ \Vt_e:=\psi^{-1}(\Vt_0) \supset \Vt_{\infty},\;\;\;\Wt_e:=\p^{-1}(\Wt_0) \supset \Wt_{\infty},\]
and we arrive at a diagram
$$\xymatrix{ \Wt_e \ar[r]^{\p} \ar[d]& \Wt_0 \ar[d] \\
\Vt_e \ar[r]^{\psi} & \Vt_0 }$$
that extends the previous one. Obviously, the maps $\varphi$ and $\psi$
are not unique, but its restrictions to $\Wt_{\infty}$ and $\Vt_{\infty}$ are.
 
As the Taylor series of $\varphi$ is given 
by the series $\Phi_{\infty}$, which is $Id+O(2)$, $\varphi$ is a diffeomorphism near the origin.
Consequently, by restriction to smaller polydiscs, we may and will assume that
\begin{enumerate}
 \item $\psi$ is a diffeomorphism between $\Vt_e $ and $\Vt_0$,
 \item $\p$ is a diffeomorphism between $\Wt_e $ and $\Wt_0$.
\end{enumerate}

As the map $\varphi$ arose from the transformation $\Phi_{\infty}$, it has the
property that after restriction to $\Wt_\infty$, it transforms
\[F_0=H+\sum_{i=1}^d \omega_i (p_i^2+q_i^2)\]
to $F_{\infty}=A_{\infty}$. Furthermore
\[A_{\infty}=A_0+T_{\infty},\;\;\; T_{\infty} \in (R_0+I^2) \cap \Ot^k(W_{\infty}).\]
This means that for $\omega \in Z_{\infty}$ one has
\[F_0 \circ \varphi(\omega,\tau,q,p)=A_{\infty}(\omega,\tau, q,p), \]
and moreover, for such a value, the map $\varphi(\omega,-)$ is an analytic Poisson morphism in the variables $(\tau,q,p)$.

 
\subsection{The elliptic normal form Theorem}
Using the coordinate transformation and the frequency map, we may
sum up our results in the following way:

\begin{theorem}
\label{T::elliptic} 
Let $a=(a_n)$ be a sequence satisfying the Bruno condition and  $\alpha \in \RM(a)$.
Let $H \in \RM\{q,p\}$ be a real analytic function with an elliptic fixed point:
\[ H=\frac{1}{2}\sum_{i=1}^d\a_i (p_i^2+q_i^2)+O(3). \]
Then there exists  an open neighbourhood of the origin $U \subset \CM^d$, $V \subset \CM^{2d}$ with real
parts $\Ut,\Vt$ a curved $C^\infty$-map
\[ \b: \Ut \to  \RM^d,\]
and $C^\infty$-maps:
$$
\xymatrix{\Ut \times \Vt \ar[rr]^-\Psi \ar[dr] &  & \Ut \times \RM^{2d} \ar[dl] \\
 &\Ut  &}
$$
such that for any $\tau \in \b^{-1}(\RM^{2d}(a))$, one has:
\begin{enumerate}[{\rm i)}]
\item The Taylor series expansion of $\b$ at the origin is equal to $\del B(H)$.
\item The map $\Psi$ is a fibred diffeomorphism over its image.
\item The map $\Psi(\tau,-)$ is an analytic symplectomorphism.
\item $H \circ \Psi(\tau,q,p)= \frac{1}{2}\sum_{i=1}^n \b_i(\tau) (p_i^2+q_i^2) +T_{\infty}(\tau,q,p) $
\item $T_{\infty}(\tau,-) \in I^2+\CM$, where $I \subset \Ot^c(V)$ 
is the ideal generated by the $p_i^2+q_i^2-\tau_i$'s.
\end{enumerate}
\end{theorem}
The map $\Psi$ of the theorem is defined in terms of the map $\varphi$ and the 
map $\tau \mapsto \omega(\tau)$ of the previous section by the relation
$$\Psi(\tau,q,p)=\p(\omega(\tau),\tau,q,p) .$$
 
 We note that in the extremal case where $X_\infty=\{(0,\a)\}$, the frequency map
$\b$ is constant, and the condition $\b(\tau) \in \RM^d(a)$ is always
satisfied. In this case the map $\Psi$  is therefore analytic, because $\p$ is analytic in the $\tau$-variables.
So the theorem implies that $H$ is integrable, and thus we recover the classical result of R\"u{\ss}mann stated before~\cite{Russmann}.

 In the general case,  our iteration produces a $C^\infty$ function $\b$, whose Taylor 
expansion at the origin is the formal frequency map given by the Birkhoff normal form. In a similar way, our construction shows that 
the sequence $(h_n)$ of \ref{SS::relation} converges to a limit $h_\infty$.
This limit function being the constant term in the expression
$$T_\infty(\tau,q,p)= h_\infty(\tau)+\sum_{i,j}t_{ij}(q,p)f_if_j$$ 
with $\tau \in \b^{-1}(\RM^d(a))$.
 
\subsection{A big set of invariant tori}
A direct corollary of the elliptic normal form theorem is the following.
\begin{corollary}
For $\tau \in \b^{-1}(\RM^d(a))$, the image under $\Psi(\tau,-)$ of the torus 
$$T_{\tau}: p_1^2+q_1^2=\tau_1,\dots,p_1^2+q_1^2=\tau_n$$ 
is invariant under the Hamiltonian flow of $H$. 
The motion on this torus is quasi-periodic with frequency $\b(\tau)$.  
\end{corollary}

So we get a collection of invariant tori in our Hamiltonian system, 
parametrised by the {\em inverse image} of $\RM^d(a)$ by the frequency map 
\[ \b: \Ut \to  \RM^d(a) \]
which is curved at the origin.

We will now see how these tori fit together in a neighbourhood of the
origin of our original Hamiltonian $H(p,q)$. The map 
$$\Psi: \Ut \times \Vt \to \Ut \times \RM^{2d}$$
of the previous theorem has an inverse $\Gamma=\Psi^{-1}$ over a sufficiently small neighbourhood 
of the origin of the form $\Ut \times \Bt$, $\Bt \subset \RM^{2d}$:
\[\Gamma: \Ut \times \Bt \to \Ut \times \Vt; (\tau,p,q) \mapsto (\tau, P(\tau,p,q),Q(\tau,p,q))\]
So we have the relation
\[H(p,q) =A_0(\omega(\tau),P(\tau,p,q),Q(\tau,p,q)) +T_{\infty}\circ \Gamma(\tau,p,q).\]

We can, in principle, eliminate the variables $\tau_1,\tau_2,\ldots,\tau_d$ 
from the right hand side by solving the implicit equations
\[ P_i(\tau,p,q)^2+Q_i(\tau,p,q)^2=\tau_i,\;\;i=1,2,\ldots,d,\]
which produces a map
\[ T: \Bt \lra \Ut;\;\; (p,q) \mapsto (\tau_1(p,q),\ldots,\tau_d(p,q)).\]
As one has 
\[\tau_i(p,q)=p_i^2+q_i^2+O(3),\]
the map $T$ is generically a submersion. In fact, it is a submersion on $\Bt \setminus C$, $C:=T^{-1}(\Delta)$, where $\Delta \subset \Ut$
is the set of {\em critical values} of $T$.

One now obtains a diagram 
\[
\xymatrix{
\Bt \ar[d]_-T\ar[r]^-\g &  \RM^d \times \Vt\ar[d]^{\pi}\\
\Ut \ar[r]^{\beta} &\RM^d \\
S\ar@{^{(}->}[u] \ar[r]&    \RM^d(a) \ar@{^{(}->}[u]\\
}
\]
related to our normal forms as follows.

On the right hand side we have the standard Hamiltonian
\[A_0=\frac{1}{2}\sum_{i=1}^d (\a_i+\omega_i)(p_i^2+q_i^2),\]
defined on $ \RM^d \times \Vt$, where the map 
$$\pi:(\omega,q,p) \mapsto \a+\omega$$
gives the frequency of motion.

On the left hand side we have a neighbourhood $\Bt$, on which the
original Hamiltonian $H(p,q)$ is defined. The vertical map on the left 
is the $\tau$-map $T(p,q)$ defined above.

 The horizontal map $\gamma$
stems from coordinate transformation $\Gamma$:
\begin{align*}
 \gamma:\Bt &\to  \RM^d \times \Vt\\
  (q,p)& \mapsto (\beta(T(p,q)),P(T(p,q),p,q),Q(T(p,q),p,q)) 
\end{align*}

The horizontal map in the middle is the frequency map $\b$, which is curved. The inverse image 
$S:=\b^{-1}(\RM^d(a))$ under $\b$ parametrises invariant tori for
$H$ in the neighbourhood $\Bt$. As $T$ is a submersion outside $\Delta$, 
which by Sard's theorem has measure zero, the set
$T^{-1}(S \setminus \Delta) \subset \Bt$ yields a set of positive measure 
consisting of invariant tori in the neighbourhood $\Bt$ of the elliptic critical 
point, as conjectured by Herman~\cite{Herman_ICM}.
 
     \bibliographystyle{amsplain}
\bibliography{master}
 \end{document}